\documentclass[american,english]{article}
\usepackage[latin9]{inputenc}
\usepackage{array}
\usepackage{textcomp}
\usepackage{multirow}
\usepackage{amsmath}
\usepackage{amsthm}
\usepackage{amssymb}
\usepackage{wasysym}
\PassOptionsToPackage{normalem}{ulem}
\usepackage{ulem}

\makeatletter

\providecommand{\tabularnewline}{\\}

\numberwithin{equation}{section}
\numberwithin{figure}{section}
\theoremstyle{plain}
\newtheorem{thm}{\protect\theoremname}[section]
  \theoremstyle{definition}
  \newtheorem{problem}[thm]{\protect\problemname}
  \theoremstyle{remark}
  \newtheorem{rem}[thm]{\protect\remarkname}
  \theoremstyle{definition}
  \newtheorem{defn}[thm]{\protect\definitionname}
  \theoremstyle{plain}
  \newtheorem{cor}[thm]{\protect\corollaryname}
  \theoremstyle{plain}
  \newtheorem{prop}[thm]{\protect\propositionname}
  \theoremstyle{plain}
  \newtheorem{lem}[thm]{\protect\lemmaname}

\makeatother

\usepackage{babel}
  \addto\captionsamerican{\renewcommand{\corollaryname}{Corollary}}
  \addto\captionsamerican{\renewcommand{\definitionname}{Definition}}
  \addto\captionsamerican{\renewcommand{\lemmaname}{Lemma}}
  \addto\captionsamerican{\renewcommand{\problemname}{Problem}}
  \addto\captionsamerican{\renewcommand{\propositionname}{Proposition}}
  \addto\captionsamerican{\renewcommand{\remarkname}{Remark}}
  \addto\captionsamerican{\renewcommand{\theoremname}{Theorem}}
  \addto\captionsenglish{\renewcommand{\corollaryname}{Corollary}}
  \addto\captionsenglish{\renewcommand{\definitionname}{Definition}}
  \addto\captionsenglish{\renewcommand{\lemmaname}{Lemma}}
  \addto\captionsenglish{\renewcommand{\problemname}{Problem}}
  \addto\captionsenglish{\renewcommand{\propositionname}{Proposition}}
  \addto\captionsenglish{\renewcommand{\remarkname}{Remark}}
  \addto\captionsenglish{\renewcommand{\theoremname}{Theorem}}
  \providecommand{\corollaryname}{Corollary}
  \providecommand{\definitionname}{Definition}
  \providecommand{\lemmaname}{Lemma}
  \providecommand{\problemname}{Problem}
  \providecommand{\propositionname}{Proposition}
  \providecommand{\remarkname}{Remark}
\providecommand{\theoremname}{Theorem}

\begin{document}

\title{The IA-congruence kernel of high rank free Metabelian groups}

\author{David El-Chai Ben-Ezra}
\maketitle
\begin{abstract}
The congruence subgroup problem for a finitely generated group $\Gamma$
and $G\leq Aut(\Gamma)$ asks whether the map $\hat{G}\to Aut(\hat{\Gamma})$
is injective, or more generally, what is its kernel $C\left(G,\Gamma\right)$?
Here $\hat{X}$ denotes the profinite completion of $X$. In this
paper we investigate $C\left(IA(\Phi_{n}),\Phi_{n}\right)$, where
$\Phi_{n}$ is a free metabelian group on $n\geq4$ generators, and
$IA(\Phi_{n})=\ker(Aut(\Phi_{n})\to GL_{n}(\mathbb{Z}))$. 

We show that in this case $C(IA(\Phi_{n}),\Phi_{n})$ is abelian,
but not trivial, and not even finitely generated. This behavior is
very different from what happens for free metabelian group on $n=2,3$
generators, or for finitely generated nilpotent groups.
\end{abstract}
\textbf{}

\tableofcontents{}

\section{Introduction}

The classical congruence subgroup problem (CSP) asks for, say, $G=SL_{n}\left(\mathbb{Z}\right)$
or $G=GL_{n}\left(\mathbb{Z}\right)$, whether every finite index
subgroup of $G$ contains a principal congruence subgroup, i.e. a
subgroup of the form $G\left(m\right)=\ker\left(G\to GL_{n}\left(\mathbb{Z}/m\mathbb{Z}\right)\right)$
for some $0\neq m\in\mathbb{Z}$. It is a classical $19^{\underline{th}}$
century result that the answer is negative for $n=2$. On the other
hand, quite surprisingly, it was proved in the sixties by Mennicke
\cite{key-22} and by Bass-Lazard-Serre \cite{key-23} that for $n\geq3$
the answer to the CSP is affirmative. A rich theory of the CSP for
more general arithmetic groups has been developed since then.

By the observation $GL_{n}\left(\mathbb{Z}\right)\cong Aut\left(\mathbb{Z}^{n}\right)$,
the CSP can be generalized to automorphism groups as follows: Let
$\Gamma$ be a group and $G\leq Aut\left(\Gamma\right)$. For a finite
index characteristic subgroup $M\leq\Gamma$ denote 
\[
G\left(M\right)=\ker\left(G\to Aut\left(\Gamma/M\right)\right).
\]
A finite index subgroup of $G$ which contains $G\left(M\right)$
for some $M$ is called a ``congruence subgroup''. The CSP for the
pair $\left(G,\Gamma\right)$ asks whether every finite index subgroup
of $G$ is a congruence subgroup. 

One can easily see that the CSP is equivalent to the question: Is
the congruence map $\hat{G}=\underleftarrow{\lim}G/U\to\underleftarrow{\lim}G/G\left(M\right)$
injective? Here, $U$ ranges over all finite index normal subgroups
of $G$, and $M$ ranges over all finite index characteristic subgroups
of $\Gamma$. When $\Gamma$ is finitely generated, it has only finitely
many subgroups of given index $m$, and thus, the charateristic subgroups:
$M_{m}=\cap\left\{ \Delta\vartriangleleft\Gamma\,|\,\left[\Gamma:\Delta\right]|m\right\} $
are of finite index in $\Gamma$. Hence, one can write $\hat{\Gamma}=\underleftarrow{\lim}_{m\in\mathbb{N}}\Gamma/M_{m}$
and have\footnote{By the celebrated theorem of Nikolov and Segal which asserts that
every finite index subgroup of a finitely generated profinite group
is open \cite{key-17-1}, the second inequality is actually an equality.
However, we do not need it. } 
\begin{eqnarray*}
\underleftarrow{\lim}G/G\left(M\right) & = & \underleftarrow{\lim}_{m\in\mathbb{N}}G/G\left(M_{m}\right)\leq\underleftarrow{\lim}_{m\in\mathbb{N}}Aut(\Gamma/M_{m})\\
 & \leq & Aut(\underleftarrow{\lim}_{m\in\mathbb{N}}(\Gamma/M_{m}))=Aut(\hat{\Gamma}).
\end{eqnarray*}
Therefore, when $\Gamma$ is finitely generated, the CSP is equivalent
to the question: Is the congruence map: $\hat{G}\to Aut(\hat{\Gamma})$
injective? More generally, the CSP asks what is the kernel $C\left(G,\Gamma\right)$
of this map. For $G=Aut\left(\Gamma\right)$ we will also use the
simpler notation $C\left(\Gamma\right)=C\left(G,\Gamma\right)$. The
classical congruence subgroup result mentioned above can therefore
be reformulated as $C(\mathbb{Z}^{n})=\left\{ e\right\} $ for $n\geq3$,
and it is also known that $C(\mathbb{Z}^{2})=\hat{F}_{\omega}$, where
$\hat{F}_{\omega}$ is the free non-abelian profinite group on a countable
number of generators (cf. \cite{key-17}, \cite{key-4}).

Very few results are known when $\Gamma$ is non-abelian. Most of
the results are related to $\Gamma=\pi(S_{g,n})$, the fundamental
group of the closed surface of genus $g$ with $n$ punctures (see
\cite{key-16-1}, \cite{key-18}, \cite{key-3}, \cite{key-19}, \cite{key-6-1}).
As observed in \cite{key-5}, the result of Asada in \cite{key-3}
actually gives an affirmative solution to the case $\Gamma=F_{2}$,
$G=Aut(F_{2})$ (see also \cite{key-7}). Note that for every $n>0$,
one has $\pi(S_{g,n})\cong F_{2g+n-1}$ = the free group on $2g+n-1$
generators. Hence, the aforementioned results relate to various subgroups
of the automorphism group of finitely generated free groups. However,
the CSP for the full $Aut(F_{n})$ when $n\geq3$ is still unsettled. 

Denote now the free metabelian group on $n$ generators by $\Phi_{n}=F_{n}/F_{n}''$.
Considering the metabelian case, it was shown in \cite{key-7} (see
also \cite{key-6}) that $C\left(\Phi_{2}\right)=\hat{F}_{\omega}$.
In addition, it was proven there that $C\left(\Phi_{3}\right)\supseteq\hat{F}_{\omega}$.
The basic motivation which led to this paper was to complete the picture
in the free metabelian case and investigate $C\left(\Phi_{n}\right)$
for $n\geq4$. Now, denote $IA(\Phi_{n})=\ker(Aut(\Phi_{n})\to GL_{n}(\mathbb{Z}))$.
Then, the commutative exact diagram
\[
\begin{array}{ccccccccc}
1 & \to & IA\left(\Phi_{n}\right) & \to & Aut\left(\Phi_{n}\right) & \to & GL_{n}\left(\mathbb{Z}\right) & \to & 1\\
 &  &  & \searrow & \downarrow &  & \downarrow\\
 &  &  &  & Aut(\hat{\Phi}_{n}) & \to & GL_{n}(\hat{\mathbb{Z}})
\end{array}
\]
gives rise to the commutative exact diagram (see Lemma 2.1 in \cite{key-5})
\[
\begin{array}{ccccccc}
\widehat{IA\left(\Phi_{n}\right)} & \to & \widehat{Aut\left(\Phi_{n}\right)} & \to & \widehat{GL_{n}\left(\mathbb{Z}\right)} & \to & 1\\
 & \searrow & \downarrow &  & \downarrow\\
 &  & Aut(\hat{\Phi}_{n}) & \to & GL_{n}(\hat{\mathbb{Z}}) & .
\end{array}
\]
Hence, by using the fact that $\widehat{GL_{n}\left(\mathbb{Z}\right)}\to GL_{n}(\hat{\mathbb{Z}})$
is injective for $n\geq3$, one can obtain that $C\left(\Phi_{n}\right)$
is an image of $C\left(IA\left(\Phi_{n}\right),\Phi_{n}\right)$.
Thus, for investigating $C\left(\Phi_{n}\right)$ it seems to be worthwhile
to investigate $C\left(IA\left(\Phi_{n}\right),\Phi_{n}\right)$. 

The first goal of the present paper is to prove the following theorem:
\begin{thm}
\label{thm:main}For every $n\geq4$, the group $C\left(IA\left(\Phi_{n}\right),\Phi_{n}\right)$
contains a subgroup $C$ which satisfies the following properties: 
\begin{itemize}
\item $C$ is isomorphic to a product $C=\prod_{i=1}^{n}C_{i}$ of $n$
copies of
\[
C_{i}\cong\ker(\widehat{SL_{n-1}\left(\mathbb{Z}[x^{\pm1}]\right)}\to SL_{n-1}(\widehat{\mathbb{Z}[x^{\pm1}]})).
\]
\item $C$ is a direct factor of $C\left(IA\left(\Phi_{n}\right),\Phi_{n}\right)$,
i.e. there is a normal subgroup $N\vartriangleleft C\left(IA\left(\Phi_{n}\right),\Phi_{n}\right)$
such that $C\left(IA\left(\Phi_{n}\right),\Phi_{n}\right)=N\times C$.
\end{itemize}
\end{thm}

Using techniques of Kassabov and Nikolov in \cite{key-24-1} one can
show that the subgroups $C_{i}$ are not finitely generated. So as
an immediate corollary, we obtain the following theorem:
\begin{thm}
\label{cor:not finitely}For every $n\geq4$, the group $C\left(IA\left(\Phi_{n}\right),\Phi_{n}\right)$
is not finitely generated.
\end{thm}

It will be shown in an upcoming paper that when $\Gamma$ is a finitely
generated nilpotent group (of any class), then $C\left(IA\left(\Gamma\right),\Gamma\right)=\left\{ e\right\} $
is always trivial. So the free metabelian cases behave completely
different from nilpotent cases. This result gives the impression that
$C\left(IA\left(\Phi_{n}\right),\Phi_{n}\right)$ is ``big''. On
the other hand, we have the following theorem (see \cite{key-6-2}):
\begin{thm}
\label{thm:central}For every $n\geq4$, the group $C\left(IA\left(\Phi_{n}\right),\Phi_{n}\right)$
is central in $\widehat{IA\left(\Phi_{n}\right)}$.
\end{thm}

We remark that in the case of arithmetic groups, the congruence kernel
is known to have a dichotomous behavior: it is central if and only
if it is finite (see \cite{key-50}, Theorem 2). So in some sense,
the congruence kernel $C\left(IA\left(\Phi_{n}\right),\Phi_{n}\right)$
for $n\geq4$ has an intermediate behavior: central, but not finite.
The latter is similar to the behavior of the congruence kernel
\[
\ker(\widehat{SL_{d}(\mathbb{Z}[x])}\to SL_{d}(\widehat{\mathbb{Z}[x]}))\,\,\,\,\textrm{for}\,\,\,\,d\geq3
\]
that was investigated in \cite{key-24-1} (see Theorem 4.1). 

Theorem \ref{thm:central} has already been stated in \cite{key-6-2}.
However, a substantial portion of the proof of Theorem \ref{thm:central}
appears in this paper - this is the second goal of this present paper.
To be more precise, all the steps of the proof of Theorem \ref{thm:central}
that involve arguments in Algebraic K-theory are given in this paper,
and in \cite{key-6-2} we describe the structure of the proof, and
present all the other steps. As will be presented in $\varoint$\ref{sec:Finishing},
the steps that are given in this present paper by themselves, will
be sufficient for showing that the subgroup $C\leq C(IA(\Phi_{n}),\Phi_{n})$
that presented in Theorem \ref{thm:main} is contained in the center
of $\widehat{IA(\Phi_{n})}$. We remark that the main results in this
paper that are used in \cite{key-6-2} in order to prove Theorem \ref{thm:central}
are Lemma \ref{lem:main} and our work in $\varoint$\ref{sec:Finishing}
(see Remark \ref{rem:IA_m-IA^m} for a more precise description).
The following problem is still open:
\begin{problem}
\label{prob:Is2-1}Is $C\left(IA\left(\Phi_{n}\right),\Phi_{n}\right)=\prod_{i=1}^{n}C_{i}$
or does it contain more elements?
\end{problem}

\begin{rem}
\label{rem:action}Considering the action of $Aut\left(\Phi_{n}\right)$
on $IA\left(\Phi_{n}\right)$ by conjugation, we have a natural map
$Aut\left(\Phi_{n}\right)\to Aut(IA\left(\Phi_{n}\right))$ in which
the copy of $IA\left(\Phi_{n}\right)$ in $Aut\left(\Phi_{n}\right)$
is mapped onto $IA\left(\Phi_{n}\right)\to Inn(IA\left(\Phi_{n}\right))$.
Denote now $IA_{n,m}=\cap\left\{ N\vartriangleleft IA\left(\Phi_{n}\right)\,|\,[IA\left(\Phi_{n}\right):N]\,|\,m\right\} $.
Then as for every $n\geq4$, the group $IA\left(\Phi_{n}\right)$
is finitely generated \cite{key-24}, the characteristic subgroups
$IA_{n,m}\leq IA\left(\Phi_{n}\right)$ are of finite index. Hence
$\widehat{IA\left(\Phi_{n}\right)}=\underleftarrow{\lim}_{m\in\mathbb{N}}(IA\left(\Phi_{n}\right)/IA_{n,m})$
and therefore the action of $Aut\left(\Phi_{n}\right)$ on $IA\left(\Phi_{n}\right)$
induces an action of $Aut\left(\Phi_{n}\right)$ on $\widehat{IA\left(\Phi_{n}\right)}$
so we have a map $Aut\left(\Phi_{n}\right)\to\underleftarrow{\lim}_{m\in\mathbb{N}}Aut(IA\left(\Phi_{n}\right)/IA_{n,m}))\leq Aut(\widehat{IA\left(\Phi_{n}\right)})$.
The latter gives rise to a map
\[
\widehat{Aut\left(\Phi_{n}\right)}\to\underleftarrow{\lim}_{m\in\mathbb{N}}Aut(IA\left(\Phi_{n}\right)/IA_{n,m}))\leq Aut(\widehat{IA\left(\Phi_{n}\right)})
\]
that actually gives an action of $\widehat{Aut\left(\Phi_{n}\right)}$
on $\widehat{IA\left(\Phi_{n}\right)}$, such that the closure $\overline{IA\left(\Phi_{n}\right)}$
of $IA\left(\Phi_{n}\right)$ in $\widehat{Aut\left(\Phi_{n}\right)}$
acts trivially on $Z(\widehat{IA\left(\Phi_{n}\right)})$, the center
of $\widehat{IA\left(\Phi_{n}\right)}$. Thus, as we have $\widehat{Aut\left(\Phi_{n}\right)}/\overline{IA\left(\Phi_{n}\right)}=\widehat{GL_{n}\left(\mathbb{Z}\right)}$
we obtain a natural action of $\widehat{GL_{n}\left(\mathbb{Z}\right)}$
on $Z(\widehat{IA\left(\Phi_{n}\right)})$. It will be clear from
the description in the paper that the permutation matrices permute
the copies $C_{i}$ through this natural action. 
\end{rem}

The aforementioned behavior of $C\left(IA\left(\Phi_{n}\right),\Phi_{n}\right)$
for $n\geq4$ is also different from the behavior of $C\left(IA\left(\Phi_{n}\right),\Phi_{n}\right)$
for $n=2,3$. More precisly, as $C(\mathbb{Z}^{3})=\{e\}$, similar
arguments show that when $n=3$ the group $C\left(\Phi_{3}\right)$
is an image of $C\left(IA\left(\Phi_{3}\right),\Phi_{3}\right)$.
So as $C\left(\Phi_{3}\right)\supseteq\hat{F}_{\omega}$ \cite{key-7},
we obtain that $C\left(IA\left(\Phi_{3}\right),\Phi_{3}\right)$ is
infinite non-abelian. On the other hand, regarding the case $n=2$,
it is known that $IA\left(\Phi_{2}\right)=Inn(\Phi_{2})$ (see \cite{key-13})
and it is known that the center of $\Phi_{2}$ and $\hat{\Phi}_{2}$
is trivial (see \cite{key-6}). It follows that we have a canonical
isomorphism 
\[
\widehat{IA\left(\Phi_{2}\right)}=\widehat{Inn(\Phi_{2})}\cong\hat{\Phi}_{2}\cong Inn(\hat{\Phi}_{2})\leq Aut(\hat{\Phi}_{2})
\]
so $C\left(IA\left(\Phi_{2}\right),\Phi_{2}\right)=\left\{ e\right\} $
is trivial. Our results show that when $n\geq4$, the behavior of
$C\left(IA\left(\Phi_{n}\right),\Phi_{n}\right)$ stabilizes and it
is abelian, but not trivial.

We also note that considering our basic motivation, as $C\left(\Phi_{n}\right)$
is an image of $C\left(IA\left(\Phi_{n}\right),\Phi_{n}\right)$ we
actually obtain from Theorem \ref{thm:central} that when $n\geq4$,
the situation is dramatically different from the cases of $n=2,3$
described above, and:
\begin{thm}
\label{thm:full}For every $n\geq4$, the group $C\left(\Phi_{n}\right)$
is abelian.
\end{thm}

We remark that despite the result of the latter theorem, we do not
know whether $C\left(\Phi_{n}\right)$ is also not finitely generated.
In fact we cannot even prove at this point that it is not trivial. 

The paper is organized as follows. For a ring $R$, ideal $H\vartriangleleft R$
and $d\in\mathbb{N}$ denote 
\[
GL_{d}(R,H)=\ker(GL_{d}(R)\to GL_{d}(R/H)).
\]
For $n\in\mathbb{N}$ denote also the ring $R_{n}=\mathbb{Z}[x_{1}^{\pm1},\ldots,x_{n}^{\pm1}]=\mathbb{Z}[\mathbb{Z}^{n}]$.
Using the Magnus embedding of $IA\left(\Phi_{n}\right)$, in which
$IA\left(\Phi_{n}\right)$ can be viewed as
\[
IA\left(\Phi_{n}\right)=\left\{ A\in GL_{n}(R_{n})\,|\,A\left(\begin{array}{c}
x_{1}-1\\
\vdots\\
x_{n}-1
\end{array}\right)=\left(\begin{array}{c}
x_{1}-1\\
\vdots\\
x_{n}-1
\end{array}\right)\right\} 
\]
we obtain in $\varoint$\ref{sec:some-properties}, for every $1\leq i\leq n$,
a natural embedding 
\[
GL_{n-1}(R_{n},(x_{i}-1)R_{n})\hookrightarrow IA\left(\Phi_{n}\right)
\]
and a surjective natural homomorphism
\[
IA\left(\Phi_{n}\right)\overset{\rho_{i}}{\twoheadrightarrow}GL_{n-1}(\mathbb{Z}[x_{i}^{\pm1}],(x_{i}-1)\mathbb{Z}[x_{i}^{\pm1}])
\]
in which the obvious copy of the subgroup $GL_{n-1}(\mathbb{Z}[x_{i}^{\pm1}],(x_{i}-1)\mathbb{Z}[x_{i}^{\pm1}])$
in $GL_{n-1}(R_{n},(x_{i}-1)R_{n})$ is mapped onto itself via the
composition map (Proposition \ref{prop:semi-direct}). This description,
combined with some classical notions and results from Algebraic K-theory
presented in $\varoint$\ref{sec:K-theory}, enables us in $\varoint$\ref{sec: not f.g.}
to show that for every $n\geq4$ and $1\leq i\leq n$, the group $C\left(IA\left(\Phi_{n}\right),\Phi_{n}\right)$
contains a copy of
\begin{eqnarray}
C_{i} & \cong & \ker(\widehat{GL_{n-1}(\mathbb{Z}[x_{i}^{\pm1}],(x_{i}-1)\mathbb{Z}[x_{i}^{\pm1}])}\to GL_{n-1}(\widehat{\mathbb{Z}[x_{i}^{\pm1}]}))\nonumber \\
 & \cong & \ker(\widehat{SL_{n-1}(\mathbb{Z}[x_{i}^{\pm1}])}\to SL_{n-1}(\widehat{\mathbb{Z}[x_{i}^{\pm1}]}))\label{eq: iso}
\end{eqnarray}
such that $C\left(IA\left(\Phi_{n}\right),\Phi_{n}\right)$ is maped
onto $C_{i}$ through the map $\hat{\rho}_{i}$ which is induced by
$\rho_{i}$. The second isomorphism in Equation (\ref{eq: iso}) is
obtained by using a main lemma, Lemma \ref{lem:main}, combined with
some classical results from Algebraic K-theory (Propositions \ref{prop:injective}
and \ref{prop:S-G}). The proof of Lemma \ref{lem:main} will be postponed
until the end of the paper. In particular, we get that for every $1\leq i\leq n$
one has
\[
C\left(IA\left(\Phi_{n}\right),\Phi_{n}\right)=(C\left(IA\left(\Phi_{n}\right),\Phi_{n}\right)\cap\ker\hat{\rho}_{i})\rtimes C_{i}.
\]
(see Proposition \ref{cor:C semi-direct}). In $\varoint$\ref{sec: not f.g.}
we aslo show that the copies $C_{i}$ lie in $\ker\hat{\rho}_{j}$
whenever $j\neq i$ (Proposition \ref{prop:ker-contain}). In particular
we get that the copies $C_{i}$ intersect each other trivially. Then,
following the techniques of Kassabov and Nikolov in \cite{key-24-1}
we show that $C_{i}$ is not finitely generated, and therefore deduce
that $C\left(IA\left(\Phi_{n}\right),\Phi_{n}\right)$ is not finitely
generated either, i.e. we prove Theorem \ref{cor:not finitely} (see
the end of $\varoint$\ref{sec: not f.g.}). Then, in $\varoint$\ref{sec:Finishing}
we show that the copies $C_{i}$ lie in the center of $\widehat{IA\left(\Phi_{n}\right)}$,
using classical results from Algebraic K-theory and the main lemma,
Lemma \ref{lem:main}. In particular, using the aforementioned results,
we obtain that
\[
C\left(IA\left(\Phi_{n}\right),\Phi_{n}\right)=(C\left(IA\left(\Phi_{n}\right),\Phi_{n}\right)\cap_{i=1}^{n}\ker\hat{\rho}_{i})\times\prod_{i=1}^{n}C_{i}.
\]
This completes the proof of Theorem \ref{thm:main}. 

After that we turn to prove the main lemma, Lemma \ref{lem:main}.
In $\varoint$\ref{sec:elementary} we introduce some elements in
$\left\langle IA\left(\Phi_{n}\right)^{m}\right\rangle $ which are
needed for the proof of Lemma \ref{lem:main}. In $\varoint$\ref{sec:Second},
using classical results from algebraic K-theory, we end the paper
by proving Lemma \ref{lem:main} which asserts that for every $1\leq i\leq n$,
we have
\begin{equation}
GL_{n-1}(R_{n},(x_{i}-1)R_{n})\cap E{}_{n-1}\left(R_{n},H_{n,m^{2}}\right)\subseteq\left\langle IA\left(\Phi_{n}\right)^{m}\right\rangle \label{eq:intersect}
\end{equation}
where:
\begin{itemize}
\item $GL_{n-1}(R_{n},(x_{i}-1)R_{n})$ denotes its appropriate copy in
$IA\left(\Phi_{n}\right)$ described above.
\item $E_{n-1}\left(R_{n},H_{n,m^{2}}\right)$ is the subgroup of $E_{n-1}(R_{n})=\left\langle I_{n-1}+rE_{i,j}\,|\,r\in R_{n}\right\rangle $
which is generated as a normal subgroup by the elementary matrices
of the form $I_{n-1}+hE_{i,j}$ for $h\in H_{n,m^{2}}=\ker\left(R_{n}\to\mathbb{Z}_{m^{2}}[\mathbb{Z}_{m^{2}}^{n}]\right)$,
$1\leq i\neq j\leq n$. Here, $I_{n-1}$ is the $(n-1)\times(n-1)$
unit matrix and $E_{i,j}$ is the matrix which has $1$ in the $\left(i,j\right)$-th
entry and $0$ elsewhere.
\item The intersection in Inclusion (\ref{eq:intersect}) is obtained by
viewing the copy of $GL_{n-1}(R_{n},(x_{i}-1)R_{n})$ in $IA(\Phi_{n})$
as a subgroup of $GL_{n-1}(R_{n})$.
\end{itemize}
We note that as described above, Lemma \ref{lem:main} is used in
two places along the paper. It is used once to prove the second isomorphism
in Equation (\ref{eq: iso}). The second place is in the proof that
the group $C$ lies in the center of $\widehat{IA\left(\Phi_{n}\right)}$.
We also note that almost all the work that we do in order to show
that $C$ lies in the center of $\widehat{IA\left(\Phi_{n}\right)}$,
including Lemma \ref{lem:main} (but also most of $\varoint$\ref{sec:Finishing}),
is used in \cite{key-6-2} to prove Theorem \ref{thm:central} (see
Remark \ref{rem:IA_m-IA^m}). 

\textbf{Acknowledgements:} I wish to offer my thanks to my supervisor
during the research, Prof. Alexander Lubotzky, for his sensitive and
devoted guidance. During the period of the research, I was supported
by the Rudin foundation and, not concurrently, by NSF research training
grant (RTG). 

\section{\label{sec:K-theory}Some background in algebraic K-theory}

In this section we fix some notations and recall some definitions
and background in algebraic K-theory which will be used throughout
the paper. One can find more general information in the references
(\cite{key-31-2}, \cite{key-32}, \cite{key-31-1}). In this section
$R$ will always denote a commutative ring with identity. We start
with recalling the following notations. Let $R$ be a commutative
ring, $H\vartriangleleft R$ an ideal, and $d\in\mathbb{N}$. Then:
\begin{itemize}
\item $GL_{d}\left(R\right)=\left\{ A\in M_{n}\left(R\right)\,|\,\det\left(A\right)\in R^{*}\right\} $.
\item $SL_{d}\left(R\right)=\left\{ A\in GL_{d}\left(R\right)\,|\,\det\left(A\right)=1\right\} $.
\item $E_{d}\left(R\right)=\left\langle I_{d}+rE_{i,j}\,|\,r\in R,\,1\leq i\neq j\leq d\right\rangle $.
\item $GL_{d}\left(R,H\right)=\ker\left(GL_{d}\left(R\right)\to GL_{d}\left(R/H\right)\right)$.
\item $SL_{d}\left(R,H\right)=\ker\left(SL_{d}\left(R\right)\to SL_{d}\left(R/H\right)\right)$.
\item $E_{d}\left(R,H\right)$ = the normal subgroup of $E_{d}\left(R\right)$,
which is generated as a normal subgroup by the elementary matrices
of the form $I_{d}+hE_{i,j}$ for $h\in H$.
\end{itemize}
For every $d\geq3$, the subgroup $E_{d}\left(R,H\right)$ is normal
in $GL_{d}\left(R\right)$ (see Corollary 1.4 in \cite{key-33}).
Hence, we can consider the following groups:
\[
K_{1}\left(R;d\right)=GL_{d}\left(R\right)/E_{d}\left(R\right)\,\,\,\,\,\,\,\,K_{1}\left(R,H;d\right)=GL_{d}\left(R,H\right)/E_{d}\left(R,H\right)
\]
\[
SK_{1}\left(R;d\right)=SL_{d}\left(R\right)/E_{d}\left(R\right)\,\,\,\,\,\,\,\,SK_{1}\left(R,H;d\right)=SL_{d}\left(R,H\right)/E_{d}\left(R,H\right).
\]

We now go ahead with the following definition:
\begin{defn}
Let $R$ be a commutative ring, and $3\leq d\in\mathbb{N}$. We define
the ``Steinberg group'' $St_{d}\left(R\right)$ to be the group
which generated by the elements $x_{i,j}\left(r\right)$ for $r\in R$
and $1\leq i\neq j\leq d$, under the relations:
\begin{itemize}
\item $x_{i,j}\left(r_{1}\right)\cdot x_{i,j}\left(r_{2}\right)=x_{i,j}\left(r_{1}+r_{2}\right)$.
\item $\left[x_{i,j}\left(r_{1}\right),x_{j,k}\left(r_{2}\right)\right]=x_{i,k}\left(r_{1}\cdot r_{2}\right)$.
\item $\left[x_{i,j}\left(r_{1}\right),x_{k,l}\left(r_{2}\right)\right]=1$.
\end{itemize}
for every different $1\leq i,j,k,l\leq d$ and every $r_{1},r_{2}\in R$.
\end{defn}

As the elementary matrices $I_{d}+rE_{i,j}$ satisfy the relations
which define $St_{d}\left(R\right)$, the map $x_{i,j}\left(r\right)\mapsto I_{d}+rE_{i,j}$
defines a natural homomorphism $\phi_{d}:St_{d}\left(R\right)\to E{}_{d}\left(R\right)$.
The kernel of this map is denoted by $K_{2}\left(R;d\right)=\ker\left(\phi_{d}\right)$.
Now, for two invertible elements $u,v\in R^{*}$ and $1\leq i\neq j\leq d$
define the ``Steinberg symbol'' by
\begin{eqnarray*}
 &  & \left\{ u,v\right\} _{i,j}=h_{i,j}\left(uv\right)h_{i,j}\left(u\right)^{-1}h_{i,j}\left(v\right)^{-1}\in St_{d}\left(R\right)\\
\\
 &  & \textrm{where}\,\,\,\,h_{i,j}\left(u\right)=w_{i,j}\left(u\right)w_{i,j}\left(-1\right)\\
\textrm{} &  & \textrm{and\,\,\,}\,\,\,\,w_{i,j}\left(u\right)=x_{i,j}\left(u\right)x_{j,i}\left(-u^{-1}\right)x_{i,j}\left(u\right).
\end{eqnarray*}
One can show that $\left\{ u,v\right\} _{i,j}\in K_{2}\left(R;d\right)$
and lie in the center of $St_{d}\left(R\right)$. In addition, for
every $3\leq d\in\mathbb{N}$, the Steinberg symbols $\left\{ u,v\right\} _{i,j}$
do not depend on the indices $i,j$, so they can be denoted simply
by $\left\{ u,v\right\} $ (see \cite{key-34}). The Steinberg symbols
satisfy many identities. For example
\begin{eqnarray}
\left\{ uv,w\right\} =\left\{ u,w\right\} \left\{ v,w\right\}  & , & \left\{ u,vw\right\} =\left\{ u,v\right\} \left\{ u,w\right\} .\label{eq:DS}
\end{eqnarray}
In the semi-local case we have the following (\cite{key-35}, Theorem
2.7):
\begin{thm}
\label{thm:dennis-stein-1}Let $R$ be a semi-local commutative ring
and $d\geq3$. Then, $K_{2}\left(R;d\right)$ is generated by the
Steinberg symbols $\left\{ u,v\right\} $ for $u,v\in R^{*}$. In
particular, $K_{2}\left(R;d\right)$ is central in $St_{d}\left(R\right)$.
\end{thm}

Let now $R$ be a commutative ring, $H\vartriangleleft R$ an ideal
and $d\geq3$. Denote $\bar{R}=R/H$. Clearly, there is a natural
map $E_{d}\left(R\right)\to E{}_{d}\left(\bar{R}\right)$. It is clear
that $E_{d}\left(R,H\right)$ lies in the kernel of the latter map,
so we have a map
\[
\pi_{d}:E_{d}\left(R\right)/E_{d}\left(R,H\right)\to E_{d}\left(\bar{R}\right).
\]

In addition, it is easy to see that we have a surjective map
\[
\psi_{d}:St_{d}\left(\bar{R}\right)\twoheadrightarrow E_{d}\left(R\right)/E_{d}\left(R,H\right)
\]
defined by: $x_{i,j}\left(\bar{r}\right)\to I_{d}+rE_{i,j}$, such
that $\phi_{d}:St_{d}\left(\bar{R}\right)\to E_{d}\left(\bar{R}\right)$
satisfies: $\phi_{d}=\pi_{d}\circ\psi_{d}$. Therefore, we obtain
the surjective map
\begin{eqnarray*}
K_{2}\left(\bar{R};d\right)=\ker\left(\phi_{d}\right)\overset{\psi_{d}}{\twoheadrightarrow}\ker\left(\pi_{d}\right) & = & \left(E_{d}\left(R\right)\cap SL_{d}\left(R,H\right)\right)/E_{d}\left(R,H\right)\\
 & \leq & SK_{1}\left(R,H;d\right).
\end{eqnarray*}
In particular, it implies that if $E_{d}\left(R\right)=SL_{d}\left(R\right)$,
then we have a natural surjective map
\[
K_{2}\left(R/H;d\right)\twoheadrightarrow SK_{1}\left(R,H;d\right).
\]
From this one can easily deduce the following corollary, which will
be needed later in the paper.
\begin{cor}
\label{cor:important-cor}Let $R$ be a commutative ring, $H\vartriangleleft R$
ideal of finite index and $d\geq3$. Assume also that $E_{d}\left(R\right)=SL_{d}\left(R\right)$.
Then:
\begin{enumerate}
\item $SK_{1}\left(R,H;d\right)$ is a finite group.
\item $SK_{1}\left(R,H;d\right)$ is central in $GL_{d}\left(R\right)/E_{d}\left(R,H\right)$.
\item Every element of $SK_{1}\left(R,H;d\right)$ has a representative
in $SL_{d}\left(R,H\right)$ of the form $\left(\begin{array}{cc}
A & 0\\
0 & I_{d-2}
\end{array}\right)$ such that $A\in SL_{2}\left(R,H\right)$. 
\end{enumerate}
\end{cor}

\begin{proof}
The ring $\bar{R}=R/H$ is finite. In particular, $\bar{R}$ is Artinian
and hence semi-local. Thus, by Theorem \ref{thm:dennis-stein-1},
$K_{2}\left(\bar{R};d\right)$ is an abelian group which is generated
by the Steinberg symbols $\left\{ u,v\right\} $ for $u,v\in\bar{R}^{*}$.
As $\bar{R}$ is finite, so is the number of the Steinberg symbols.
From Equation (\ref{eq:DS}) we obtain that the order of any Steinberg
symbol is finite. So $K_{2}\left(\bar{R};d\right)$ is a finitely
generated abelian group whose generators are of finite order. Thus,
$K_{2}\left(\bar{R};d\right)$ is finite. Moreover, as $\bar{R}$
is semi-local, Theorem \ref{thm:dennis-stein-1} implies that $K_{2}\left(\bar{R};d\right)$
is central $St_{d}\left(\bar{R}\right)$. Now, as we assume that $E_{d}\left(R\right)=SL_{d}\left(R\right)$,
we obtain that $SK_{1}\left(R,H;d\right)$ is the image of $K_{2}\left(\bar{R};d\right)$
under the surjective map 
\[
St_{d}\left(\bar{R}\right)\twoheadrightarrow E_{d}\left(R\right)/E_{d}\left(R,H\right)=SL_{d}\left(R\right)/E_{d}\left(R,H\right).
\]
This implies Part 1 and that $SK_{1}\left(R,H;d\right)$ is central
in $SL_{d}\left(R\right)/E_{d}\left(R,H\right)$. 

Moreover, as $d\geq3$, we have $\left\{ u,v\right\} =\left\{ u,v\right\} _{1,2}$
for every $u,v\in\bar{R}^{*}$. Now, it is easy to check from the
definition of the Steinberg symbols that the image of $\left\{ u,v\right\} _{1,2}$
under the map $St_{d}\left(\bar{R}\right)\twoheadrightarrow SL_{d}\left(R\right)/E_{d}\left(R,H\right)$
is of the form 
\begin{equation}
\left(\begin{array}{cc}
A & 0\\
0 & I_{d-2}
\end{array}\right)\cdot E_{d}\left(R,H\right)\label{eq:2X2}
\end{equation}
for some $A\in SL_{2}\left(R,H\right)$. So as $SK_{1}\left(R,H;d\right)$
is generated by the images of the Steinberg symbols, the same holds
for every element in $SK_{1}\left(R,H;d\right)$. So we obtain Part
3. Now, as $d\geq3$ we can write
\[
GL_{d}\left(R\right)=SL_{d}\left(R\right)\cdot\left\{ I_{d}+\left(r-1\right)E_{3,3}\,|\,r\in R^{*}\right\} .
\]
Observe also that mod $E_{d}\left(R,H\right)$, all the elements of
the form $I_{d}+\left(r-1\right)E_{3,3}$ for $r\in R^{*}$ commute
with all the elements of the form (\ref{eq:2X2}). Hence, the centrality
of $SK_{1}\left(R,H;d\right)$ in $SL_{d}\left(R\right)/E_{d}\left(R,H\right)$
shows that actually $SK_{1}\left(R,H;d\right)$ is central in $GL_{d}\left(R\right)/E_{d}\left(R,H\right)$,
as required in Part 2.
\end{proof}

\section{\label{sec:some-properties}$IA\left(\Phi_{n}\right)$ and its subgroups}

We start our discussion of the IA-automorphism group of the free metabelian
group, $G=IA\left(\Phi_{n}\right)=\ker\left(Aut\left(\Phi_{n}\right)\to Aut\left(\Phi_{n}/\Phi'_{n}\right)=GL_{n}\left(\mathbb{Z}\right)\right)$,
by presenting some of its properties and subgroups. We begin with
the following notations:
\begin{itemize}
\item $\Phi=\Phi_{n}=F_{n}/F''_{n}$= the free metabelian group on $n$
elements. Here $F''_{n}$ denotes the second derivative of $F_{n}$,
the free group on $n$ elements.
\item $\Psi_{m}=\Phi/M_{m}$, where $M_{m}=\left(\Phi'\Phi^{m}\right)'\left(\Phi'\Phi^{m}\right)^{m}$.
\item $IG_{m}=G(M_{m})=\ker\left(IA\left(\Phi\right)\to Aut\left(\Psi_{m}\right)\right).$
\item $IA_{m}=\cap\left\{ N\vartriangleleft IA\left(\Phi\right)\,|\,[IA\left(\Phi\right):N]\,|\,m\right\} $
\item $R_{n}=\mathbb{Z}[\mathbb{Z}^{n}]=\mathbb{Z}[x_{1}^{\pm1},\ldots,x_{n}^{\pm1}]$
where $x_{1},\ldots,x_{n}$ are the generators of $\mathbb{Z}^{n}$.
\item $\mathbb{Z}_{m}=\mathbb{Z}/m\mathbb{Z}$.
\item $\sigma_{i}=x_{i}-1$ for $1\leq i\leq n$. 
\item $\vec{\sigma}=$ the column vector which has $\sigma_{i}$ in its
$i$-th entry.
\item $\mathfrak{A}=\sum_{i=1}^{n}\sigma_{i}R_{n}$ = the augmentation ideal
of $R_{n}$.
\item $H_{m}=\ker\left(R_{n}\to\mathbb{Z}_{m}[\mathbb{Z}_{m}^{n}]\right)=\sum_{i=1}^{n}\left(x_{i}^{m}-1\right)R_{n}+mR_{n}$.
\end{itemize}
By the well known Magnus embedding (see \cite{key-36}, \cite{key-37},
\cite{key-35-1}), one can identify $\Phi$ with the matrix group
\[
\Phi=\left\{ \left(\begin{array}{cc}
g & a_{1}t_{1}+\ldots+a_{n}t_{n}\\
0 & 1
\end{array}\right)\,|\,g\in\mathbb{Z}^{n},\,a_{i}\in R_{n},\,g-1=\sum_{i=1}^{n}a_{i}\sigma_{i}\right\} 
\]
where $t_{i}$ is a free basis for an $R_{n}$-module, under the identification
of the generators of $\Phi$ with the matrices
\[
\left(\begin{array}{cc}
x_{i} & t_{i}\\
0 & 1
\end{array}\right)\,\,\,\,\textrm{for}\,\,\,\,1\leq i\leq n.
\]

Moreover, for every $\alpha\in IA\left(\Phi\right)$, one can describe
$\alpha$ by its action on the generators of $\Phi$, by
\[
\alpha:\left(\begin{array}{cc}
x_{i} & t_{i}\\
0 & 1
\end{array}\right)\mapsto\left(\begin{array}{cc}
x_{i} & a_{i,1}t_{1}+\ldots+a_{i,n}t_{n}\\
0 & 1
\end{array}\right).
\]

This description gives an injective homomorphism (see \cite{key-13},
\cite{key-36})
\begin{eqnarray*}
IA\left(\Phi\right) & \hookrightarrow & GL_{n}\left(R_{n}\right)\\
\textrm{defined by}\,\,\,\,\alpha & \mapsto & \left(\begin{array}{ccc}
a_{1,1} & \cdots & a_{1,n}\\
\vdots &  & \vdots\\
a_{n,1} & \cdots & a_{n,n}
\end{array}\right)
\end{eqnarray*}
which gives an identification of $IA\left(\Phi\right)$ with the group
\begin{eqnarray*}
IA\left(\Phi\right) & = & \left\{ A\in GL_{n}\left(R_{n}\right)\,|\,A\vec{\sigma}=\vec{\sigma}\right\} \\
 & = & \left\{ I_{n}+A\in GL_{n}\left(R_{n}\right)\,|\,A\vec{\sigma}=\vec{0}\right\} .
\end{eqnarray*}

Consider now the map
\[
\begin{array}{c}
\Phi=\left\{ \left(\begin{array}{cc}
g & a_{1}t_{1}+\ldots+a_{n}t_{n}\\
0 & 1
\end{array}\right)\,|\,g\in\mathbb{Z}^{n},\,a_{i}\in R_{n},\,g-1=\sum_{i=1}^{n}a_{i}\sigma_{i}\right\} \\
\downarrow\\
\left\{ \left(\begin{array}{cc}
g & a_{1}t_{1}+\ldots+a_{n}t_{n}\\
0 & 1
\end{array}\right)\,|\,g\in\mathbb{Z}_{m}^{n},\,a_{i}\in\mathbb{Z}_{m}[\mathbb{Z}_{m}^{n}],\,g-1=\sum_{i=1}^{n}a_{i}\sigma_{i}\right\} 
\end{array}
\]
which is induced by the projections $\mathbb{Z}^{n}\to\mathbb{Z}_{m}^{n}$,
$R_{n}=\mathbb{Z}[\mathbb{Z}^{n}]\to\mathbb{Z}_{m}[\mathbb{Z}_{m}^{n}]$.
Using result of Romanovski\u{\i} \cite{key-40}, it is shown in \cite{key-6}
that this map is surjective and that $\Psi_{m}$ is canonically isomorphic
to its image. Therefore, we can identify $IG_{m}$, the principal
congruence subgroup of $IA\left(\Phi\right)$, with
\begin{eqnarray*}
IG_{m} & = & \left\{ A\in\ker\left(GL_{n}\left(R_{n}\right)\to GL_{n}\left(\mathbb{Z}_{m}[\mathbb{Z}_{m}^{n}]\right)\right)\,|\,A\vec{\sigma}=\vec{\sigma}\right\} \\
 & = & \left\{ I_{n}+A\in GL_{n}\left(R_{n},H_{m}\right)\,|\,A\vec{\sigma}=\vec{0}\right\} .
\end{eqnarray*}
\begin{prop}
\label{prop:augmentation}Let $I_{n}+A\in IA\left(\Phi\right)$ and
denote the entries of $A$ by $a_{k,l}$ for $1\leq k,l\leq n$. Then,
for every $1\leq k,l\leq n$, we have: $a_{k,l}\in\sum_{l\neq i=1}^{n}\sigma_{i}R_{n}\subseteq\mathfrak{A}$.
\end{prop}

\begin{proof}
For a given $1\leq k\leq n$, the condition $A\vec{\sigma}=\vec{0}$
gives the equality
\[
0=a_{k,1}\sigma_{1}+a_{k,2}\sigma_{2}+\ldots+a_{k,n}\sigma_{n}.
\]
Thus, for a given $1\leq l\leq n$, the map $R_{n}\to S_{l}=\mathbb{Z}[x_{l}^{\pm1}]$
which defined by $x_{i}\mapsto1$ for every $i\neq l$, maps 
\[
0=a_{k,1}\sigma_{1}+a_{k,2}\sigma_{2}+\ldots+a_{k,n}\sigma_{n}\mapsto\bar{a}_{k,l}\sigma_{l}\in\mathbb{Z}[x_{l}^{\pm1}].
\]
Hence, as $\mathbb{Z}[x_{l}^{\pm1}]$ is a domain, $\bar{a}_{k,l}=0\in\mathbb{Z}[x_{l}^{\pm1}]$.
Thus: $a_{k,l}\in\sum_{l\neq i=1}^{n}\sigma_{i}R_{n}\subseteq\mathfrak{A}$,
as required.
\end{proof}
\begin{prop}
\label{lem:det}Let $I_{n}+A\in IA\left(\Phi\right)$. Then $\det\left(I_{n}+A\right)$
is of the form: $\det\left(I_{n}+A\right)=\prod_{r=1}^{n}x_{r}^{s_{r}}$
for some $s_{r}\in\mathbb{Z}$.
\end{prop}

\begin{proof}
The invertible elements in $R_{n}$ are the elements of the form $\pm\prod_{i=1}^{n}x_{r}^{s_{r}}$
(see \cite{key-38}, chapter 8). Thus, as $I_{n}+A\in GL_{n}\left(R_{n}\right)$
we have: $\det\left(I_{n}+A\right)=\pm\prod_{i=1}^{n}x_{r}^{s_{r}}$.
However, according to Proposition \ref{prop:augmentation}, for every
entry $a_{k,l}$ of $A$ we have: $a_{k,l}\in\mathfrak{A}$. Hence,
under the projection $x_{i}\mapsto1$ for every $1\leq i\leq n$,
one has $I_{n}+A\mapsto I_{n}$ and thus, $\pm\prod_{i=1}^{n}x_{r}^{s_{r}}=\det\left(I_{n}+A\right)\mapsto\det\left(I_{n}\right)=1$.
Therefore, the option $\det\left(I_{n}+A\right)=-\prod_{i=1}^{n}x_{r}^{s_{r}}$
is impossible, as required.
\end{proof}
Let us step forward with the following definition:
\begin{defn}
Let $A\in GL_{n}\left(R_{n}\right)$, and for $1\leq i\leq n$ denote
by $A_{i,i}$ the minor which obtained from $A$ by erasing its $i$-th
row and $i$-th column. Now, for every $1\leq i\leq n$, define the
subgroup $IGL_{n-1,i}\leq IA\left(\Phi\right)$, by
\[
IGL_{n-1,i}=\left\{ I_{n}+A\in IA\left(\Phi\right)\,|\,\begin{array}{c}
\textrm{The\,\,}i\textrm{-th\,\, row\,\, of\,\,}A\textrm{\,\, is\,\,0,}\\
I_{n-1}+A_{i,i}\in GL_{n-1}\left(R_{n},\sigma_{i}R_{n}\right)
\end{array}\right\} .
\]
\end{defn}

\begin{prop}
\label{prop:iso}For every $1\leq i\leq n$ we have: $IGL_{n-1,i}\cong GL_{n-1}\left(R_{n},\sigma_{i}R_{n}\right)$.
\end{prop}

\begin{proof}
The definition of $IGL_{n-1,i}$ gives us a natural projection from
$IGL_{n-1,i}\to GL_{n-1}\left(R_{n},\sigma_{i}R_{n}\right)$ which
maps an element $I_{n}+A\in IGL_{n-1,i}$ to $I_{n-1}+A_{i,i}\in GL_{n-1}\left(R_{n},\sigma_{i}R_{n}\right)$.
Thus, all we need is to explain why this map is injective and surjective.

\uline{Injectivity:} Here, It is enough to show that given an element
$I_{n}+A\in IA\left(\Phi\right)$, every entry in the $i$-th column
is determined uniquely by the other entries in its row. Indeed, as
$A$ satisfies the condition $A\vec{\sigma}=\vec{0}$, for every $1\leq k\leq n$
we have
\begin{equation}
a_{k,1}\sigma_{1}+a_{k,2}\sigma_{2}+\ldots+a_{k,n}\sigma_{n}=0\,\,\Rightarrow\,\,a_{k,i}=\frac{-\sum_{i\neq l=1}^{n}a_{k,l}\sigma_{l}}{\sigma_{i}}\label{eq:terms of others}
\end{equation}
i.e. we have a formula for $a_{k,i}$ in terms of the other entries
in its row.

\uline{Surjectivity:} Without loss of generality we assume $i=n$.
Let $I_{n-1}+\sigma_{n}B\in GL_{n-1}\left(R_{n},\sigma_{n}R_{n}\right)$,
and denote by $\vec{b}_{l}$ the column vectors of $B$. Define
\[
\left(\begin{array}{cc}
I_{n-1}+\sigma_{n}B & -\sum_{l=1}^{n-1}\sigma_{l}\vec{b}_{l}\\
0 & 1
\end{array}\right)\in IGL_{n-1,n}
\]
and this is clearly a preimage of $I_{n-1}+\sigma_{n}B$.
\end{proof}
Under the above identification of $IGL_{n-1,i}$ with $GL_{n-1}\left(R_{n},\sigma_{i}R_{n}\right)$,
we will use throughout the paper the following notations:
\begin{defn}
\label{def:IS-IE}Let $H\vartriangleleft R_{n}$. We define
\begin{eqnarray*}
ISL_{n-1,i}\left(H\right) & = & IGL_{n-1,i}\cap SL_{n-1}\left(R_{n},H\right)\\
IE_{n-1,i}\left(H\right) & = & IGL_{n-1,i}\cap E{}_{n-1}\left(R_{n},H\right)\leq ISL_{n-1,i}\left(H\right).
\end{eqnarray*}
\end{defn}

Observe that as for every $1\leq i\leq n$, we have 
\[
GL_{n-1}(\mathbb{Z}[x_{i}^{\pm1}],\sigma_{i}\mathbb{Z}[x_{i}^{\pm1}])\leq GL_{n-1}(R_{n},\sigma_{i}R_{n})
\]
the isomorphism $GL_{n-1}(R_{n},\sigma_{i}R_{n})\cong IGL_{n-1,i}\leq IA\left(\Phi\right)$
gives also a natural embedding of $GL_{n-1}(\mathbb{Z}[x_{i}^{\pm1}],\sigma_{i}\mathbb{Z}[x_{i}^{\pm1}])$
as a subgroup of $IA\left(\Phi\right)$. Actually:
\begin{prop}
\label{prop:semi-direct}For every $1\leq i\leq n$, there is a canonical
surjective homomorphism
\[
\rho_{i}:IA\left(\Phi\right)\twoheadrightarrow GL_{n-1}(\mathbb{Z}[x_{i}^{\pm1}],\sigma_{i}\mathbb{Z}[x_{i}^{\pm1}])
\]
such that the following composition map is the identity:
\[
GL_{n-1}(\mathbb{Z}[x_{i}^{\pm1}],\sigma_{i}\mathbb{Z}[x_{i}^{\pm1}])\hookrightarrow IA\left(\Phi\right)\overset{\rho_{i}}{\twoheadrightarrow}GL_{n-1}(\mathbb{Z}[x_{i}^{\pm1}],\sigma_{i}\mathbb{Z}[x_{i}^{\pm1}]).
\]
Hence $IA\left(\Phi\right)=\ker\rho_{i}\rtimes GL_{n-1}(\mathbb{Z}[x_{i}^{\pm1}],\sigma_{i}\mathbb{Z}[x_{i}^{\pm1}])$.
\end{prop}

\begin{proof}
Without loss of generality we assume $i=n$. First, consider the homomorphism
$IA\left(\Phi\right)\to GL_{n}(\mathbb{Z}[x_{n}^{\pm1}])$, which
is induced by the projection $R_{n}\to\mathbb{Z}[x_{n}^{\pm1}]$ that
is defined by $x_{j}\mapsto1$ for every $j\neq n$. By Proposition
\ref{prop:augmentation}, given $I_{n}+A\in IA\left(\Phi\right)$,
all the entries of the $n$-th column of $A$ are in $\sum_{j=1}^{n-1}\sigma_{j}R_{n}$.
Hence, the above map $IA\left(\Phi\right)\to GL_{n}(\mathbb{Z}[x_{n}^{\pm1}])$
is actually a map
\[
IA\left(\Phi\right)\to\left\{ I_{n}+\bar{A}\in GL_{n}(\mathbb{Z}[x_{n}^{\pm1}])\,|\,\textrm{the\,\,}n\textrm{-th\,\, column\,\, of\,\,}\bar{A}\textrm{\,\, is\,\,}\vec{0}\right\} .
\]

Observe now, that the right side of the above map is mapped naturally
to $GL_{n-1}(\mathbb{Z}[x_{n}^{\pm1}])$ by erasing the $n$-th column
and the $n$-th row from every element. Hence we obtain a map 
\[
IA\left(\Phi\right)\to GL_{n-1}(\mathbb{Z}[x_{n}^{\pm1}]).
\]

Now, by Proposition \ref{prop:augmentation}, every entry of $A$
such that $I_{n}+A\in IA\left(\Phi\right)$, is in $\mathfrak{A}$.
Thus, the entries of every $\bar{A}$ such that $I_{n-1}+\bar{A}\in GL_{n-1}(\mathbb{Z}[x_{n}^{\pm1}])$
is an image of $I_{n}+A\in IA\left(\Phi\right)$, are all in $\sigma_{n}\mathbb{Z}[x_{n}^{\pm1}]$.
Hence, we actually obtain a homomorphism
\[
\rho_{n}:IA\left(\Phi\right)\to GL_{n-1}(\mathbb{Z}[x_{n}^{\pm1}],\sigma_{n}\mathbb{Z}[x_{n}^{\pm1}]).
\]

Observing that the copy of $GL_{n-1}(\mathbb{Z}[x_{n}^{\pm1}],\sigma_{n}\mathbb{Z}[x_{n}^{\pm1}])$
in $IGL_{n-1,n}$ is mapped isomorphically to itself by $\rho_{n}$,
finishes the proof.
\end{proof}
\begin{prop}
\label{prop:intersection}Denote $S_{i}=\mathbb{Z}[x_{i}^{\pm1}]\subseteq R_{n}$,
$J_{i,m}=(x_{i}^{m}-1)S_{i}+mS_{i}\subseteq H_{m}$ for $1\leq i\leq n$.
Then, by identifying $\textrm{Im}(\rho_{i})\cong GL_{n-1}(\mathbb{Z}[x_{i}^{\pm1}],\sigma_{i}\mathbb{Z}[x_{i}^{\pm1}])=GL_{n-1}(S_{i},\sigma_{i}S_{i})$,
for every $m\in\mathbb{N}$ one has
\[
\textrm{Im}(\rho_{i})\cap IG_{m}=GL_{n-1}(S_{i},\sigma_{i}J_{i,m}).
\]
\end{prop}

\begin{proof}
By the identification 
\[
IG_{m}=\left\{ I_{n}+A\in GL_{n}(R_{n},H_{m})\,|\,A\vec{\sigma}=\vec{0}\right\} 
\]
and by applying the formula of Equation (\ref{eq:terms of others})
to the $i$-th column of elements in $IGL_{n-1,i}$, it is easy to
see that the elements of $IGL_{n-1,i}$ which correspond to the elements
of $GL_{n-1}(S_{i},\sigma_{i}J_{i,m})$ are clearly in $\textrm{Im}\rho_{i}\cap IG_{m}$.
For the opposite inclusion, without loss of generality assume that
$i=n$, and let $I_{n}+A\in\textrm{Im}\rho_{n}\cap IG_{m}$. Then
$I_{n}+A$ has the form
\[
\left(\begin{array}{cc}
I_{n-1}+\sigma_{n}B & -\sum_{l=1}^{n-1}\sigma_{l}\vec{b}_{l}\\
0 & 1
\end{array}\right)\in IGL_{n-1,n}
\]
where the entries of $B$ satisfy $b_{k,l}\in S_{n}$ and $\sum_{j=1}^{n-1}\sigma_{j}b_{k,j}\in H_{m}$.
Notice now that for every $l\neq n$, by projecting $\sigma_{j}\mapsto0$
for $j\neq l,n$, we see that actually $\sigma_{l}b_{k,l}\in H_{m}$.
From here it is easy to see that we necessarily have $b_{k,l}\in H_{m}$.
I.e. $b_{k,l}\in H_{m}\cap S_{n}=(x_{n}^{m}-1)S_{n}+mS_{n}=J_{n,m}$,
and the claim follows.
\end{proof}
\begin{prop}
\label{prop:contain}For every $1\leq i\leq n$ and $m\in\mathbb{N}$
one has
\[
\rho_{i}(IG_{m^{2}})\subseteq\textrm{Im}(\rho_{i})\cap IG_{m}\subseteq\rho_{i}(IG_{m}).
\]
\end{prop}

\begin{proof}
As every element in $\textrm{Im}\rho_{i}$ is mapped to itself via
$\rho_{i}$ we clearly have $\textrm{Im}\rho_{i}\cap IG_{m}=\rho_{i}(\textrm{Im}\rho_{i}\cap IG_{m})\subseteq\rho_{i}(IG_{m})$.
On the other hand, if $I_{n}+A\in IG_{m^{2}}$ then viewing $\textrm{Im}\rho_{i}\cong GL_{n-1}(S_{i},\sigma_{i}S_{i})$
for $S_{i}=\mathbb{Z}[x_{i}^{\pm1}]$, the entries of $\rho_{i}(I_{n}+A)=I_{n-1}+B$
belong to $(x_{i}^{m^{2}}-1)S_{i}+m^{2}\sigma_{i}S_{i}$. Observe
now that we have: $\sum_{r=0}^{m-1}x_{i}^{mr}\subseteq(x_{i}^{m}-1)S_{i}+mS_{i}=J_{i,m}$.
Hence
\begin{equation}
x_{i}^{m^{2}}-1=\sigma_{i}\sum_{r=1}^{m^{2}-1}x_{i}^{r}=\sigma_{i}\sum_{r=0}^{m-1}x_{i}^{r}\sum_{r=0}^{m-1}x_{i}^{mr}\in\sigma_{i}J_{i,m}\label{eq:identity}
\end{equation}
So by proposition \ref{prop:intersection} $\rho_{i}(I_{n}+A)\in\textrm{Im}\rho_{i}\cap IG_{m}$
as required.
\end{proof}
\begin{prop}
\label{prop:equals}For every $m\in\mathbb{N}$ and $1\leq i\leq n$
one has
\[
\rho_{i}\left(IA_{m}\right)=\textrm{Im}(\rho_{i})\cap IA_{m}
\]
where $IA_{m}=\cap\left\{ N\vartriangleleft IA\left(\Phi\right)\,|\,[IA\left(\Phi\right):N]\,|\,m\right\} $.
\end{prop}

\begin{proof}
As every element in $\textrm{Im}\rho_{i}$ is mapped to itself via
$\rho_{i}$ we clearly have $\textrm{Im}\rho_{i}\cap IA_{m}=\rho_{i}\left(\textrm{Im}\rho_{i}\cap IA_{m}\right)\subseteq\rho_{i}\left(IA_{m}\right)$.
For the opposite, assume that $\alpha\in IA_{m}$, and denote $\rho_{i}\left(\alpha\right)=\beta\in\textrm{Im}\rho_{i}$.
We want to show that $\beta\in IA_{m}$. So let $N\vartriangleleft IA(\Phi)$
such that $[IA(\Phi):N]|m$. Then obviously $[\textrm{Im}\rho_{i}:(N\cap\textrm{Im}\rho_{i})]|m$.
Thus, as $\rho_{i}$ is surjective $[IA(\Phi):\rho_{i}^{-1}(N\cap\textrm{Im}\rho_{i})]|m$
so $\alpha\in\rho_{i}^{-1}(N\cap\textrm{Im}\rho_{i})$ and hence $\beta=\rho_{i}(\alpha)\in N\cap\textrm{Im}\rho_{i}\leq N$.
As this is valid for every such $N$, we have: $\beta\in IA_{m}$,
as required.
\end{proof}
We close this section with the following definition:
\begin{defn}
\label{def:tag}For every $1\leq i\leq n$, denote
\[
IGL'_{n-1,i}=\left\{ I_{n}+A\in IA\left(\Phi\right)\,|\,\textrm{The\,\,}i\textrm{-th\,\, row\,\, of\,\,}A\textrm{\,\, is\,\,}0\right\} .
\]
\end{defn}

Obviously, $IGL_{n-1,i}\leq IGL'{}_{n-1,i}$, and by the same injectivity
argument as in the proof of Proposition \ref{prop:iso}, one can deduce
that:
\begin{prop}
\label{prop:tag}The subgroup $IGL'_{n-1,i}\leq IA\left(\Phi\right)$
is canonically embedded in $GL_{n-1}\left(R_{n}\right)$, by the map:
$I_{n}+A\mapsto I_{n-1}+A_{i,i}$.
\end{prop}

\begin{rem}
Note that in general $IGL_{n-1,i}\lneqq IGL'{}_{n-1,i}$. For example,
$I_{4}+\sigma_{3}E_{1,2}-\sigma_{2}E_{1,3}\in IGL'{}_{3,4}\setminus IGL_{3,4}$.
\end{rem}

\section{\label{sec: not f.g.}The subgroups $C_{i}$ }

In this section we define the subgroups $C_{i}\leq C\left(IA\left(\Phi_{n}\right),\Phi_{n}\right)$,
and we show that for each $i$ we can view $C\left(IA\left(\Phi_{n}\right),\Phi_{n}\right)$
as a semi-direct product of $C_{i}$ with another subgroup. We also
show that when $n\geq4$
\[
C_{i}\cong\ker(\widehat{SL_{n-1}(\mathbb{Z}[x^{\pm1}])}\to SL_{n-1}(\widehat{\mathbb{Z}[x^{\pm1}]}))
\]
and use it to show that $C\left(IA\left(\Phi_{n}\right),\Phi_{n}\right)$
is not finitely generated. We recall the notations:
\begin{itemize}
\item $\Phi=\Phi_{n}$.
\item $\Psi_{m}=\Phi/M_{m}$, where $M_{m}=\left(\Phi'\Phi^{m}\right)'\left(\Phi'\Phi^{m}\right)^{m}$.
\item $IG_{m}=G(M_{m})=\ker(IA\left(\Phi\right)\to Aut(\Psi_{m})).$
\item $IA_{m}=\cap\left\{ N\vartriangleleft IA\left(\Phi\right)\,|\,[IA\left(\Phi\right):N]\,|\,m\right\} $.
\end{itemize}
It is proven in \cite{key-6} that $\hat{\Phi}=\underleftarrow{\lim}\Psi_{m}$.
So, as for every $m\in\mathbb{N}$ the group $\ker(\Phi\to\Psi_{m})$
is characteristic in $\Phi$, we can write explicitly
\begin{eqnarray*}
C\left(IA\left(\Phi\right),\Phi\right) & = & \ker(\widehat{IA\left(\Phi\right)}\to Aut(\hat{\Phi}))\\
 & = & \ker(\widehat{IA\left(\Phi\right)}\to\underleftarrow{\lim}Aut(\Psi_{m}))\\
 & = & \ker(\widehat{IA\left(\Phi\right)}\to\underleftarrow{\lim}(IA\left(\Phi\right)/IG_{m})).
\end{eqnarray*}

Now, as for every $n\geq4$ we know that $IA\left(\Phi\right)$ is
finitely generated (see \cite{key-24}), as explained in Remark \ref{rem:action},
we have $\widehat{IA\left(\Phi\right)}=\underleftarrow{\lim}(IA\left(\Phi\right)/IA_{m})$.
Hence
\begin{eqnarray*}
C(IA\left(\Phi\right),\Phi) & = & \ker(\underleftarrow{\lim}(IA\left(\Phi\right)/IA_{m})\to\underleftarrow{\lim}(IA\left(\Phi\right)/IG_{m}))\\
 & = & \ker(\underleftarrow{\lim}(IA\left(\Phi\right)/IA_{m})\to\underleftarrow{\lim}(IA\left(\Phi\right)/IG_{m}\cdot IA_{m}))\\
 & = & \underleftarrow{\lim}(IA_{m}\cdot IG_{m}/IA_{m}).
\end{eqnarray*}
Similarly, we can write $C(IA\left(\Phi\right),\Phi)=\underleftarrow{\lim}(IA_{m}\cdot IG_{m^{2}}/IA_{m})$.

Remember now that for every $1\leq i\leq n$ the composition map
\[
GL_{n-1}(\mathbb{Z}[x_{i}^{\pm1}],\sigma_{i}\mathbb{Z}[x_{i}^{\pm1}])\hookrightarrow IA\left(\Phi\right)\overset{\rho_{i}}{\twoheadrightarrow}GL_{n-1}(\mathbb{Z}[x_{i}^{\pm1}],\sigma_{i}\mathbb{Z}[x_{i}^{\pm1}])
\]
is the identity on $GL_{n-1}\left(\mathbb{Z}[x_{i}^{\pm1}],\sigma_{i}\mathbb{Z}[x_{i}^{\pm1}]\right)$.
Hence, the induced composition map of the profinite completions

\[
\widehat{GL_{n-1}(\mathbb{Z}[x_{i}^{\pm1}],\sigma_{i}\mathbb{Z}[x_{i}^{\pm1}])}\overset{\hat{\varrho}}{\to}\widehat{IA\left(\Phi\right)}\overset{\hat{\rho}_{i}}{\twoheadrightarrow}\widehat{GL_{n-1}(\mathbb{Z}[x_{i}^{\pm1}],\sigma_{i}\mathbb{Z}[x_{i}^{\pm1}])}
\]
is the identity on $\widehat{GL_{n-1}(\mathbb{Z}[x_{i}^{\pm1}],\sigma_{i}\mathbb{Z}[x_{i}^{\pm1}])}$.
In particular, the map $\hat{\varrho}$ is injective, so we can write
\[
\widehat{GL_{n-1}(\mathbb{Z}[x_{i}^{\pm1}],\sigma_{i}\mathbb{Z}[x_{i}^{\pm1}])}\hookrightarrow\widehat{IA\left(\Phi\right)}\overset{\hat{\rho}_{i}}{\twoheadrightarrow}\widehat{GL_{n-1}(\mathbb{Z}[x_{i}^{\pm1}],\sigma_{i}\mathbb{Z}[x_{i}^{\pm1}]).}
\]
This enables us to write: $IA\left(\Phi\right)=\ker\rho_{i}\rtimes\textrm{Im}\rho_{i}$
and $\widehat{IA\left(\Phi\right)}=\ker\hat{\rho}_{i}\rtimes\textrm{Im}\hat{\rho}_{i}$. 
\begin{defn}
We define
\[
C_{i}=C\left(IA\left(\Phi\right),\Phi\right)\cap\textrm{Im}\hat{\rho}_{i}=\ker(\textrm{Im}\hat{\rho}_{i}\to Aut(\hat{\Phi})).
\]
\end{defn}

\begin{prop}
\label{prop:ker-contain}If $1\leq i\neq j\leq n$, then $C_{i}\subseteq\ker\hat{\rho}_{j}$.
In particular, for every $i\neq j$ we have: $C_{i}\cap C_{j}=\left\{ e\right\} $.
\end{prop}

\begin{proof}
By the explicit description $\widehat{IA\left(\Phi\right)}=\underleftarrow{\lim}(IA\left(\Phi\right)/IA_{m})$,
one can write
\begin{eqnarray*}
C_{i} & = & \ker(\textrm{Im}\hat{\rho}_{i}\to Aut(\hat{\Phi}))\\
 & = & \ker(\underleftarrow{\lim}(IA_{m}\cdot\textrm{Im}\rho_{i}/IA_{m})\to\underleftarrow{\lim}(IA\left(\Phi\right)/IG_{m}))\\
 & = & \ker(\underleftarrow{\lim}(IA_{m}\cdot\textrm{Im}\rho_{i}/IA_{m})\to\underleftarrow{\lim}(IA\left(\Phi\right)/IG_{m}\cdot IA_{m}))\\
 & = & \underleftarrow{\lim}((IA_{m}\cdot\textrm{Im}\rho_{i})\cap(IA_{m}\cdot IG_{m}))/IA_{m}
\end{eqnarray*}
and similarly $C_{i}=\underleftarrow{\lim}((IA_{m}\cdot\textrm{Im}\rho_{i})\cap(IA_{m}\cdot IG_{m^{2}}))/IA_{m}$.
We claim now that
\begin{align*}
 & (IA_{m}\cdot\textrm{Im}\rho_{i})\cap(IA_{m}\cdot IG_{m^{2}})\\
 & \subseteq IA_{m}\cdot(\textrm{Im}\rho_{i}\cap IG_{m})\\
 & \subseteq(IA_{m}\cdot\textrm{Im}\rho_{i})\cap(IA_{m}\cdot IG_{m}).
\end{align*}

The second inclusion is obvious. For the first one, we have to show
that if $ar=bs$ such that $a,b\in IA_{m}$, $r\in\textrm{Im}\rho_{i}$
and $s\in IG_{m^{2}}$, then there exist $c\in IA_{m}$ and $t\in\textrm{Im}\rho_{i}\cap IG_{m}$
such that $ar=bs=ct$. Indeed, write: $\textrm{Im}\rho_{i}\ni r=a^{-1}bs$.
Then: $r=\rho_{i}\left(r\right)=\rho_{i}(a^{-1}b)\rho_{i}\left(s\right)$,
and by Propositions \ref{prop:contain} and \ref{prop:equals}
\begin{eqnarray*}
\rho_{i}(a^{-1}b) & \in & \rho_{i}(IA_{m})=\textrm{Im}\rho_{i}\cap IA_{m}\\
\rho_{i}\left(s\right) & \in & \rho_{i}(IG_{m^{2}})\subseteq\textrm{Im}\rho_{i}\cap IG_{m}.
\end{eqnarray*}
Therefore, by defining $c=a\cdot\rho_{i}(a^{-1}b)$, and $t=\rho_{i}\left(s\right)$
we get the required inclusion. Thus, for $j\neq i$ we have
\begin{align*}
C_{i} & =\underleftarrow{\lim}(IA_{m}\cdot(\textrm{Im}\rho_{i}\cap IG_{m})/IA_{m})\\
 & \overset{\hat{\rho}_{j}}{\twoheadrightarrow}\underleftarrow{\lim}\rho_{j}(IA_{m})\cdot\rho_{j}(\textrm{Im}\rho_{i}\cap IG_{m})/\rho_{j}(IA_{m}).
\end{align*}
Using the definition of $\rho_{j}$ it is not difficult to show that
\begin{eqnarray*}
\rho_{j}(\textrm{Im}\rho_{i}\cap IG_{m}) & = & \left\langle I_{n}+m(\sigma_{i}E_{k,j}-\sigma_{j}E_{k,i})\,|\,k\neq i,j\right\rangle \\
 & = & \rho_{j}(\left\langle I_{n}+m(\sigma_{i}E_{k,j}-\sigma_{j}E_{k,i})\,|\,k\neq i,j\right\rangle )\\
 & = & \rho_{j}(\left\langle I_{n}+\sigma_{i}E_{k,j}-\sigma_{j}E_{k,i}\,|\,k\neq i,j\right\rangle ^{m})\subseteq\rho_{j}(IA_{m}).
\end{eqnarray*}
Hence, $C_{i}\subseteq\ker\hat{\rho}_{j}$, as required.
\end{proof}
We can now prove the following proposition:
\begin{prop}
\label{cor:C semi-direct}For every $1\leq i\leq n$ we have
\[
C_{i}\hookrightarrow C(IA\left(\Phi\right),\Phi)\overset{\hat{\rho}_{i}}{\twoheadrightarrow}C_{i}.
\]
In particular: $C\left(IA\left(\Phi\right),\Phi\right)=(\ker\hat{\rho}_{i}\cap C\left(IA\left(\Phi\right),\Phi\right))\rtimes C_{i}$.
\end{prop}

\begin{proof}
In the proof of Proposition \ref{prop:ker-contain} we saw that 
\[
C_{i}=\underleftarrow{\lim}(IA_{m}\cdot(\textrm{Im}\rho_{i}\cap IG_{m})/IA_{m}).
\]
Similarly $C_{i}=\underleftarrow{\lim}(IA_{m}\cdot(\textrm{Im}\rho_{i}\cap IG_{m^{2}})/IA_{m})$.
We remind that by Propositions \ref{prop:contain} and \ref{prop:equals}
we have 
\begin{align*}
\rho_{i}(IG_{m^{2}}) & \subseteq\textrm{Im}\rho_{i}\cap IG_{m}\subseteq\rho_{i}(IG_{m})\\
\rho_{i}(IA_{m}) & =\textrm{Im}\rho_{i}\cap IA_{m}.
\end{align*}
Therefore, we have 
\begin{eqnarray*}
C_{i} & = & \underleftarrow{\lim}IA_{m}\cdot(\textrm{Im}\rho_{i}\cap IG_{m})/IA_{m}=\underleftarrow{\lim}IA_{m}\cdot(\textrm{Im}\rho_{i}\cap IG_{m^{2}})/IA_{m}\\
 & \hookrightarrow & \underleftarrow{\lim}IA_{m}\cdot IG_{m}/IA_{m}=\underleftarrow{\lim}IA_{m}\cdot IG_{m^{2}}/IA_{m}=C(IA\left(\Phi\right),\Phi)\\
 & \overset{\hat{\rho}_{i}}{\twoheadrightarrow} & \underleftarrow{\lim}\rho_{i}(IA_{m})\cdot\rho_{i}(IG_{m})/\rho_{i}(IA_{m})=\underleftarrow{\lim}\rho_{i}(IA_{m})\cdot\rho_{i}(IG_{m^{2}})/\rho_{i}(IA_{m})\\
 & = & \underleftarrow{\lim}(\textrm{Im}\rho_{i}\cap IA_{m})\cdot(\textrm{Im}\rho_{i}\cap IG_{m})/(\textrm{Im}\rho_{i}\cap IA_{m})\\
 & = & \underleftarrow{\lim}IA_{m}\cdot(\textrm{Im}\rho_{i}\cap IG_{m})/IA_{m}=C_{i}.
\end{eqnarray*}
The latter equality follows from the inclusion $\textrm{Im}\rho_{i}\cap IG_{m}\subseteq\textrm{Im}\rho_{i}$. 
\end{proof}

\subsection*{Computing $C_{i}$}

We turn now to the computation of $C_{i}$. We are going to show that
$C_{i}$ are canonically isomorphic to
\[
\ker(\widehat{SL_{n-1}(\mathbb{Z}[x^{\pm1}])}\to SL_{n-1}(\widehat{\mathbb{Z}[x^{\pm1}]}))
\]
and going to use it in order to show that $C\left(IA\left(\Phi\right),\Phi\right)$
is not finitely generated. So fix $n\geq4$, $1\leq i_{0}\leq n$,
and denote:
\begin{itemize}
\item $x=x_{i_{0}}$.
\item $\sigma=\sigma_{i_{0}}=x_{i_{0}}-1$.
\item $IGL_{n-1}=IGL_{n-1,i_{0}}$.
\item $IE_{n-1}(H)=IE_{n-1,i_{0}}(H)$.
\item $S=\mathbb{Z}[x^{\pm1}]=\mathbb{Z}[x_{i_{0}}^{\pm1}]$. 
\item $J_{m}=\left(x^{m}-1\right)S+mS$ for $m\in\mathbb{N}$.
\item $\rho=\rho_{i_{0}}:IA\left(\Phi\right)\twoheadrightarrow GL_{n-1}\left(S,\sigma S\right)$.
\item $\hat{\rho}=\hat{\rho}_{i_{0}}:\widehat{IA\left(\Phi\right)}\twoheadrightarrow\widehat{GL_{n-1}\left(S,\sigma S\right)}$.
\end{itemize}
Now, write (the last equality is by Proposition \ref{prop:intersection})
\begin{eqnarray*}
C_{i_{0}} & = & \ker(\textrm{Im}\hat{\rho}\to Aut(\hat{\Phi}))\\
 & = & \ker(\widehat{GL_{n-1}\left(S,\sigma S\right)}\to Aut(\hat{\Phi}))\\
 & = & \ker(\widehat{GL_{n-1}\left(S,\sigma S\right)}\to\underleftarrow{\lim}\left(IA\left(\Phi\right)/IG_{m}\right))\\
 & = & \ker(\widehat{GL_{n-1}\left(S,\sigma S\right)}\to\underleftarrow{\lim}\left(GL_{n-1}\left(S,\sigma S\right)\cdot IG_{m}/IG_{m}\right))\\
 & = & \ker(\widehat{GL_{n-1}\left(S,\sigma S\right)}\to\underleftarrow{\lim}GL_{n-1}\left(S,\sigma S\right)/\left(GL_{n-1}\left(S,\sigma S\right)\cap IG_{m}\right))\\
 & = & \ker(\widehat{GL_{n-1}\left(S,\sigma S\right)}\to\underleftarrow{\lim}GL_{n-1}\left(S,\sigma S\right)/GL_{n-1}\left(S,\sigma J_{m}\right)).
\end{eqnarray*}
Now, by the same computation as in Proposition \ref{prop:contain}
one can show that for every $m\in\mathbb{N}$ we have $(J_{m^{2}}\cap\sigma S)\subseteq\sigma J_{m}\subseteq(J_{m}\cap\sigma S)$,
so the latter is equal to
\begin{align*}
 & \ker(\widehat{GL_{n-1}\left(S,\sigma S\right)}\to\underleftarrow{\lim}GL_{n-1}\left(S,\sigma S\right)/(GL_{n-1}\left(S,\sigma S\right)\cap GL_{n-1}(S,J_{m})))\\
 & =\ker(\widehat{GL_{n-1}\left(S,\sigma S\right)}\to\underleftarrow{\lim}(GL_{n-1}\left(S,\sigma S\right)\cdot GL_{n-1}(S,J_{m}))/GL_{n-1}(S,J_{m}))\\
 & =\ker(\widehat{GL_{n-1}\left(S,\sigma S\right)}\to\underleftarrow{\lim}GL_{n-1}\left(S\right)/GL_{n-1}(S,J_{m})\\
 & =\ker(\widehat{GL_{n-1}\left(S,\sigma S\right)}\to\underleftarrow{\lim}GL_{n-1}(S/J_{m})).
\end{align*}

Now, if $\bar{S}$ is a finite quotient of $S$, then as $x$ is invertible
in $S$, its image $\bar{x}\in\bar{S}$ is invertible in $\bar{S}$.
Thus, there exists $r\in\mathbb{N}$ such that $\bar{x}^{r}=1_{\bar{S}}$.
In addition, there exists $t\in\mathbb{N}$ such that $\underset{t}{\underbrace{1_{\bar{S}}+\ldots+1_{\bar{S}}}}=0_{\bar{S}}$.
Therefore, for $m=r\cdot t$ the map $S\to\bar{S}$ factorizes through
$\mathbb{Z}_{m}[\mathbb{Z}_{m}]\cong S/J_{m}$. Thus, we have $\hat{S}=\underleftarrow{\lim}(S/J_{m})$,
which implies that: $GL_{n-1}(\hat{S})=\underleftarrow{\lim}GL_{n-1}(S/J_{m})$.
Therefore

\[
C_{i_{0}}=\ker(\widehat{GL_{n-1}\left(S,\sigma S\right)}\to GL_{n-1}(\hat{S})).
\]
Now, the short exact sequence
\[
1\to GL_{n-1}\left(S,\sigma S\right)\to GL_{n-1}\left(S\right)\to GL_{n-1}\left(\mathbb{Z}\right)\to1
\]
gives rise to the exact sequence (see \cite{key-5}, Lemma 2.1)
\[
\widehat{GL_{n-1}\left(S,\sigma S\right)}\to\widehat{GL_{n-1}\left(S\right)}\to\widehat{GL_{n-1}\left(\mathbb{Z}\right)}\to1
\]
which gives rise to the commutative diagram
\[
\begin{array}{cccccccc}
\widehat{GL_{n-1}\left(S,\sigma S\right)} & \to & \widehat{GL_{n-1}\left(S\right)} & \to & \widehat{GL_{n-1}\left(\mathbb{Z}\right)} & \to & 1\\
 & \searrow & \downarrow &  & \downarrow\\
 &  & GL_{n-1}(\hat{S}) & \to & GL_{n-1}(\mathbb{\hat{Z}}) & \to & 1 & .
\end{array}
\]
Assuming $n\geq4$ and using the affirmative answer to the classical
congruence subgroup problem (\cite{key-22}, \cite{key-23}), the
map: $\widehat{GL_{n-1}\left(\mathbb{Z}\right)}\to GL_{n-1}(\mathbb{\hat{Z}})$
is injective. Thus, by diagram chasing we obtain that the kernel $\ker(\widehat{GL_{n-1}\left(S,\sigma S\right)}\to GL_{n-1}(\hat{S}))$
is mapped onto $\ker(\widehat{GL_{n-1}\left(S\right)}\to GL_{n-1}(\hat{S}))$.
In order to proceed from here we need the following lemma:
\begin{lem}
\label{lem:pf-comp}Let $d\geq3$ and denote: $D_{m}=\left\{ I_{d}+\left(x^{k\cdot m}-1\right)E_{1,1}\,|\,k\in\mathbb{Z}\right\} $
for $m\in\mathbb{N}$. Then
\begin{eqnarray*}
\widehat{GL_{d}\left(S\right)} & = & \underleftarrow{\lim}\left(GL_{d}\left(S\right)/\left(D_{m}E_{d}\left(S,J_{m}\right)\right)\right)\\
\widehat{SL_{d}\left(S\right)} & = & \underleftarrow{\lim}\left(SL_{d}\left(S\right)/E_{d}\left(S,J_{m}\right)\right).
\end{eqnarray*}
\end{lem}

\begin{proof}
We will prove the first part and the second is similar but easier.
We first claim that $D_{m}E_{d}\left(S,J_{m}\right)$ is a finite
index normal subgroup of $GL_{d}\left(S\right)$. Indeed, by a well-known
result of Suslin \cite{key-33}, $SL_{d}\left(S\right)=E_{d}\left(S\right)$.
Thus, by Corollary \ref{cor:important-cor}, $SK_{1}(S,J_{m};d)=SL_{d}\left(S,J_{m}\right)/E_{d}\left(S,J_{m}\right)$
is finite. As the subgroup $SL_{d}\left(S,J_{m}\right)$ is of finite
index in $SL_{d}\left(S\right)$, so is $E_{d}\left(S,J_{m}\right)$.
Now, it is not difficult to see that the group of invertible elements
of $S$ is equal to $S^{*}=\left\{ \pm x^{k}\,|\,k\in\mathbb{Z}\right\} $
(see \cite{key-38}, chapter 8). So as $\left\{ x^{k\cdot m}\,|\,k\in\mathbb{Z}\right\} $
is of finite index in $S^{*}$, the subgroup $D_{m}SL_{d}\left(S\right)$
is of finite index in $GL_{d}\left(S\right)$. We deduce that also
$D_{m}E_{d}\left(S,J_{m}\right)$ is of finite index in $GL_{d}\left(S\right)$.
It remains to show that $D_{m}E_{d}\left(S,J_{m}\right)$ is normal
in $GL_{d}\left(S\right)$. 

We already stated previously (see $\varoint$\ref{sec:K-theory})
that $E_{d}\left(S,J_{m}\right)$ is normal in $GL_{d}\left(S\right)$.
Thus, noticing the group identity
\[
gheg^{-1}=h(h^{-1}ghg^{-1})(geg^{-1})
\]
it is enough to show that the commutators of the elements of $D_{m}$
with any set of generators of $GL_{d}\left(S\right)$, are in $E_{d}(S,J_{m})$.
By the aforementioned result of Suslin and as $S^{*}=\left\{ \pm x^{r}\,|\,r\in\mathbb{Z}\right\} $,
the group $GL_{d}\left(S\right)$ is generated by the elements of
the forms
\begin{eqnarray*}
1. & I_{d}+\left(\pm x-1\right)E_{1,1}\\
2. & I_{d}+rE_{i,j} & r\in S,\,2\leq i\neq j\leq d\\
3. & I_{d}+rE_{1,j} & r\in S,\,2\leq j\leq d\\
4. & I_{d}+rE_{i,1} & r\in S,\,2\leq i\leq d.
\end{eqnarray*}
Now, obviously, the elements of $D_{m}$ commute with the elements
of the forms 1 and 2. In addition, for the elements of the forms 3
and 4, one can easily compute that
\begin{eqnarray*}
\left[I_{d}+\left(x^{k\cdot m}-1\right)E_{1,1},I_{d}+rE_{1,j}\right] & = & I_{d}+r\left(x^{k\cdot m}-1\right)E_{1,j}\in E_{d}\left(S,J_{m}\right)\\
\left[I_{d}+\left(x^{k\cdot m}-1\right)E_{1,1},I_{d}+rE_{i,1}\right] & = & I_{d}+r\left(x^{-k\cdot m}-1\right)E_{i,1}\in E_{d}\left(S,J_{m}\right)
\end{eqnarray*}
for every $2\leq i,j\leq d$, as required.

Now, clearly, every finite index normal subgroup of $GL_{d}\left(S\right)$
contains $D_{m}$ for some $m\in\mathbb{N}$. In addition, it is not
hard to show that when $d\geq3$, every finite index normal subgroup
$N\vartriangleleft GL_{d}\left(S\right)$ contains $E_{d}(S,J)$ for
some finite index ideal $J\vartriangleleft S$ (see \cite{key-24-1},
Section 1). Thus, as we saw previously that every finite index ideal
$J\vartriangleleft S_{n}$ contains $J_{m}$ for some $m$, we obtain
that $\widehat{GL_{d}\left(S\right)}=\underleftarrow{\lim}\left(GL_{d}\left(S\right)/\left(D_{m}E_{d}\left(S,J_{m}\right)\right)\right)$,
as required.
\end{proof}
In order to prove the following proposition, we are going to use Lemma
\ref{lem:main}, that its proof is left to the last section of the
paper.
\begin{prop}
\label{prop:injective}Let $n\geq4$. Then, the map $\widehat{GL_{n-1}\left(S,\sigma S\right)}\to\widehat{GL_{n-1}\left(S\right)}$
is injective. Hence, the surjective map
\[
C_{i_{0}}=\ker(\widehat{GL_{n-1}\left(S,\sigma S\right)}\to GL_{n-1}(\hat{S}))\twoheadrightarrow\ker(\widehat{GL_{n-1}\left(S\right)}\to GL_{n-1}(\hat{S}))
\]
is an isomorphism.
\end{prop}

\begin{proof}
We showed in the previous lemma that 
\[
\widehat{GL_{n-1}\left(S\right)}=\underleftarrow{\lim}GL_{n-1}\left(S\right)/(D_{m}E_{n-1}(S,J_{m}))
\]
where: $D_{m}=\left\{ I_{n-1}+\left(x^{k\cdot m}-1\right)E_{1,1}\,|\,k\in\mathbb{Z}\right\} $
and $J_{m}=\left(x^{m}-1\right)S+mS$. Hence, the image of $\widehat{GL_{n-1}\left(S,\sigma S\right)}$
in $\widehat{GL_{n-1}\left(S\right)}$ is
\begin{align*}
 & \underleftarrow{\lim}(GL_{n-1}\left(S,\sigma S\right)\cdot D_{m}E_{n-1}(S,J_{m}))/(D_{m}E_{n-1}(S,J_{m}))\\
 & =\underleftarrow{\lim}GL_{n-1}\left(S,\sigma S\right)/(GL_{n-1}\left(S,\sigma S\right)\cap D_{m}E_{n-1}(S,J_{m})).
\end{align*}
Using that $D_{m}\subseteq GL_{n-1}\left(S,\sigma S\right)$, one
can see that the latter equals to
\[
\underleftarrow{\lim}GL_{n-1}\left(S,\sigma S\right)/(D_{m}(GL_{n-1}\left(S,\sigma S\right)\cap E_{n-1}(S,J_{m}))).
\]
Recall now the following notations:
\begin{itemize}
\item $R_{n}=\mathbb{Z}[x_{1}^{\pm1},\ldots,x_{n}^{\pm1}]$.
\item $H_{m}=\sum_{i=1}^{n}(x_{i}^{m}-1)R_{n}+mR_{n}\vartriangleleft R_{n}$.
\item $IE_{n-1}(H_{m})=IGL_{n-1}\cap E{}_{n-1}(R_{n},H_{m})$ under the
identification of $IGL_{n-1}$ with $GL_{n-1}(R_{n},\sigma R_{n})$.
\end{itemize}
Then, following the definition of the map $\rho:IA\left(\Phi\right)\twoheadrightarrow GL_{n-1}\left(S,\sigma S\right)$
we have 
\begin{eqnarray*}
\left\langle IA\left(\Phi\right)^{m}\right\rangle  & \overset{\rho}{\twoheadrightarrow} & \left\langle GL_{n-1}\left(S,\sigma S\right)^{m}\right\rangle \\
IE_{n-1}(H_{m}) & \overset{\rho}{\twoheadrightarrow} & GL_{n-1}\left(S,\sigma S\right)\cap E_{n-1}(S,J_{m}).
\end{eqnarray*}
So as by the main Lemma (Lemma \ref{lem:main}) we have $IE_{n-1}(H_{m^{2}})\subseteq\left\langle IA\left(\Phi\right)^{m}\right\rangle $,
we have also
\[
GL_{n-1}\left(S,\sigma S\right)\cap E_{n-1}(S,J_{m^{2}})\subseteq\left\langle GL_{n-1}\left(S,\sigma S\right)^{m}\right\rangle .
\]
As obviously $D_{m^{2}}\subseteq\left\langle GL_{n-1}\left(S,\sigma S\right)^{m}\right\rangle $,
we deduce the following natural surjective maps
\begin{eqnarray*}
 &  & \underleftarrow{\lim}GL_{n-1}\left(S,\sigma S\right)/(D_{m}(GL_{n-1}\left(S,\sigma S\right)\cap E_{n-1}(S,J_{m})))\\
 &  & =\underleftarrow{\lim}GL_{n-1}\left(S,\sigma S\right)/(D_{m^{2}}(GL_{n-1}\left(S,\sigma S\right)\cap E_{n-1}(S,J_{m^{2}})))\\
 &  & \twoheadrightarrow\underleftarrow{\lim}GL_{n-1}\left(S,\sigma S\right)/\left\langle GL_{n-1}\left(S,\sigma S\right)^{m}\right\rangle \\
 &  & \twoheadrightarrow\widehat{GL_{n-1}\left(S,\sigma S\right)}\\
 &  & \twoheadrightarrow\underleftarrow{\lim}GL_{n-1}\left(S,\sigma S\right)/(D_{m}(GL_{n-1}\left(S,\sigma S\right)\cap E_{n-1}(S,J_{m})))
\end{eqnarray*}
such that the composition gives the identity map. Hence, these maps
are also injective, and in particular, the map
\[
\widehat{GL_{n-1}\left(S,\sigma S\right)}\twoheadrightarrow\underleftarrow{\lim}GL_{n-1}\left(S,\sigma S\right)/(D_{m}(GL_{n-1}\left(S,\sigma S\right)\cap E_{n-1}(S,J_{m})))
\]
is injective, as required.
\end{proof}
\begin{prop}
\label{prop:S-G}Let $d\geq3$. Then, the natural embedding $SL_{d}\left(S\right)\leq GL_{d}\left(S\right)$
induces a natural isomorphism
\[
\ker(\widehat{GL_{d}\left(S\right)}\to GL_{d}(\hat{S}))\cong\ker(\widehat{SL_{d}\left(S\right)}\to SL_{d}(\hat{S})).
\]
\end{prop}

\begin{proof}
By Lemma \ref{lem:pf-comp} we have
\begin{eqnarray*}
 &  & \ker(\widehat{GL_{d}\left(S\right)}\to GL_{d}(\hat{S})=\underleftarrow{\lim}GL_{d}\left(S/J_{m}\right))\\
 &  & =\ker(\underleftarrow{\lim}GL_{d}\left(S\right)/D_{m}E_{d}\left(S,J_{m}\right)\to\underleftarrow{\lim}GL_{d}\left(S\right)/GL_{d}\left(S,J_{m}\right))\\
 &  & =\underleftarrow{\lim}GL_{d}\left(S,J_{m}\right)/D_{m}E_{d}\left(S,J_{m}\right)
\end{eqnarray*}
where $D_{m}=\left\{ I_{n-1}+\left(x^{k\cdot m}-1\right)E_{1,1}\,|\,k\in\mathbb{Z}\right\} $.
We claim now that when $m>2$ then $GL_{d}\left(S,J_{m}\right)=D_{m}SL_{d}\left(S,J_{m}\right)$.
Indeed, for every $A\in GL_{d}\left(S,J_{m}\right)$ we have $\det\left(A\right)=\pm x^{k}$
for some $k\in\mathbb{Z}$. However, as under the map $S\to\mathbb{Z}_{m}[\mathbb{Z}_{m}]$
we have $A\mapsto I_{d}$, the map $S\to\mathbb{Z}_{m}[\mathbb{Z}_{m}]$
also implies $\det\left(A\right)\mapsto1$. Hence $\det\left(A\right)=\pm x^{k\cdot m}$
for some $k\in\mathbb{Z}$, and when $m>2$ we even get $\det\left(A\right)=x^{k\cdot m}$
for some $k\in\mathbb{Z}$. It follows that $GL_{d}\left(S,J_{m}\right)=D_{m}SL_{d}\left(S,J_{m}\right)$.
Therefore, since $D_{m}\cap SL_{d}\left(S,J_{m}\right)=\left\{ I_{d}\right\} $,
we deduce that
\begin{eqnarray*}
\ker(\widehat{GL_{d}\left(S\right)}\to GL_{d}(\hat{S})) & = & \underleftarrow{\lim}D_{m}SL_{d}\left(S,J_{m}\right)/D_{m}E_{d}\left(S,J_{m}\right)\\
 & = & \underleftarrow{\lim}SL_{d}\left(S,J_{m}\right)/E_{d}\left(S,J_{m}\right)\\
 & = & \underleftarrow{\lim}\ker(\widehat{SL_{d}\left(S\right)}\to SL_{d}(\hat{S})).
\end{eqnarray*}
\end{proof}
The immediate corollary from Propositions \ref{prop:injective} and
\ref{prop:S-G} is:
\begin{cor}
For every $n\geq4$, we have $C_{i_{0}}\cong\ker(\widehat{SL_{n-1}\left(S\right)}\to SL_{n-1}(\hat{S}))$.
\end{cor}

We close the section by showing that $\ker(\widehat{SL_{n-1}\left(S\right)}\to SL_{n-1}(\hat{S}))$
is not finitely generated, using the techniques in \cite{key-24-1}.
It is known that the group ring $S=\mathbb{Z}[x^{\pm1}]=\mathbb{Z}\left[\mathbb{Z}\right]$
is Noetherian (see \cite{key-30}, \cite{key-31}). In addition, it
is known that the Krull dimension of $\mathbb{Z}$ is: $\dim\left(\mathbb{Z}\right)=1$
and thus $\dim\left(S\right)=\dim\left(\mathbb{\mathbb{Z}\left[\mathbb{Z}\right]}\right)=2$
(see \cite{key-25}). Therefore, by Proposition 1.6 in \cite{key-33},
as $n-1\geq3$, for every $J\vartriangleleft S$, the canonical map
\[
SK_{1}\left(S,J;n-1\right)\to SK_{1}\left(S,J\right):=\underset{d\in\mathbb{N}}{\underrightarrow{\lim}}SK_{1}\left(S,J;d\right)
\]
is surjective. Hence, the canonical map (when $J\vartriangleleft S$
ranges over all finite index ideals of $S$)
\begin{eqnarray*}
\ker(\widehat{SL_{n-1}\left(S\right)}\to SL_{n-1}(\hat{S})) & = & \underleftarrow{\lim}\left(SL_{n-1}\left(S,J\right)/E_{n-1}\left(S,J\right)\right)\\
 & = & \underleftarrow{\lim}SK_{1}\left(S,J;n-1\right)\\
 & \to & \underleftarrow{\lim}SK_{1}\left(S,J\right)
\end{eqnarray*}
is surjective, so it is enough to show that $\underleftarrow{\lim}SK_{1}\left(S,J\right)$
is not finitely generated. By a result of Bass (see \cite{key-31-1},
chapter 5, Corollary 9.3), for every $J\vartriangleleft K\vartriangleleft S$
of finite index in $S$, the map: $SK_{1}\left(S,J\right)\to SK_{1}\left(S,K\right)$
is surjective. Hence, it is enough to show that for every $l\in\mathbb{N}$
there exists a finite index ideal $J\vartriangleleft S$ such that
$SK_{1}\left(S,J\right)$ is generated by at least $l$ elements.
Now, as $SK_{1}\left(S\right)=1$ \cite{key-33}, we obtain the following
exact sequence for every $J\vartriangleleft S$ (see Theorem 6.2 in
\cite{key-32}) 
\[
K_{2}\left(S\right)\to K_{2}\left(S/J\right)\to SK_{1}\left(S,J\right)\to SK_{1}\left(S\right)=1.
\]
In addition, by a classical result of Quillen (\cite{key-39}, \cite{key-31-2}
Theorem 5.3.30), we have
\[
K_{2}\left(S\right)=K_{2}\left(\mathbb{Z}[x^{\pm1}]\right)=K_{2}\left(\mathbb{Z}\right)\oplus K_{1}\left(\mathbb{Z}\right)
\]
so by the classical facts $K_{2}\left(\mathbb{Z}\right)=K_{1}\left(\mathbb{Z}\right)=\left\{ \pm1\right\} $
(see \cite{key-32} chapters 3 and 10) we deduce that $K_{2}\left(S\right)$
is of order $4$. Hence, it is enough to prove that for every $l\in\mathbb{N}$
there exists a finite index ideal $J\vartriangleleft S$ such that
$K_{2}\left(S/J\right)$ is generated by at least $l$ elements. Following
\cite{key-24-1}, we state the following proposition (which holds
by the proof of Theorem 2.8 in \cite{key-35}):
\begin{prop}
Let $p$ be a prime, $l\in\mathbb{N}$ and denote by $P\vartriangleleft\mathbb{Z}\left[y\right]$
the ideal which is generated by $p^{2}$ and $y^{p^{l}}$. Then, for
$\bar{S}=\mathbb{Z}\left[y\right]/P$, the group $K_{2}\left(\bar{S}\right)$
is an elementary abelian $p$-group of rank $\geq l$.
\end{prop}

Observe now that for every $l\geq0$
\[
\left(y+1\right)^{p^{l+1}}=(y^{p^{l}}+1+p\cdot a\left(y\right))^{p}=1\,\,\,\,\textrm{mod}\,\,P
\]
so $y+1$ is invertible in $\bar{S}$. Therefore we have a well defined
sujective homomorphism $S\to\bar{S}$ which is defined by sending
$x\mapsto y+1$. In particular, $J=\ker\left(S\to\bar{S}\right)$
is a finite index ideal of $S$ which satisfies the above requirements.
This shows that indeed $C_{i_{0}}=\ker(\widehat{SL_{n-1}\left(S\right)}\to SL_{n-1}(\hat{S}))$
is not finitely generated, and by the description in Proposition \ref{cor:C semi-direct}
it follows that $C\left(IA\left(\Phi\right),\Phi\right)$ is not finitely
generated either. 

\section{\label{sec:Finishing}The centrality of $C_{i}$}

In this section we will prove that for every $n\geq4$, the copies
$C_{i}$ lie in the center of $\widehat{IA\left(\Phi\right)}$. Along
the section we will assume that $n\geq4$ is constant, and will show
it for $i=n$. Symmetrically, it will be valid for every $i$. We
recall:
\begin{itemize}
\item $R_{n}=\mathbb{Z}[x_{1}^{\pm1},\ldots,x_{n}^{\pm1}]$.
\item $H_{m}=\sum_{i=1}^{n}\left(x_{i}^{m}-1\right)R_{n}+mR_{n}$.
\item $IG_{m}=\left\{ I_{n}+A\in GL_{n}\left(R_{n},H_{m}\right)\,|\,A\vec{\sigma}=\vec{0}\right\} $.
\item $IA_{m}=\cap\left\{ N\vartriangleleft IA\left(\Phi\right)\,|\,[IA\left(\Phi\right):N]\,|\,m\right\} $.
\item $S=S_{n}=\mathbb{Z}[x_{n}^{\pm1}]$.
\item $\textrm{Im}\rho\cap IG_{m}=\textrm{Im}\rho_{n}\cap IG_{m}\simeq GL_{n-1}\left(S,\sigma_{n}H_{m}\cap S)\right)$
(see Proposition \ref{prop:intersection}).
\end{itemize}
We saw in Section \ref{sec: not f.g.} that we can write
\begin{eqnarray*}
C_{n} & = & \underleftarrow{\lim}\left(IA_{m}\cdot\left(\textrm{Im}\rho\cap IG_{m}\right)/IA_{m}\right)\\
 & = & \underleftarrow{\lim}\left(IA_{m}\cdot\left(\textrm{Im}\rho\cap IG_{m^{4}}\right)/IA_{m}\right)\\
 & \leq & \underleftarrow{\lim}\left(IA\left(\Phi\right)/IA_{m}\right)=\widehat{IA\left(\Phi\right)}.
\end{eqnarray*}
Hence, if we want to show that $C_{n}$ lies in the center of $\widehat{IA\left(\Phi\right)}$,
it suffices to show that for every $m\in\mathbb{N},$ the group $IA_{m}\cdot\left(\textrm{Im}\rho\cap IG_{m^{4}}\right)/IA_{m}$
lies in the center of $IA\left(\Phi\right)/IA_{m}$. 

We first claim that under the isomorphism $\textrm{Im}\rho\cap IG_{m^{4}}\simeq GL_{n-1}\left(S,\sigma_{n}H_{m^{4}}\cap S)\right)$
one has
\begin{equation}
IA_{m}\cdot\left(\textrm{Im}\rho\cap IG_{m^{4}}\right)/IA_{m}\subseteq IA_{m}\cdot SL_{n-1}\left(S,\sigma_{n}H_{m^{2}}\cap S)\right)/IA_{m}.\label{eq:IG=00003DSL}
\end{equation}
Indeed, if $\alpha\in\textrm{Im}\rho\cap IG_{m^{4}}$ then $\det(\alpha)\in1+\sigma_{n}H_{m^{4}}\cap S\subseteq1+H_{m^{4}}\cap S$.
Combining it with Proposition \ref{lem:det}, $\det(\alpha)$ has
the form $\det(\alpha)=x_{n}^{m^{4}t}$ for some $t\in\mathbb{Z}$.
Hence
\[
\det((I_{n}+\sigma_{n}E_{1,1}-\sigma_{1}E_{1,n})^{-m^{4}t}\cdot\alpha)=1.
\]
Now, as we have (see the computation in the proof of Proposition \ref{prop:contain})
\begin{align*}
x_{n}^{m^{4}t} & =1+(x_{n}^{m^{4}t}-1)=1+\sigma_{n}\sum_{i=1}^{m^{4}-1}(x_{n}^{t})^{i}\\
 & =1+\sigma_{n}((x_{n}^{m^{2}t}-1)S+m^{2}S)\in1+\sigma_{n}H_{m^{2}}\cap S
\end{align*}
we obtain that
\begin{align*}
(I_{n}+\sigma_{n}E_{1,1}-\sigma_{1}E_{1,n})^{m^{4}t} & \in\left\langle IA(\Phi)^{m}\right\rangle \cap GL_{n-1}\left(S,\sigma_{n}H_{m^{2}}\cap S)\right)\\
 & \subseteq IA_{m}\cap GL_{n-1}\left(S,\sigma_{n}H_{m^{2}}\cap S)\right).
\end{align*}
Therefore, writing $\alpha=(I_{n}+\sigma_{n}E_{1,1}-\sigma_{1}E_{1,n})^{m^{4}t}\cdot((I_{n}+\sigma_{n}E_{1,1}-\sigma_{1}E_{1,n})^{-m^{4}t}\cdot\alpha)$
we deduce that 
\[
\textrm{Im}\rho\cap IG_{m^{4}}\subseteq IA_{m}\cdot SL_{n-1}\left(S,\sigma_{n}H_{m^{2}}\cap S)\right)
\]
and we get Inclusion (\ref{eq:IG=00003DSL}). It follows that if we
want to show that $C_{n}$ lies in the center of $\widehat{IA\left(\Phi\right)}$
it suffices to show that $IA_{m}\cdot SL_{n-1}\left(S,\sigma_{n}H_{m^{2}}\cap S)\right)/IA_{m}$
lies in the center of $IA\left(\Phi\right)/IA_{m}$. However, we are
going to show even more. We will show that:
\begin{prop}
For every $m\in\mathbb{N}$, the group
\[
IA_{m}\cdot ISL_{n-1,n}\left(\sigma_{n}H_{m^{2}}\right)/IA_{m}
\]
lies in the center of $IA\left(\Phi\right)/IA_{m}$. 
\end{prop}

Let $F$ be the free group on $f_{1},\ldots,f_{n}$. It is a classical
result by Magnus (\cite{key-1}, Chapter 3, Theorem N4) that $IA\left(F\right)$
is generated by the automorphisms of the form
\[
\alpha_{r,s,t}=\begin{cases}
f_{r}\mapsto\left[f_{t},f_{s}\right]f_{r}\\
f_{u}\mapsto f_{u} & u\neq r
\end{cases}
\]
where $\left[f_{t},f_{s}\right]=f_{t}f_{s}f_{t}^{-1}f_{s}^{-1}$ and
$1\leq r,s\neq t\leq n$ (notice that we may have $r=s$). In their
paper \cite{key-24}, Bachmuth and Mochizuki show that when $n\geq4$,
the group $IA\left(\Phi\right)$ is generated by the images of these
generators under the natural map $Aut\left(F\right)\to Aut\left(\Phi\right)$.
I.e. $IA\left(\Phi\right)$ is generated by the elements of the form
\[
E_{r,s,t}=I_{n}+\sigma_{t}E_{r,s}-\sigma_{s}E_{r,t}\,\,\,\,1\leq r,s\neq t\leq n.
\]
Therefore, for showing the centrality of $C_{n}$, it is enough to
show that given:
\begin{itemize}
\item an element: $\bar{\lambda}\in IA_{m}\cdot ISL_{n-1,n}\left(\sigma_{n}H_{m^{2}}\right)/IA_{m}$,
\item and one of generators: $E_{r,s,t}=I_{n}+\sigma_{t}E_{r,s}-\sigma_{s}E_{r,t}$
for $1\leq r,s\neq t\leq n$,
\end{itemize}
there exists $\lambda\in ISL_{n-1,n}\left(\sigma_{n}H_{m^{2}}\right)$,
a representative of $\bar{\lambda}$, such that $\left[E_{r,s,t},\lambda\right]\in IA_{m}$.
So assume that we have an element: $\bar{\lambda}\in IA_{m}\cdot ISL_{n-1,n}\left(\sigma_{n}H_{m^{2}}\right)/IA_{m}$.
Then, a representative for $\bar{\lambda}$ has the form
\[
\lambda=\left(\begin{array}{cc}
I_{n-1}+\sigma_{n}B & -\sum_{i=1}^{n-1}\sigma_{i}\vec{b}_{i}\\
0 & 1
\end{array}\right)\in ISL_{n-1,n}\left(\sigma_{n}H_{m^{2}}\right)
\]
for some $\left(n-1\right)\times\left(n-1\right)$ matrix $B$ which
its entries $b_{i,j}$ admit $b_{i,j}\in H_{m^{2}}$ and its column
vectors denoted by $\vec{b}_{i}$. 
\begin{lem}
\label{prop:rep}Let $\bar{\lambda}\in IA_{m}\cdot ISL_{n-1,n}\left(\sigma_{n}H_{m^{2}}\right)/IA_{m}$.
Then, for every $1\leq l<k\leq n-1$, $\bar{\lambda}$ has a representative
in $ISL_{n-1,n}\left(\sigma_{n}H_{m^{2}}\right)$, of the following
form (the following notation means that the matrix is similar to the
identity matrix, except the entries in the $l$-th and $k$-th rows)
\begin{equation}
\underset{\,\,\,\,\,\,\,\,\,\,\,\,\,\,\,\begin{array}{c}
\uparrow\\
l\textrm{-th\,\,\,\,\ column}
\end{array}\,\,\,\,\,\,\,\,\,\,\begin{array}{c}
\uparrow\\
k\textrm{-th\,\,\,\,\ column}
\end{array}\,\,\,\,\,\,\,\,\,\,\,\,\,\,\,\,\,\begin{array}{c}
\uparrow\\
n\textrm{-th\,\,\,\,\ column}
\end{array}}{\left(\begin{array}{cccccc}
I_{l-1} & 0 & 0 & 0 & 0 & 0\\
0 & 1+\sigma_{n}a & 0 & \sigma_{n}b & 0 & -\sigma_{l}a-\sigma_{k}b\\
0 & 0 & I_{k-l-1} & 0 & 0 & 0\\
0 & \sigma_{n}c & 0 & 1+\sigma_{n}d & 0 & -\sigma_{l}c-\sigma_{k}d\\
0 & 0 & 0 & 0 & I_{n-k-1} & 0\\
0 & 0 & 0 & 0 & 0 & 1
\end{array}\right)}\begin{array}{c}
\\
\leftarrow l\textrm{-th\,\,\,\,\ row}\\
\\
\leftarrow k\textrm{-th\,\,\,\,\ row}\\
\\
\\
\end{array}\label{eq:form  1}
\end{equation}
for some $a,b,c,d\in H_{m^{2}}$.
\end{lem}

\begin{proof}
We will demonstrate the proof in the case $l=1$, $k=2$, and symmetrically,
the arguments hold for arbitrary $1\leq l<k\leq n-1$. Consider an
arbitrary representative of $\bar{\lambda}$
\[
\lambda=\left(\begin{array}{cc}
I_{n-1}+\sigma_{n}B & -\sum_{i=1}^{n-1}\sigma_{i}\vec{b}_{i}\\
0 & 1
\end{array}\right)\in ISL_{n-1,n}\left(\sigma_{n}H_{m^{2}}\right).
\]
Then $I_{n-1}+\sigma_{n}B\in SL_{n-1}\left(R_{n},\sigma_{n}H_{m^{2}}\right)$.
Consider now the ideal
\[
R_{n}\vartriangleright H'_{m^{2}}=\sum_{r=1}^{n-1}(x_{r}^{m^{2}}-1)R_{n}+\sigma_{n}(x_{n}^{m^{2}}-1)R_{n}+m^{2}R_{n}.
\]
Observe that $\sigma_{n}H_{m^{2}}\vartriangleleft H'_{m^{2}}\vartriangleleft H{}_{m^{2}}\vartriangleleft R_{n}$
and that $H'_{m^{2}}\cap\sigma_{n}R_{n}=\sigma_{n}H_{m^{2}}$. In
addition, by similar computations as in the proof of proposition \ref{prop:contain},
for every $x\in R_{n}$ we have $x^{m^{4}}-1\in\left(x-1\right)(x^{m^{2}}-1)R_{n}+\left(x-1\right)m^{2}R_{n}$,
and thus $H_{m^{4}}\subseteq H'_{m^{2}}$, so $H'_{m^{2}}$ is of
finite index in $R_{n}$. 

Now, $I_{n-1}+\sigma_{n}B\in SL_{n-1}\left(R_{n},\sigma_{n}H_{m^{2}}\right)\subseteq SL_{n-1}\left(R_{n},H'_{m^{2}}\right)$.
Thus, by the third part of Corollary \ref{cor:important-cor}, as
$H'_{m^{2}}\vartriangleleft R_{n}$ is an ideal of finite index, $n-1\geq3$
and $E_{n-1}\left(R_{n}\right)=SL_{n-1}\left(R_{n}\right)$ \cite{key-33},
one can write the matrix $I_{n-1}+\sigma_{n}B$ as
\[
I_{n-1}+\sigma_{n}B=AD\,\,\,\,\textrm{when}\,\,\,\,A=\left(\begin{array}{cc}
A' & 0\\
0 & I_{n-3}
\end{array}\right)
\]
for some $A'\in SL_{2}\left(R_{n},H'_{m^{2}}\right)$ and $D\in E{}_{n-1}\left(R_{n},H'_{m^{2}}\right)$.
Now, consider the images of $D$ and $A$ under the projection $\sigma_{n}\to0$,
which we denote by $\bar{D}$ and $\bar{A},$ respectively. Observe
that obviously, $\bar{D}\in E{}_{n-1}\left(R_{n},H'_{m^{2}}\right)$.
In addition, observe that
\[
AD\in GL_{n-1}\left(R_{n},\sigma_{n}R_{n}\right)\,\,\Rightarrow\,\,\bar{A}\bar{D}=I_{n-1}.
\]
Thus, we have $I_{n-1}+\sigma_{n}B=A\bar{A}^{-1}\bar{D}^{-1}D$. Therefore,
by replacing $D$ by $\bar{D}^{-1}D$ and $A$ by $A\bar{A}^{-1}$
we can assume that
\[
I_{n-1}+\sigma_{n}B=AD\,\,\,\,\textrm{for}\,\,\,\,A=\left(\begin{array}{cc}
A' & 0\\
0 & I_{n-3}
\end{array}\right)
\]
where: $A'\in SL_{2}\left(R_{n},H'_{m^{2}}\right)\cap GL_{2}\left(R_{n},\sigma_{n}R_{n}\right)=SL_{2}\left(R_{n},\sigma_{n}H_{m^{2}}\right)$,
and 
\begin{eqnarray*}
D & \in & E_{n-1}\left(R_{n},H'_{m^{2}}\right)\cap GL_{n-1}\left(R_{n},\sigma_{n}R_{n}\right)\\
 & \subseteq & E_{n-1}\left(R_{n},H_{m^{2}}\right)\cap GL_{n-1}\left(R_{n},\sigma_{n}R_{n}\right):=IE_{n-1,n}\left(H_{m^{2}}\right).
\end{eqnarray*}

Now, as we prove in the main lemma (Lemma \ref{lem:main}) that $IE_{n-1,n}\left(H_{m^{2}}\right)\subseteq\left\langle IA\left(\Phi\right)^{m}\right\rangle \subseteq IA_{m}$,
this argument shows that $\lambda$ can be replaced by a representative
of the form (\ref{eq:form  1}). 
\end{proof}
We now return to our initial mission. Let $\bar{\lambda}\in IA_{m}\cdot ISL_{n-1,n}\left(\sigma_{n}H_{m^{2}}\right)/IA_{m}$,
and let $E_{r,s,t}=I_{n}+\sigma_{t}E_{r,s}-\sigma_{s}E_{r,t}$ for
$1\leq r,s\neq t\leq n$ be one of the above generators for $IA\left(\Phi\right)$.
We want to show that there exists $\lambda\in ISL_{n-1,n}\left(\sigma_{n}H_{m^{2}}\right)$,
a representative of $\bar{\lambda}$, such that $\left[E_{r,s,t},\lambda\right]\in IA_{m}$.
We separate the treatment to two cases. We note that Lemma \ref{prop:rep}
is needed only for the second case, which is a bit more delicate.

The \textbf{\uline{first case}} is: $1\leq r\leq n-1$. In this
case one can take an arbitrary representative $\lambda\in ISL_{n-1,n}\left(\sigma_{n}H_{m^{2}}\right)\cong SL{}_{n-1}\left(R_{n},\sigma_{n}H_{m^{2}}\right)$.
Considering the embedding of $IGL'_{n-1,n}$ in $GL_{n-1}\left(R_{n}\right)$,
we have: $E_{r,s,t}\in IGL'_{n-1,n}\subseteq GL{}_{n-1}\left(R_{n}\right)$
(see Definition \ref{def:tag} and Proposition \ref{prop:tag}). Thus,
as by Corollary \ref{cor:important-cor} 
\[
SK_{1}\left(R_{n},H_{m^{2}};n-1\right)=SL{}_{n-1}\left(R_{n},H_{m^{2}}\right)/E_{n-1}\left(R_{n},H_{m^{2}}\right)
\]
is central in $GL_{n-1}\left(R_{n}\right)/E_{n-1}\left(R_{n},H_{m^{2}}\right)$,
we have
\[
\left[E_{r,s,t},\lambda\right]\in\left[GL_{n-1}\left(R_{n}\right),SL{}_{n-1}\left(R_{n},\sigma_{n}H_{m^{2}}\right)\right]\subseteq E_{n-1}\left(R_{n},H_{m^{2}}\right).
\]
In addition, as $SL_{n-1}\left(R_{n},\sigma_{n}H_{m^{2}}\right)\leq GL{}_{n-1}\left(R_{n},\sigma_{n}R_{n}\right)$
and $GL_{n-1}\left(R_{n},\sigma_{n}R_{n}\right)$ is normal in $GL_{n-1}\left(R_{n}\right)$,
we have
\[
\left[E_{r,s,t},\lambda\right]\in\left[GL_{n-1}\left(R_{n}\right),GL{}_{n-1}\left(R_{n},\sigma_{n}R_{n}\right)\right]\subseteq GL{}_{n-1}\left(R_{n},\sigma_{n}R_{n}\right).
\]
Thus, we obtain from Lemma \ref{lem:main} that
\begin{eqnarray*}
\left[E_{r,s,t},\lambda\right] & \in & E_{n-1}\left(R_{n},H_{m^{2}}\right)\cap GL{}_{n-1}\left(R_{n},\sigma_{n}R_{n}\right)\\
 & = & IE_{n-1,n}\left(H_{m^{2}}\right)\subseteq\left\langle IA\left(\Phi\right)^{m}\right\rangle \subseteq IA_{m}.
\end{eqnarray*}

The \textbf{\uline{second case}} is: $r=n$. This case is a bit
more complicated than the previous one, as $E_{r,s,t}$ is not in
$IGL'_{n-1,n}$. Here, by Lemma \ref{prop:rep} one can choose $\lambda\in ISL_{n-1,n}\left(\sigma_{n}H_{m^{2}}\right)$
whose $t$-th row equals to the standard vector $\vec{e}_{t}$. As
$t\neq r=n$, we obtain thus that both $\lambda,E_{r,s,t}\in IGL'_{n-1,t}$.
Considering the embedding $IGL'_{n-1,t}\hookrightarrow GL_{n-1}\left(R_{n}\right)$,
we have $E_{r,s,t}\in GL_{n-1}\left(R_{n},\sigma_{t}R_{n}\right)$.
In addition, remember that $\lambda$ has the form
\[
\lambda=\left(\begin{array}{cc}
I_{n-1}+\sigma_{n}B & -\sum_{i=1}^{n-1}\sigma_{i}\vec{b}_{i}\\
0 & 1
\end{array}\right)
\]
for $I_{n-1}+\sigma_{n}B\in SL_{n-1}\left(R_{n},\sigma_{n}H_{m^{2}}\right)$,
so that the entries of $\vec{b}_{i}$ are in $H_{m^{2}}$. It follows
that regarding the embedding $IGL'_{n-1,t}\hookrightarrow GL_{n-1}\left(R_{n}\right)$
we have $\lambda\in SL_{n-1}\left(R_{n},H_{m^{2}}\right)$.
\begin{rem}
Note that when considering $\lambda\in IGL'_{n-1,n}\hookrightarrow GL_{n-1}\left(R_{n}\right)$,
i.e. when considering $\lambda\in GL_{n-1}\left(R_{n}\right)$ through
the embedding of $IGL'_{n-1,n}$ in $GL_{n-1}\left(R_{n}\right)$,
we have $\lambda\in GL_{n-1}\left(R_{n},\sigma_{n}H_{m^{2}}\right)\leq GL_{n-1}\left(R_{n}\right)$.
However, when we consider $\lambda\in IGL'_{n-1,t}\hookrightarrow GL_{n-1}\left(R_{n}\right)$
we do not necessarily have: $\lambda\in GL_{n-1}\left(R_{n},\sigma_{n}H_{m^{2}}\right)$,
but we still have $\lambda\in GL_{n-1}\left(R_{n},H_{m^{2}}\right)$. 
\end{rem}

Thus, by similar arguments as in the first case
\begin{eqnarray*}
\left[E_{r,s,t},\lambda\right] & \in & \left[GL_{n-1}\left(R_{n},\sigma_{t}R_{n}\right),SL_{n-1}\left(R_{n},H_{m^{2}}\right)\right]\\
 & \subseteq & E_{n-1}\left(R_{n},H_{m^{2}}\right)\cap GL{}_{n-1}\left(R_{n},\sigma_{t}R_{n}\right)\\
 & = & IE_{n-1,t}\left(H_{m^{2}}\right)\subseteq\left\langle IA\left(\Phi\right)^{m}\right\rangle \subseteq IA_{m}.
\end{eqnarray*}

This finishes the argument which shows that $C_{i}$ are central in
$\widehat{IA\left(\Phi\right)}$.
\begin{rem}
\label{rem:IA_m-IA^m}One can follow and see that completely similar
arguments gives that the group
\[
\left\langle IA\left(\Phi\right)^{m}\right\rangle \cdot ISL_{n-1,n}\left(\sigma_{n}H_{m^{2}}\right)/\left\langle IA\left(\Phi\right)^{m}\right\rangle 
\]
lies in the center of $IA\left(\Phi\right)/\left\langle IA\left(\Phi\right)^{m}\right\rangle $.
The reason is that the only property of $IA_{m}$ that we used here
was that $\left\langle IA\left(\Phi\right)^{m}\right\rangle \subseteq IA_{m}$.
This claim is used in \cite{key-6-2} to prove Theorem \ref{thm:central}.
We note that in this paper we were careful not to use the subgroups
$\left\langle IA\left(\Phi\right)^{m}\right\rangle $ directly as
we still didn't show that they are of finite index in $IA\left(\Phi\right)$,
and therefore we cannot write $\widehat{IA\left(\Phi\right)}=\underleftarrow{\lim}\left(IA\left(\Phi\right)/\left\langle IA\left(\Phi\right)^{m}\right\rangle \right)$.
However, on the way of proving Theorem \ref{thm:central}, we do show
that $\left\langle IA\left(\Phi\right)^{m}\right\rangle $ are of
finite index in $IA\left(\Phi\right)$ (provided $n\geq4$).
\end{rem}

\section{\label{sec:elementary}Some elementary elements of $\left\langle IA\left(\Phi_{n}\right)^{m}\right\rangle $}

In this section we introduce some elements in $\left\langle IA\left(\Phi_{n}\right)^{m}\right\rangle $,
which are needed for the proof of Lemma \ref{lem:main}. In \cite{key-6-2}
we introduce a list of elements in $\left\langle IA\left(\Phi_{n}\right)^{m}\right\rangle $
that contains the list below (see Propositions 4.1 and 4.2 therein).
However, we will not need here the whole list of \cite{key-6-2},
and also do not need all the notations that are used in \cite{key-6-2}.
Hence, for the convenience of the reader we include here only the
list that is needed for the proof of Lemma \ref{lem:main}, and repeat
the arguments that are related to this shorter list.
\begin{prop}
\label{prop:type 1.1}Let $n\geq4$, $1\leq u\leq n$ and $m\in\mathbb{N}$.
Denote by $\vec{e}_{i}$ the $i$-th row standard vector. Then, the
elements of $IA\left(\Phi_{n}\right)$ of the following form, lie
in $\left\langle IA\left(\Phi_{n}\right)^{m}\right\rangle $: 
\begin{equation}
\left(\begin{array}{ccccccc}
 & I_{u-1} &  & 0 &  & 0\\
a_{u,1} & \cdots & a_{u,u-1} & 1 & a_{u,u+1} & \cdots & a_{u,n}\\
 & 0 &  & 0 &  & I_{n-u}
\end{array}\right)\leftarrow u\textrm{-th\,\,\,\,\ row}\label{eq:notation}
\end{equation}
when $\left(a_{u,1},\ldots,a_{u,u-1},0,a_{u,u+1},\ldots,a_{u,n}\right)$
is a linear combination of the vectors
\begin{eqnarray*}
 & 1. & \left\{ m\left(\sigma_{i}\vec{e}_{j}-\sigma_{j}\vec{e}_{i}\right)\,|\,i,j\neq u,\,i\neq j\right\} \\
 & 2. & \left\{ (x_{k}^{m}-1)\left(\sigma_{i}\vec{e}_{j}-\sigma_{j}\vec{e}_{i}\right)\,|\,i,j,k\neq u,\,i\neq j\right\} \\
 & 3. & \left\{ \sigma_{u}(x_{u}^{m}-1)\left(\sigma_{i}\vec{e}_{j}-\sigma_{j}\vec{e}_{i}\right)\,|\,i,j\neq u,\,i\neq j\right\} 
\end{eqnarray*}
with coefficients in $R_{n}$. Notation (\ref{eq:notation}) means
that the matrix is similar to the identity matrix, except the entries
in the $u$-th row.
\end{prop}

\begin{proof}
Without loss of generality, we assume that $u=1$. Observe now that
for every $a_{i},b_{i}\in R_{n}$ for $2\leq i\leq n$ one has
\[
\left(\begin{array}{cccc}
1 & a_{2} & \cdots & a_{n}\\
0 &  & I_{n-1}
\end{array}\right)\left(\begin{array}{cccc}
1 & b_{2} & \cdots & b_{n}\\
0 &  & I_{n-1}
\end{array}\right)=\left(\begin{array}{cccc}
1 & a_{2}+b_{2} & \cdots & a_{n}+b_{n}\\
0 &  & I_{n-1}
\end{array}\right).
\]
Hence, it is enough to prove that the elements of the following forms
belong to $\left\langle IA\left(\Phi_{n}\right)^{m}\right\rangle $
(when we write $a\vec{e}_{i}$ we mean that the entry of the $i$-th
column in the first row is $a$):
\begin{eqnarray*}
1. & \left(\begin{array}{cc}
1 & mf\left(\sigma_{i}\vec{e}_{j}-\sigma_{j}\vec{e}_{i}\right)\\
0 & I_{n-1}
\end{array}\right) & i,j\neq1,\,i\neq j,\,f\in R_{n}\\
2. & \left(\begin{array}{cc}
1 & (x_{k}^{m}-1)f\left(\sigma_{i}\vec{e}_{j}-\sigma_{j}\vec{e}_{i}\right)\\
0 & I_{n-1}
\end{array}\right) & i,j,k\neq1,\,i\neq j,\,f\in R_{n}\\
3. & \left(\begin{array}{cc}
1 & \sigma_{1}(x_{1}^{m}-1)f\left(\sigma_{i}\vec{e}_{j}-\sigma_{j}\vec{e}_{i}\right)\\
0 & I_{n-1}
\end{array}\right) & i,j\neq1,\,i\neq j,\,f\in R_{n}.
\end{eqnarray*}

We start with the elements of Form 1. Here we have
\[
\left(\begin{array}{cc}
1 & mf\left(\sigma_{i}\vec{e}_{j}-\sigma_{j}\vec{e}_{i}\right)\\
0 & I_{n-1}
\end{array}\right)=\left(\begin{array}{cc}
1 & f\left(\sigma_{i}\vec{e}_{j}-\sigma_{j}\vec{e}_{i}\right)\\
0 & I_{n-1}
\end{array}\right)^{m}\in\left\langle IA\left(\Phi_{n}\right)^{m}\right\rangle .
\]

We pass to the elements of Form 2. In this case we have
\begin{align*}
\left\langle IA\left(\Phi_{n}\right)^{m}\right\rangle \ni & \left[\left(\begin{array}{cc}
1 & f\left(\sigma_{i}\vec{e}_{j}-\sigma_{j}\vec{e}_{i}\right)\\
0 & I_{n-1}
\end{array}\right)^{-1},\left(\begin{array}{cc}
x_{k} & -\sigma_{1}\vec{e}_{k}\\
0 & I_{n-1}
\end{array}\right)^{m}\right]\\
= & \left(\begin{array}{cc}
1 & (x_{k}^{m}-1)f\left(\sigma_{i}\vec{e}_{j}-\sigma_{j}\vec{e}_{i}\right)\\
0 & I_{n-1}
\end{array}\right).
\end{align*}

We finish with the elements of Form 3. The computation here is more
complicated than in the previous cases, so we will demonstrate it
for the special case: $n=4$, $i=2$, $j=3$. It is clear that symmetrically,
with similar arguments, the same holds in general when $n\geq4$ for
every $i,j\neq1,\,i\neq j$. By similar arguments as in the previous
case we get
\[
\left\langle IA\left(\Phi_{4}\right)^{m}\right\rangle \ni\left(\begin{array}{cccc}
1 & 0 & 0 & 0\\
0 & 1 & 0 & 0\\
0 & 0 & 1 & 0\\
0 & \sigma_{3}(x_{1}^{m}-1)f & -\sigma_{2}(x_{1}^{m}-1)f & 1
\end{array}\right).
\]
Therefore, we also have
\begin{align*}
\left\langle IA\left(\Phi_{4}\right)^{m}\right\rangle \ni & \left[\left(\begin{array}{cccc}
x_{4} & 0 & 0 & -\sigma_{1}\\
0 & 1 & 0 & 0\\
0 & 0 & 1 & 0\\
0 & 0 & 0 & 1
\end{array}\right),\left(\begin{array}{cccc}
1 & 0 & 0 & 0\\
0 & 1 & 0 & 0\\
0 & 0 & 1 & 0\\
0 & \sigma_{3}(x_{1}^{m}-1)f & -\sigma_{2}(x_{1}^{m}-1)f & 1
\end{array}\right)\right]\\
= & \left(\begin{array}{cccc}
1 & -\sigma_{3}\sigma_{1}(x_{1}^{m}-1)f & \sigma_{2}\sigma_{1}(x_{1}^{m}-1)f & 0\\
0 & 1 & 0 & 0\\
0 & 0 & 1 & 0\\
0 & 0 & 0 & 1
\end{array}\right).
\end{align*}
\end{proof}

\section{\label{sec:Second}A main lemma}

We recall and present some new notations that will be used in this
section:
\begin{itemize}
\item $IA^{m}=\left\langle IA\left(\Phi\right)^{m}\right\rangle $, where
$\Phi=\Phi_{n}$.
\item $R_{n}=\mathbb{Z}[\mathbb{Z}^{n}]=\mathbb{Z}[x_{1}^{\pm1},\ldots,x_{n}^{\pm1}]$
where $x_{1},\ldots,x_{n}$ are the generators of $\mathbb{Z}^{n}$.
\item $\sigma_{r}=x_{r}-1$ for $1\leq r\leq n$. 
\item $U_{r,m}=(x_{r}^{m}-1)R_{n}$ for $1\leq r\leq n$ and $m\in\mathbb{N}$. 
\item $O_{m}=mR_{n}$.
\item $H_{m}=\sum_{r=1}^{n}(x_{r}^{m}-1)R_{n}+mR_{n}=\sum_{r=1}^{n}U_{r,m}+O_{m}$.
\item $IE_{n-1,i}\left(H\right)=IGL_{n-1,i}\cap E{}_{n-1}\left(R_{n},H\right)\leq ISL_{n-1,i}\left(H\right)$
for $H\vartriangleleft R_{n}$ under the identification of $IGL_{n-1,i}\leq IA(\Phi)$
with $GL_{n-1}\left(R_{n},\sigma_{i}R_{n}\right)$ (see Proposition
\ref{prop:iso} and Definition \ref{def:IS-IE}).
\end{itemize}
In this section we prove the main lemma which asserts that:
\begin{lem}
(Main lemma) \label{lem:main}For every $n\geq4$, $m\in\mathbb{N}$
and $1\leq i\leq n$ we have
\[
IE_{n-1,i}(H_{m^{2}})\subseteq IA^{m}.
\]
\end{lem}

For simplifying the proof and the notations, we will prove the lemma
for the special case $i=n$, and symmetrically, all the arguments
are valid for every $1\leq i\leq n$. 

In addition, using the identification $IGL_{n-1,n}\cong GL_{n-1}\left(R_{n},\sigma_{n}R_{n}\right)$,
we will identify $IGL_{n-1,n}$ with $GL_{n-1}\left(R_{n},\sigma_{n}R_{n}\right)$,
and the group $IE_{n-1,n}\left(H_{m}\right)$ with $GL_{n-1}\left(R_{n},\sigma_{n}R_{n}\right)\cap E{}_{n-1}\left(R_{n},H_{m}\right)$.
So the goal of this section is proving that 
\[
GL_{n-1}\left(R_{n},\sigma_{n}R_{n}\right)\cap E_{n-1}\left(R_{n},H_{m^{2}}\right)\subseteq IA^{m}.
\]

Throughout the proof we use also elements of $IGL'_{n-1,n}$ (see
Definition \ref{def:tag}). We remind that 
\[
IE_{n-1,n}\left(H_{m}\right)\leq IGL_{n-1,n}\leq IGL'{}_{n-1,n}\hookrightarrow GL{}_{n-1}\left(R_{n}\right)
\]
(Proposition \ref{prop:tag}), so all the elememts that are being
used throughout the section are naturally embedded in $GL_{n-1}\left(R_{n}\right)$.
Using this embedding we will do all the computations in $GL_{n-1}\left(R_{n}\right)$,
and make the notations simpler by omitting the $n$-th row and column
from each matrix.

We note that many ideas in the proof of Lemma \ref{lem:main} below
are based on ideas of the proof of the ``Main Lemma'' in \cite{key-24}
(See Section 4 therein). However, our arguments do not rely directly
on the arguments in \cite{key-24}, so on the whole we cannot make
a formal reference to \cite{key-24} throughout the proof of Lemma
\ref{lem:main}.

\subsection{Decomposing the proof }

In this subsection we will start the proof of Lemma \ref{lem:main}.
At the end of the subsection, there will be a few tasks left, which
will be accomplished in the forthcoming subsections. We start with
the following definition:
\begin{defn}
For every $m\in\mathbb{N}$, define the following ideal of $R_{n}$:
\[
T_{m}=\sum_{r=1}^{n}\sigma_{r}^{2}U_{r,m}+\sum_{r=1}^{n}\sigma_{r}O_{m}+O_{m}^{2}.
\]
\end{defn}

Observe that as for every $x\in R_{n}$ we have $\sum_{j=0}^{m-1}x^{j}\in(x-1)R_{n}+mR_{n}$,
one has
\begin{eqnarray*}
x^{m^{2}}-1 & = & (x-1)\sum_{j=0}^{m^{2}-1}x^{j}=(x-1)\sum_{j=0}^{m-1}x^{j}\sum_{j=0}^{m-1}x^{jm}\\
 & \in & (x-1)\left((x-1)R_{n}+mR_{n}\right)\left((x^{m}-1)R_{n}+mR_{n}\right)\\
 & \subseteq & (x-1)^{2}(x^{m}-1)R_{n}+(x-1)^{2}mR_{n}+(x-1)m^{2}R_{n}.
\end{eqnarray*}
It follows that $H_{m^{2}}\subseteq T_{m}$. Hence, it is enough to
prove that
\[
GL_{n-1}\left(R_{n},\sigma_{n}R_{n}\right)\cap E_{n-1}\left(R_{n},T_{m}\right)\subseteq IA^{m}.
\]
Equivalently, it is enough to prove that the group
\[
\left(GL_{n-1}\left(R_{n},\sigma_{n}R_{n}\right)\cap E_{n-1}\left(R_{n},T_{m}\right)\right)\cdot IA^{m}/IA^{m}
\]
is trivial. We continue with the following proposition, which is actually
a proposition of Suslin (Corollary 1.4 in \cite{key-33}) with some
elaborations of \cite{key-24} (see the remark that follows Proposition
3.5 in \cite{key-24} and the beginning of the proof of the ``Main
Lemma'' in Section 4 therein).
\begin{prop}
\label{lem:suslin-bachmuth}Let $R$ be a commutative ring, $d\geq3$,
and $H\vartriangleleft R$ ideal. Then, $E_{d}\left(R,H\right)$ is
generated by the matrices of the form
\begin{equation}
\left(I_{d}-fE_{i,j}\right)\left(I_{d}+hE_{j,i}\right)\left(I_{d}+fE_{i,j}\right)\label{eq:form}
\end{equation}
for $h\in H$, $f\in R$ and $1\leq i\neq j\leq d$. 
\end{prop}

\begin{proof}
In the proof of Corollary 1.4 in \cite{key-33}, Suslin shows that
whenever $d\geq3$, $E_{d}\left(R,H\right)$ is generated by the elements
of the form
\[
I_{d}+h\vec{u}^{t}(u_{j}\vec{e}_{i}-u_{i}\vec{e}_{j})
\]
where $h\in H$, $i\neq j$, and $\vec{u}=(u_{1},u_{2},...,u_{d})\in R^{n}$
such that $\vec{u}\cdot\vec{v}^{t}=1$ for some $\vec{v}\in R^{n}$.
In the remark which follows Proposition 3.5 in \cite{key-24}, Bachmuth
and Mochizuki observe that 
\begin{align*}
I_{d}+h\vec{u}^{t}(u_{j}\vec{e}_{i}-u_{i}\vec{e}_{j})= & (I_{d}+h(u_{i}\vec{e}_{i}+u_{j}\vec{e}_{j})^{t}(u_{j}\vec{e}_{i}-u_{i}\vec{e}_{j}))\\
 & \cdot\prod_{l\neq i,j}(I_{d}+h(u_{l}\vec{e}_{l})^{t}u_{j}\vec{e}_{i})\cdot\prod_{l\neq i,j}(I_{d}-h(u_{l}\vec{e}_{l})^{t}u_{i}\vec{e}_{j}).
\end{align*}
Hence, by observing that all the factors in the above expression are
all of the form
\begin{equation}
I_{d}+h\left(f_{1}\vec{e}_{i}+f_{2}\vec{e}_{j}\right)^{t}\left(f_{2}\vec{e}_{i}-f_{1}\vec{e}_{j}\right)\label{eq: form Aux}
\end{equation}
for some $f_{1},f_{2}\in R$, $h\in H$ and $1\leq i\neq j\leq d$,
it is enough to show that the matrices of the form (\ref{eq: form Aux})
are generated by the matrices of the form (\ref{eq:form}). We will
show it for the case $i,j,d=1,2,3$ and it will be clear that the
general argument is similar. So we have the matrix
\[
I_{d}+h\left(f_{1}\vec{e}_{1}+f_{2}\vec{e}_{2}\right)^{t}\left(f_{2}\vec{e}_{1}-f_{1}\vec{e}_{2}\right)=\left(\begin{array}{ccc}
1+hf_{1}f_{2} & -hf_{1}^{2} & 0\\
hf_{2}^{2} & 1-hf_{1}f_{2} & 0\\
0 & 0 & 1
\end{array}\right)
\]
for some $f_{1},f_{2}\in R$ and $h\in H$, which is equal to
\[
=\left(\begin{array}{ccc}
1 & 0 & -hf_{1}\\
0 & 1 & -hf_{2}\\
f_{2} & -f_{1} & 1
\end{array}\right)\left(\begin{array}{ccc}
1 & 0 & hf_{1}\\
0 & 1 & hf_{2}\\
-f_{2} & f_{1} & 1
\end{array}\right)
\]
\[
=\left(\begin{array}{ccc}
1 & 0 & 0\\
0 & 1 & 0\\
f_{2} & -f_{1} & 1
\end{array}\right)\left(\begin{array}{ccc}
1 & 0 & -hf_{1}\\
0 & 1 & -hf_{2}\\
0 & 0 & 1
\end{array}\right)\left(\begin{array}{ccc}
1 & 0 & 0\\
0 & 1 & 0\\
-f_{2} & f_{1} & 1
\end{array}\right)\left(\begin{array}{ccc}
1 & 0 & hf_{1}\\
0 & 1 & hf_{2}\\
0 & 0 & 1
\end{array}\right).
\]
As the matrix 
\[
\left(\begin{array}{ccc}
1 & 0 & hf_{1}\\
0 & 1 & hf_{2}\\
0 & 0 & 1
\end{array}\right)=\left(\begin{array}{ccc}
1 & 0 & hf_{1}\\
0 & 1 & 0\\
0 & 0 & 1
\end{array}\right)\left(\begin{array}{ccc}
1 & 0 & 0\\
0 & 1 & hf_{2}\\
0 & 0 & 1
\end{array}\right)
\]
is generated by the matrices of the form (\ref{eq:form}), it remains
to show that 
\[
\left(\begin{array}{ccc}
1 & 0 & 0\\
0 & 1 & 0\\
f_{2} & -f_{1} & 1
\end{array}\right)\left(\begin{array}{ccc}
1 & 0 & -hf_{1}\\
0 & 1 & -hf_{2}\\
0 & 0 & 1
\end{array}\right)\left(\begin{array}{ccc}
1 & 0 & 0\\
0 & 1 & 0\\
-f_{2} & f_{1} & 1
\end{array}\right)
\]
\begin{eqnarray*}
 & = & \left(\begin{array}{ccc}
1 & 0 & 0\\
0 & 1 & 0\\
f_{2} & -f_{1} & 1
\end{array}\right)\left(\begin{array}{ccc}
1 & 0 & -hf_{1}\\
0 & 1 & 0\\
0 & 0 & 1
\end{array}\right)\left(\begin{array}{ccc}
1 & 0 & 0\\
0 & 1 & 0\\
-f_{2} & f_{1} & 1
\end{array}\right)\\
 &  & \cdot\left(\begin{array}{ccc}
1 & 0 & 0\\
0 & 1 & 0\\
f_{2} & -f_{1} & 1
\end{array}\right)\left(\begin{array}{ccc}
1 & 0 & 0\\
0 & 1 & -hf_{2}\\
0 & 0 & 1
\end{array}\right)\left(\begin{array}{ccc}
1 & 0 & 0\\
0 & 1 & 0\\
-f_{2} & f_{1} & 1
\end{array}\right)
\end{eqnarray*}
is generated by the matrices of the form (\ref{eq:form}). Now
\[
\left(\begin{array}{ccc}
1 & 0 & 0\\
0 & 1 & 0\\
f_{2} & -f_{1} & 1
\end{array}\right)\left(\begin{array}{ccc}
1 & 0 & -hf_{1}\\
0 & 1 & 0\\
0 & 0 & 1
\end{array}\right)\left(\begin{array}{ccc}
1 & 0 & 0\\
0 & 1 & 0\\
-f_{2} & f_{1} & 1
\end{array}\right)
\]
\begin{eqnarray*}
 & = & \left(\begin{array}{ccc}
1 & 0 & 0\\
0 & 1 & 0\\
0 & -hf_{1}^{2}f_{2} & 1
\end{array}\right)\left(\begin{array}{ccc}
1 & -hf_{1}^{2} & 0\\
0 & 1 & 0\\
0 & 0 & 1
\end{array}\right)\\
 &  & \cdot\left[\left(\begin{array}{ccc}
1 & 0 & 0\\
0 & 1 & 0\\
f_{2} & 0 & 1
\end{array}\right)\left(\begin{array}{ccc}
1 & 0 & -hf_{1}\\
0 & 1 & 0\\
0 & 0 & 1
\end{array}\right)\left(\begin{array}{ccc}
1 & 0 & 0\\
0 & 1 & 0\\
-f_{2} & 0 & 1
\end{array}\right)\right]
\end{eqnarray*}
is generated by the matrices of the form (\ref{eq:form}), and by
similar computation
\[
\left(\begin{array}{ccc}
1 & 0 & 0\\
0 & 1 & 0\\
f_{2} & -f_{1} & 1
\end{array}\right)\left(\begin{array}{ccc}
1 & 0 & 0\\
0 & 1 & -hf_{2}\\
0 & 0 & 1
\end{array}\right)\left(\begin{array}{ccc}
1 & 0 & 0\\
0 & 1 & 0\\
-f_{2} & f_{1} & 1
\end{array}\right)
\]
is generated by these matrices as well.
\end{proof}
We proceed with the following lemma. Some of the ideas in its proof
are based on the proof of Proposition 3.5 in \cite{key-24}.
\begin{lem}
\label{lem:decomposition}Let $n\geq4$. Recall $U_{r,m}=(x_{r}^{m}-1)R_{n}$,
$O_{m}=mR_{n}$, and denote the corresponding ideals of $R_{n-1}=\mathbb{Z}[x_{1}^{\pm1},\ldots,x_{n-1}^{\pm1}]\subseteq R_{n}$
by
\[
\bar{O}_{m}=mR_{n-1}\subseteq O_{m},\,\,\,\,\bar{U}_{r,m}=(x_{r}^{m}-1)R_{n-1}\subseteq U_{r,m}\,\,\,\,\textrm{for}\,\,\,\,1\leq r\leq n-1.
\]
Then, every element of $GL_{n-1}\left(R_{n},\sigma_{n}R_{n}\right)\cap E_{n-1}\left(R_{n},T_{m}\right)$
can be decomposed as a product of elements of the following four forms:\\

\begin{tabular}{|l||l||>{\raggedright}p{3.8cm}|}
\hline 
\noalign{\vskip\doublerulesep}
1. & $A^{-1}\left(I_{n-1}+hE_{i,j}\right)A$ & $h\in\sigma_{n}O_{m}$\tabularnewline[\doublerulesep]
\hline 
\hline 
\noalign{\vskip\doublerulesep}
\multirow{1}{*}[-0.25cm]{2.} & \multirow{1}{*}[-0.25cm]{$A^{-1}\left(I_{n-1}+hE_{i,j}\right)A$} & $h\in\sigma_{n}^{2}U_{n,m},\,\sigma_{n}\sigma_{r}^{2}U_{r,m}$ for
$1\leq r\leq n-1$\tabularnewline[\doublerulesep]
\hline 
\hline 
\noalign{\vskip\doublerulesep}
3. & $A^{-1}\left[\left(I_{n-1}+hE_{i,j}\right),\left(I_{n-1}+fE_{j,i}\right)\right]A$ & $h\in\bar{O}_{m}^{2}$, $f\in\sigma_{n}R_{n}$\tabularnewline[\doublerulesep]
\hline 
\hline 
\noalign{\vskip\doublerulesep}
\multirow{1}{*}[-0.25cm]{4.} & \multirow{1}{*}[-0.25cm]{$A^{-1}\left[\left(I_{n-1}+hE_{i,j}\right),\left(I_{n-1}+fE_{j,i}\right)\right]A$} & $h\in\sigma_{r}^{2}\bar{U}_{r,m},\sigma_{r}\bar{O}_{m}$ for $1\leq r\leq n-1$,
$f\in\sigma_{n}R_{n}$\tabularnewline[\doublerulesep]
\hline 
\end{tabular}\\
\\
where $A\in GL_{n-1}(R_{n})$ and $i\neq j$. 
\end{lem}

\begin{rem}
\label{rem:Normal}Notice that as $GL_{n-1}\left(R_{n},\sigma_{n}R_{n}\right)$
is normal in $GL_{n-1}\left(R_{n}\right)$, every element of the above
forms is an element of $GL_{n-1}\left(R_{n},\sigma_{n}R_{n}\right)\cong IGL_{n-1,n}\leq IA\left(\Phi\right)$.
\end{rem}

\begin{proof}
(of Lemma \ref{lem:decomposition}) Let $B\in GL_{n-1}\left(R_{n},\sigma_{n}R_{n}\right)\cap E_{n-1}\left(R_{n},T_{m}\right)$.
We first claim that for proving the lemma, it is enough to show that
$B$ can be decomposed as a product of the elements in the lemma (Lemma
\ref{lem:decomposition}), and arbitrary elements in $GL_{n-1}\left(R_{n-1}\right)$.
Indeed, assume that we can write $B=A_{1}D_{1}\cdots A_{n}D_{n}$
for some $D_{i}$ of the forms in the lemma and $A_{i}\in GL_{n-1}\left(R_{n-1}\right)$
(notice that $A_{1}$ or $D_{n}$ might be equal to $I_{n-1}$). Observe
now that we can therefore write
\[
B=A_{1}D_{1}A_{1}^{-1}\cdot\ldots\cdot\left(A_{1}\cdot\ldots\cdot A_{n}\right)D_{n}\left(A_{1}\cdot\ldots\cdot A_{n}\right)^{-1}\left(A_{1}\cdot\ldots\cdot A_{n}\right)
\]
and by definition, the conjugations of the $D_{i}$-s can also be
considered as of the forms in the lemma. On the other hand, we have
\[
\left(A_{1}\cdot\ldots\cdot A_{n}\right)D_{n}^{-1}\left(A_{1}\cdot\ldots\cdot A_{n}\right)^{-1}\cdot\ldots\cdot A_{1}D_{1}^{-1}A_{1}^{-1}B=A_{1}\cdot\ldots\cdot A_{n}
\]
and as the matrices of the forms in the lemma are all in $GL_{n-1}\left(R_{n},\sigma_{n}R_{n}\right)$
(by Remark \ref{rem:Normal}) we deduce that
\[
A_{1}\cdot\ldots\cdot A_{n}\in GL_{n-1}\left(R_{n},\sigma_{n}R_{n}\right)\cap GL_{n-1}\left(R_{n-1}\right)=\left\{ I_{n-1}\right\} 
\]
i.e. $A_{1}\cdot\ldots\cdot A_{n}=I_{n-1}$. Hence
\[
B=A_{1}D_{1}A_{1}^{-1}\cdot\ldots\cdot\left(A_{1}\cdot\ldots\cdot A_{n}\right)D_{n}\left(A_{1}\cdot\ldots\cdot A_{n}\right)^{-1}
\]
i.e. $B$ is a product of matrices of the forms in the lemma, as required.

So let $B\in GL_{n-1}\left(R_{n},\sigma_{n}R_{n}\right)\cap E_{n-1}\left(R_{n},T_{m}\right)$.
According to Proposition \ref{lem:suslin-bachmuth}, as $B\in E_{n-1}\left(R_{n},T_{m}\right)$
and $n-1\geq3$, we can write $B$ as a product of elements of the
form
\[
\left(I_{n-1}-fE_{i,j}\right)\left(I_{n-1}+hE_{j,i}\right)\left(I_{n-1}+fE_{i,j}\right)
\]
for some $f\in R_{n}$, $h\in T_{m}=\sum_{r=1}^{n}\sigma_{r}^{2}U_{r,m}+\sum_{r=1}^{n}\sigma_{r}O_{m}+O_{m}^{2}$
and $1\leq i\neq j\leq n-1$. We will show now that every element
of the above form can be written as a product of the elements of the
forms in the lemma and elements of $GL_{n-1}\left(R_{n-1}\right)$. 

So let $h\in T$ and $f\in R_{n}$. Observe first that by division
by $\sigma_{n}$ (with residue) one has
\begin{eqnarray*}
T_{m} & = & \sum_{r=1}^{n}\sigma_{r}^{2}U_{r,m}+\sum_{r=1}^{n}\sigma_{r}O_{m}+O_{m}^{2}\\
 & \subseteq & \sigma_{n}(\sum_{r=1}^{n-1}\sigma_{r}^{2}U_{r,m}+\sigma_{n}U_{n,m}+O_{m})+\sum_{r=1}^{n-1}\sigma_{r}^{2}\bar{U}_{r,m}+\sum_{r=1}^{n-1}\sigma_{r}\bar{O}_{m}+\bar{O}_{m}^{2}.
\end{eqnarray*}
Hence, we can decompose $h=\sigma_{n}h_{1}+h_{2}$ for some: $h_{1}\in\sum_{r=1}^{n-1}\sigma_{r}^{2}U_{r,m}+\sigma_{n}U_{n,m}+O_{m}$
and $h_{2}\in\sum_{r=1}^{n-1}\sigma_{r}^{2}\bar{U}_{r,m}+\sum_{r=1}^{n-1}\sigma_{r}\bar{O}_{m}+\bar{O}_{m}^{2}$.
Therefore, we can write
\begin{align*}
 & \left(I_{n-1}-fE_{i,j}\right)\left(I_{n-1}+hE_{j,i}\right)\left(I_{n-1}+fE_{i,j}\right)\\
 & =\left(I_{n-1}-fE_{i,j}\right)\left(I_{n-1}+\sigma_{n}h_{1}E_{j,i}\right)\left(I_{n-1}+fE_{i,j}\right)\\
 & \,\,\,\,\,\,\cdot\left(I_{n-1}-fE_{i,j}\right)\left(I_{n-1}+h_{2}E_{j,i}\right)\left(I_{n-1}+fE_{i,j}\right).
\end{align*}
Thus, as the matrix $\left(I_{n-1}-fE_{i,j}\right)\left(I_{n-1}+\sigma_{n}h_{1}E_{j,i}\right)\left(I_{n-1}+fE_{i,j}\right)$
is clearly a product of elements of Forms 1 and 2 in the lemma, it
is enough to deal with the matrix
\[
\left(I_{n-1}-fE_{i,j}\right)\left(I_{n-1}+h_{2}E_{j,i}\right)\left(I_{n-1}+fE_{i,j}\right)
\]
when $h_{2}\in\sum_{r=1}^{n-1}\sigma_{r}^{2}\bar{U}_{r,m}+\sum_{r=1}^{n-1}\sigma_{r}\bar{O}_{m}+\bar{O}_{m}^{2}$.
Let us now write: $f=\sigma_{n}f_{1}+f_{2}$ for some $f_{1}\in R_{n}$
and $f_{2}\in R_{n-1}$, and write
\begin{align*}
 & \left(I_{n-1}-fE_{i,j}\right)\left(I_{n-1}+h_{2}E_{j,i}\right)\left(I_{n-1}+fE_{i,j}\right)\\
 & =\left(I_{n-1}-f_{2}E_{i,j}\right)\left(I_{n-1}-\sigma_{n}f_{1}E_{i,j}\right)\\
 & \,\,\,\,\,\,\cdot\left(I_{n-1}+h_{2}E_{j,i}\right)\left(I_{n-1}+\sigma_{n}f_{1}E_{i,j}\right)\left(I_{n-1}+f_{2}E_{i,j}\right).
\end{align*}
Now, as $\left(I_{n-1}\pm f_{2}E_{i,j}\right)\in GL_{n-1}\left(R_{n-1}\right)$,
it is enough to deal with the element
\[
\left(I_{n-1}-\sigma_{n}f_{1}E_{i,j}\right)\left(I_{n-1}+h_{2}E_{j,i}\right)\left(I_{n-1}+\sigma_{n}f_{1}E_{i,j}\right)
\]
which can be written as a product of elements of the form
\begin{eqnarray*}
 &  & \left(I_{n-1}-\sigma_{n}f_{1}E_{i,j}\right)\left(I_{n-1}+kE_{j,i}\right)\left(I_{n-1}+\sigma_{n}f_{1}E_{i,j}\right)\\
 &  & k\in\bar{O}_{m}^{2},\,\,\sigma_{r}^{2}\bar{U}_{r,m},\,\,\sigma_{r}\bar{O}_{m},\,\,\,\,\textrm{for}\,\,\,\,1\leq r\leq n-1.
\end{eqnarray*}
Finally, as for every such $k$ one can write
\begin{eqnarray*}
 &  & \left(I_{n-1}-\sigma_{n}f_{1}E_{i,j}\right)\left(I_{n-1}+kE_{j,i}\right)\left(I_{n-1}+\sigma_{n}f_{1}E_{i,j}\right)\\
 &  & =\left(I_{n-1}+kE_{j,i}\right)\left[\left(I_{n-1}-kE_{j,i}\right),\left(I_{n-1}-\sigma_{n}f_{1}E_{i,j}\right)\right]
\end{eqnarray*}
and $\left(I_{n-1}+kE_{j,i}\right)\in GL_{n-1}\left(R_{n-1}\right)$,
we actually finished.
\end{proof}
\begin{cor}
For proving Lemma \ref{lem:main}, it is enough to show that every
element of the forms in Lemma \ref{lem:decomposition} is in $IA^{m}$.
\end{cor}

We start here by dealing with the elements of Form $1$:
\begin{prop}
\label{prop:form1}Recall $O_{m}=mR_{n}$. The elements of the following
form are in $IA^{m}$:
\[
A^{-1}\left(I_{n-1}+hE_{i,j}\right)A,\,\,\,\,\textrm{for}\,\,\,\,A\in GL_{n-1}\left(R_{n}\right),\,\,\,\,h\in\sigma_{n}O_{m}\,\,\,\,\textrm{and}\,\,\,\,i\neq j.
\]
\end{prop}

\begin{proof}
In this case we can write $h=\sigma_{n}mh'$ for some $h'\in R_{n}$.
So as
\[
A^{-1}\left(I_{n-1}+\sigma_{n}h'E_{i,j}\right)A\in GL_{n-1}\left(R_{n},\sigma_{n}R_{n}\right)\leq IA\left(\Phi\right)
\]
we obtain that
\begin{eqnarray*}
A^{-1}\left(I_{n-1}+hE_{i,j}\right)A & = & A^{-1}\left(I_{n-1}+\sigma_{n}mh'E_{i,j}\right)A=\\
 & = & \left(A^{-1}\left(I_{n-1}+\sigma_{n}h'E_{i,j}\right)A\right)^{m}\in IA^{m}
\end{eqnarray*}
as required.
\end{proof}
We will devote the remaining sections to deal with the elements of
the other three forms. In these cases the proof will be more difficult,
and we will need the help of the following computations.

\subsection{Some auxiliary computations}
\begin{prop}
\label{prop:computations}For every $f,g\in R_{n}$ we have the following
equalities:
\[
\left(\begin{array}{ccc}
1-fg & -fg & 0\\
fg & 1+fg & 0\\
0 & 0 & 1
\end{array}\right)
\]
\begin{eqnarray}
 & = & \left(\begin{array}{ccc}
1 & 0 & 0\\
fg & 1 & 0\\
fg^{2} & 0 & 1
\end{array}\right)\left(\begin{array}{ccc}
1 & 0 & 0\\
0 & 1+fg & -f\\
0 & fg^{2} & 1-fg
\end{array}\right)\label{eq:comp1}\\
 &  & \cdot\left(\begin{array}{ccc}
1 & -fg & 0\\
0 & 1 & 0\\
0 & -fg^{2} & 1
\end{array}\right)\left(\begin{array}{ccc}
1-fg & 0 & f\\
0 & 1 & 0\\
-fg^{2} & 0 & 1+fg
\end{array}\right)\left(\begin{array}{ccc}
1 & 0 & -f\\
0 & 1 & f\\
0 & 0 & 1
\end{array}\right)\nonumber 
\end{eqnarray}
\begin{eqnarray}
 & = & \left(\begin{array}{ccc}
1 & 0 & 0\\
fg & 1 & 0\\
-fg^{2} & 0 & 1
\end{array}\right)\left(\begin{array}{ccc}
1 & 0 & 0\\
0 & 1+fg & f\\
0 & -fg^{2} & 1-fg
\end{array}\right)\label{eq:comp2}\\
 &  & \cdot\left(\begin{array}{ccc}
1 & -fg & 0\\
0 & 1 & 0\\
0 & fg^{2} & 1
\end{array}\right)\left(\begin{array}{ccc}
1-fg & 0 & -f\\
0 & 1 & 0\\
fg^{2} & 0 & 1+fg
\end{array}\right)\left(\begin{array}{ccc}
1 & 0 & f\\
0 & 1 & -f\\
0 & 0 & 1
\end{array}\right)\nonumber 
\end{eqnarray}
\begin{eqnarray}
 & = & \left(\begin{array}{ccc}
1 & 0 & 0\\
0 & 1 & 0\\
f & f & 1
\end{array}\right)\left(\begin{array}{ccc}
1-fg & 0 & fg^{2}\\
0 & 1 & 0\\
-f & 0 & 1+fg
\end{array}\right)\left(\begin{array}{ccc}
1 & 0 & 0\\
fg & 1 & -fg^{2}\\
0 & 0 & 1
\end{array}\right)\nonumber \\
 &  & \cdot\left(\begin{array}{ccc}
1 & 0 & 0\\
0 & 1+fg & fg^{2}\\
0 & -f & 1-fg
\end{array}\right)\left(\begin{array}{ccc}
1 & -fg & -fg^{2}\\
0 & 1 & 0\\
0 & 0 & 1
\end{array}\right)\label{eq:comp3}
\end{eqnarray}
\begin{eqnarray}
 & = & \left(\begin{array}{ccc}
1 & 0 & 0\\
0 & 1 & 0\\
-f & -f & 1
\end{array}\right)\left(\begin{array}{ccc}
1-fg & 0 & -fg^{2}\\
0 & 1 & 0\\
f & 0 & 1+fg
\end{array}\right)\left(\begin{array}{ccc}
1 & 0 & 0\\
fg & 1 & fg^{2}\\
0 & 0 & 1
\end{array}\right)\nonumber \\
 &  & \cdot\left(\begin{array}{ccc}
1 & 0 & 0\\
0 & 1+fg & -fg^{2}\\
0 & f & 1-fg
\end{array}\right)\left(\begin{array}{ccc}
1 & -fg & fg^{2}\\
0 & 1 & 0\\
0 & 0 & 1
\end{array}\right)\label{eq:comp4}
\end{eqnarray}
\end{prop}

\begin{proof}
We use square brackets to help the reader follow the steps of the
computation. Here is the computation for Equation ($\ref{eq:comp1}$):
\[
\left(\begin{array}{ccc}
1-fg & -fg & 0\\
fg & 1+fg & 0\\
0 & 0 & 1
\end{array}\right)=\left(\begin{array}{ccc}
1 & 0 & f\\
0 & 1 & -f\\
g & g & 1
\end{array}\right)\left(\begin{array}{ccc}
1 & 0 & -f\\
0 & 1 & f\\
-g & -g & 1
\end{array}\right)
\]
\[
=\left(\begin{array}{ccc}
1 & 0 & 0\\
0 & 1 & 0\\
g & g & 1
\end{array}\right)\left(\begin{array}{ccc}
1 & 0 & f\\
0 & 1 & -f\\
0 & 0 & 1
\end{array}\right)\left(\begin{array}{ccc}
1 & 0 & 0\\
0 & 1 & 0\\
-g & -g & 1
\end{array}\right)\left(\begin{array}{ccc}
1 & 0 & -f\\
0 & 1 & f\\
0 & 0 & 1
\end{array}\right)
\]
\begin{eqnarray*}
 & = & \left[\left(\begin{array}{ccc}
1 & 0 & 0\\
0 & 1 & 0\\
g & g & 1
\end{array}\right)\left(\begin{array}{ccc}
1 & 0 & 0\\
0 & 1 & -f\\
0 & 0 & 1
\end{array}\right)\left(\begin{array}{ccc}
1 & 0 & 0\\
0 & 1 & 0\\
-g & -g & 1
\end{array}\right)\right]\\
 &  & \cdot\left[\left(\begin{array}{ccc}
1 & 0 & 0\\
0 & 1 & 0\\
g & g & 1
\end{array}\right)\left(\begin{array}{ccc}
1 & 0 & f\\
0 & 1 & 0\\
0 & 0 & 1
\end{array}\right)\left(\begin{array}{ccc}
1 & 0 & 0\\
0 & 1 & 0\\
-g & -g & 1
\end{array}\right)\right]\left(\begin{array}{ccc}
1 & 0 & -f\\
0 & 1 & f\\
0 & 0 & 1
\end{array}\right)
\end{eqnarray*}
\[
=\left(\begin{array}{ccc}
1 & 0 & 0\\
fg & 1+fg & -f\\
fg^{2} & fg^{2} & 1-fg
\end{array}\right)\left(\begin{array}{ccc}
1-fg & -fg & f\\
0 & 1 & 0\\
-fg^{2} & -fg^{2} & 1+fg
\end{array}\right)\left(\begin{array}{ccc}
1 & 0 & -f\\
0 & 1 & f\\
0 & 0 & 1
\end{array}\right)
\]
\begin{eqnarray*}
 & = & \left(\begin{array}{ccc}
1 & 0 & 0\\
fg & 1 & 0\\
fg^{2} & 0 & 1
\end{array}\right)\left(\begin{array}{ccc}
1 & 0 & 0\\
0 & 1+fg & -f\\
0 & fg^{2} & 1-fg
\end{array}\right)\\
 &  & \cdot\left(\begin{array}{ccc}
1 & -fg & 0\\
0 & 1 & 0\\
0 & -fg^{2} & 1
\end{array}\right)\left(\begin{array}{ccc}
1-fg & 0 & f\\
0 & 1 & 0\\
-fg^{2} & 0 & 1+fg
\end{array}\right)\left(\begin{array}{ccc}
1 & 0 & -f\\
0 & 1 & f\\
0 & 0 & 1
\end{array}\right).
\end{eqnarray*}
Equation (\ref{eq:comp2}) is obtained similarly by changing the signs
of $f$ and $g$ simultaneously. Here is the computation for Equation
(\ref{eq:comp3}):
\[
\left(\begin{array}{ccc}
1-fg & -fg & 0\\
fg & 1+fg & 0\\
0 & 0 & 1
\end{array}\right)=\left(\begin{array}{ccc}
1 & 0 & g\\
0 & 1 & -g\\
f & f & 1
\end{array}\right)\left(\begin{array}{ccc}
1 & 0 & -g\\
0 & 1 & g\\
-f & -f & 1
\end{array}\right)
\]
\[
=\left(\begin{array}{ccc}
1 & 0 & 0\\
0 & 1 & 0\\
f & f & 1
\end{array}\right)\left(\begin{array}{ccc}
1 & 0 & g\\
0 & 1 & -g\\
0 & 0 & 1
\end{array}\right)\left(\begin{array}{ccc}
1 & 0 & 0\\
0 & 1 & 0\\
-f & -f & 1
\end{array}\right)\left(\begin{array}{ccc}
1 & 0 & -g\\
0 & 1 & g\\
0 & 0 & 1
\end{array}\right)
\]
\begin{eqnarray*}
 & = & \left(\begin{array}{ccc}
1 & 0 & 0\\
0 & 1 & 0\\
f & f & 1
\end{array}\right)\left[\left(\begin{array}{ccc}
1 & 0 & g\\
0 & 1 & -g\\
0 & 0 & 1
\end{array}\right)\left(\begin{array}{ccc}
1 & 0 & 0\\
0 & 1 & 0\\
-f & 0 & 1
\end{array}\right)\left(\begin{array}{ccc}
1 & 0 & -g\\
0 & 1 & g\\
0 & 0 & 1
\end{array}\right)\right]\\
 &  & \cdot\left[\left(\begin{array}{ccc}
1 & 0 & g\\
0 & 1 & -g\\
0 & 0 & 1
\end{array}\right)\left(\begin{array}{ccc}
1 & 0 & 0\\
0 & 1 & 0\\
0 & -f & 1
\end{array}\right)\left(\begin{array}{ccc}
1 & 0 & -g\\
0 & 1 & g\\
0 & 0 & 1
\end{array}\right)\right]
\end{eqnarray*}
\[
=\left(\begin{array}{ccc}
1 & 0 & 0\\
0 & 1 & 0\\
f & f & 1
\end{array}\right)\left(\begin{array}{ccc}
1-fg & 0 & fg^{2}\\
fg & 1 & -fg^{2}\\
-f & 0 & 1+fg
\end{array}\right)\left(\begin{array}{ccc}
1 & -fg & -fg^{2}\\
0 & 1+fg & fg^{2}\\
0 & -f & 1-fg
\end{array}\right)
\]
\begin{eqnarray*}
 & = & \left(\begin{array}{ccc}
1 & 0 & 0\\
0 & 1 & 0\\
f & f & 1
\end{array}\right)\left(\begin{array}{ccc}
1-fg & 0 & fg^{2}\\
0 & 1 & 0\\
-f & 0 & 1+fg
\end{array}\right)\left(\begin{array}{ccc}
1 & 0 & 0\\
fg & 1 & -fg^{2}\\
0 & 0 & 1
\end{array}\right)\\
 &  & \cdot\left(\begin{array}{ccc}
1 & 0 & 0\\
0 & 1+fg & fg^{2}\\
0 & -f & 1-fg
\end{array}\right)\left(\begin{array}{ccc}
1 & -fg & -fg^{2}\\
0 & 1 & 0\\
0 & 0 & 1
\end{array}\right)
\end{eqnarray*}
and Equation (\ref{eq:comp4}) is obtained similarly by changing the
signs of $f$ and $g$ simultaneously. 
\end{proof}
In the following corollary, a $3\times3$ matrix $B\in GL_{3}\left(R_{n}\right)$
denotes the block matrix 
\[
\left(\begin{array}{cc}
B & 0\\
0 & I_{n-4}
\end{array}\right)\in GL_{n-1}\left(R_{n}\right).
\]
\begin{cor}
\label{cor:computations}Let $n\geq4$, $f\in\sigma_{n}(\sum_{r=1}^{n-1}\sigma_{r}U_{r,m}+U_{n,m}+O_{m})$
and $g\in R_{n}$. Then, mod $IA^{m}$ we have the following equalities
(the indices are intended to help us later to recognize forms of matrices:
form $7$, form $12$ etc.):
\begin{equation}
\left(\begin{array}{ccc}
1-fg & -fg & 0\\
fg & 1+fg & 0\\
0 & 0 & 1
\end{array}\right)_{13}\label{eq:cor-comp1}
\end{equation}
\[
\equiv\left(\begin{array}{ccc}
1 & 0 & 0\\
0 & 1+fg & -f\\
0 & fg^{2} & 1-fg
\end{array}\right)_{1}\left(\begin{array}{ccc}
1-fg & 0 & f\\
0 & 1 & 0\\
-fg^{2} & 0 & 1+fg
\end{array}\right)_{2}
\]
\[
\equiv\left(\begin{array}{ccc}
1 & 0 & 0\\
0 & 1+fg & f\\
0 & -fg^{2} & 1-fg
\end{array}\right)_{3}\left(\begin{array}{ccc}
1-fg & 0 & -f\\
0 & 1 & 0\\
fg^{2} & 0 & 1+fg
\end{array}\right)_{4}
\]
\[
\equiv\left(\begin{array}{ccc}
1-fg & 0 & fg^{2}\\
0 & 1 & 0\\
-f & 0 & 1+fg
\end{array}\right)_{5}\left(\begin{array}{ccc}
1 & 0 & 0\\
0 & 1+fg & fg^{2}\\
0 & -f & 1-fg
\end{array}\right)_{6}
\]
\[
\equiv\left(\begin{array}{ccc}
1-fg & 0 & -fg^{2}\\
0 & 1 & 0\\
f & 0 & 1+fg
\end{array}\right)_{7}\left(\begin{array}{ccc}
1 & 0 & 0\\
0 & 1+fg & -fg^{2}\\
0 & f & 1-fg
\end{array}\right)_{8}
\]
and
\begin{equation}
\left(\begin{array}{ccc}
1-fg & fg & 0\\
-fg & 1+fg & 0\\
0 & 0 & 1
\end{array}\right)_{14}\label{eq:cor-comp2}
\end{equation}
\[
\equiv\left(\begin{array}{ccc}
1-fg & 0 & -fg^{2}\\
0 & 1 & 0\\
f & 0 & 1+fg
\end{array}\right)_{7}\left(\begin{array}{ccc}
1 & 0 & 0\\
0 & 1+fg & fg^{2}\\
0 & -f & 1-fg
\end{array}\right)_{6}.
\]

Moreover (the inverse of a matrix is denoted by the same index - one
can observe that the inverse of each matrix in these equations is
obtained by changing the sign of $f$)
\begin{equation}
\left(\begin{array}{ccc}
1-fg & 0 & -fg\\
0 & 1 & 0\\
fg & 0 & 1+fg
\end{array}\right)_{15}\label{eq:cor-comp3}
\end{equation}
\[
\equiv\left(\begin{array}{ccc}
1 & 0 & 0\\
0 & 1-fg & fg^{2}\\
0 & -f & 1+fg
\end{array}\right)_{8}\left(\begin{array}{ccc}
1-fg & f & 0\\
-fg^{2} & 1+fg & 1\\
0 & 0 & 0
\end{array}\right)_{9}
\]
\[
\equiv\left(\begin{array}{ccc}
1 & 0 & 0\\
0 & 1-fg & -fg^{2}\\
0 & f & 1+fg
\end{array}\right)_{6}\left(\begin{array}{ccc}
1-fg & -f & 0\\
fg^{2} & 1+fg & 0\\
0 & 0 & 1
\end{array}\right)_{10}
\]
\[
\equiv\left(\begin{array}{ccc}
1-fg & fg^{2} & 0\\
-f & 1+fg & 0\\
0 & 0 & 1
\end{array}\right)_{11}\left(\begin{array}{ccc}
1 & 0 & 0\\
0 & 1-fg & -f\\
0 & fg^{2} & 1+fg
\end{array}\right)_{3}
\]
\[
\equiv\left(\begin{array}{ccc}
1-fg & -fg^{2} & 0\\
f & 1+fg & 0\\
0 & 0 & 1
\end{array}\right)_{12}\left(\begin{array}{ccc}
1 & 0 & 0\\
0 & 1-fg & f\\
0 & -fg^{2} & 1+fg
\end{array}\right)_{1}
\]
and
\begin{equation}
\left(\begin{array}{ccc}
1-fg & 0 & fg\\
0 & 1 & 0\\
-fg & 0 & 1+fg
\end{array}\right)_{16}\label{eq:cor-comp4}
\end{equation}
\[
\equiv\left(\begin{array}{ccc}
1-fg & -fg^{2} & 0\\
f & 1+fg & 0\\
0 & 0 & 1
\end{array}\right)_{12}\left(\begin{array}{ccc}
1 & 0 & 0\\
0 & 1-fg & -f\\
0 & fg^{2} & 1+fg
\end{array}\right)_{3}
\]
and
\begin{equation}
\left(\begin{array}{ccc}
1 & 0 & 0\\
0 & 1-fg & -fg\\
0 & fg & 1+fg
\end{array}\right)_{17}\label{eq:cor-comp5}
\end{equation}
\[
\equiv\left(\begin{array}{ccc}
1-fg & 0 & fg^{2}\\
0 & 1 & 0\\
-f & 0 & 1+fg
\end{array}\right)_{5}\left(\begin{array}{ccc}
1+fg & -fg^{2} & 0\\
f & 1-fg & 1\\
0 & 0 & 0
\end{array}\right)_{11}
\]
and
\begin{equation}
\left(\begin{array}{ccc}
1 & 0 & 0\\
0 & 1-fg & fg\\
0 & -fg & 1+fg
\end{array}\right)_{18}\label{eq:cor-comp6}
\end{equation}
\[
\equiv\left(\begin{array}{ccc}
1+fg & f & 0\\
-fg^{2} & 1-fg & 0\\
0 & 0 & 1
\end{array}\right)_{10}\left(\begin{array}{ccc}
1-fg & 0 & -f\\
0 & 1 & 0\\
fg^{2} & 0 & 1+fg
\end{array}\right)_{4}
\]
\end{cor}

\begin{rem}
We remark that as $f\in\sigma_{n}R_{n}$, then every matrix which
takes part in the above epualities is indeed in $GL_{n-1}\left(R_{n},\sigma_{n}R_{n}\right)\cong IGL_{n-1,n}\leq IA\left(\Phi\right)$. 
\end{rem}

\begin{proof}
As $f\in\sigma_{n}(\sum_{r=1}^{n-1}\sigma_{r}U_{r,m}+U_{n,m}+O_{m})$,
Equation (\ref{eq:cor-comp1}) is obtained by applying Proposition
\ref{prop:computations} combined with Proposition \ref{prop:type 1.1}.
Equation (\ref{eq:cor-comp2}) is obtained similarly by transposing
all the computations which led to the first part of Equation (\ref{eq:cor-comp1}).
Similarly, by switching the roles of the second row and column with
the third row and column, one obtains Equations (\ref{eq:cor-comp3})
and (\ref{eq:cor-comp4}). By switching one more time the roles of
the first row and column with the second row and column, we obtain
Equations (\ref{eq:cor-comp5}) and (\ref{eq:cor-comp6}) as well.
\end{proof}

\subsection{Elements of Form $2$}
\begin{prop}
\label{prop:form2}Recall $U_{r,m}=(x_{r}^{m}-1)R_{n}$. The elements
of the following form, belong to $IA^{m}$:
\[
A^{-1}\left(I_{n-1}+hE_{i,j}\right)A
\]
where $A\in GL_{n-1}(R_{n})$, $h\in\sigma_{n}\sigma_{r}^{2}U_{r,m},\sigma_{n}^{2}U_{n,m}$
for $1\leq r\leq n-1$ and $i\neq j$.
\end{prop}

Notice that for every $n\geq4$, the groups $E_{n-1}(\sigma_{n}^{2}U_{n,m})$
and $E_{n-1}(\sigma_{n}\sigma_{r}^{2}U_{r,m})$ for $1\leq r\leq n-1$
are normal in $GL_{n-1}(R_{n})$, and thus, all the above elements
are in $E_{n-1}(\sigma_{n}^{2}U_{n,m})$ and $E_{n-1}(\sigma_{n}\sigma_{r}^{2}U_{r,m})$.
Hence, for proving Proposition \ref{prop:form2}, it is enough to
show that for every $1\leq r\leq n-1$ we have, $E_{n-1}(\sigma_{n}^{2}U_{n,m}),E_{n-1}(\sigma_{n}\sigma_{r}^{2}U_{r,m})\subseteq IA^{m}$.
Therefore, by Proposition \ref{lem:suslin-bachmuth}, for proving
Proposition \ref{prop:form2}, it is enough to show that the elements
of the following form are in $IA^{m}$:
\[
\left(I_{n-1}-fE_{j,i}\right)\left(I_{n-1}+hE_{i,j}\right)\left(I_{n-1}+fE_{j,i}\right)
\]
when $h\in\sigma_{n}\sigma_{r}^{2}U_{r,m},\sigma_{n}^{2}U_{n,m}$
for $1\leq r\leq n-1$, $f\in R_{n}$ and $i\neq j$. We will prove
it in a few stages, starting with the following lemma.
\begin{lem}
\label{lem:sum}Let $h\in\sigma_{n}\sigma_{r}U_{r,m},\sigma_{n}U_{n,m}$
for $1\leq r\leq n-1$ and $f_{1},f_{2}\in R_{n}$. Assume that the
elements of the forms
\begin{eqnarray*}
 & \left(I_{n-1}\pm f_{1}E_{j,i}\right)\left(I_{n-1}+hE_{i,j}\right)\left(I_{n-1}\mp f_{1}E_{j,i}\right)\\
 & \left(I_{n-1}\pm f_{2}E_{j,i}\right)\left(I_{n-1}+hE_{i,j}\right)\left(I_{n-1}\mp f_{2}E_{j,i}\right)
\end{eqnarray*}
for every $1\leq i\neq j\leq n-1$, belong to $IA^{m}$. Then, the
elements of the form
\[
\left(I_{n-1}\pm\left(f_{1}+f_{2}\right)E_{j,i}\right)\left(I_{n-1}+hE_{i,j}\right)\left(I_{n-1}\mp\left(f_{1}+f_{2}\right)E_{j,i}\right)
\]
for $1\leq i\neq j\leq n-1$, also belong to $IA^{m}$.
\end{lem}

\begin{proof}
Observe first that by Proposition \ref{prop:type 1.1}, all the matrices
of the form $I_{n-1}+hE_{i,j}$ for $h\in\sigma_{n}\sigma_{r}U_{r,m},\sigma_{n}U_{n,m}$
belong to $IA^{m}$. We will use it in the following computations.
Without loss of generality, under the assumptions of the proposition,
we will show that for $i,j=2,1$ we have
\[
\left(I_{n-1}-\left(f_{1}+f_{2}\right)E_{1,2}\right)\left(I_{n-1}+hE_{2,1}\right)\left(I_{n-1}+\left(f_{1}+f_{2}\right)E_{1,2}\right)\in IA^{m}
\]
and the general argument is similar. In the following computation
we use the following notations:
\begin{itemize}
\item A matrix $\left(\begin{array}{cc}
B & 0\\
0 & I_{n-4}
\end{array}\right)\in GL_{n-1}(R_{n})$ is denoted by $B\in GL_{3}(R_{n})$. 
\item ``$=$'' denotes an equality between matrices in $GL_{n-1}(R_{n})$.
\item ``$\equiv$'' denotes an equality in $IA\left(\Phi\right)/IA^{m}$
.
\item We use square brackets to help the reader follow the steps of the
computation. 
\end{itemize}
So let's compute:

\[
\left(I_{n-1}-\left(f_{1}+f_{2}\right)E_{1,2}\right)\left(I_{n-1}+hE_{2,1}\right)\left(I_{n-1}+\left(f_{1}+f_{2}\right)E_{1,2}\right)
\]
\[
=\left(\begin{array}{ccc}
1-h\left(f_{1}+f_{2}\right) & -h\left(f_{1}+f_{2}\right)^{2} & 0\\
h & 1+h\left(f_{1}+f_{2}\right) & 0\\
0 & 0 & 1
\end{array}\right)
\]
\[
=\left(\begin{array}{ccc}
1 & 0 & -\left(f_{1}+f_{2}\right)\\
0 & 1 & 1\\
-h & -h\left(f_{1}+f_{2}\right) & 1
\end{array}\right)\left(\begin{array}{ccc}
1 & 0 & \left(f_{1}+f_{2}\right)\\
0 & 1 & -1\\
h & h\left(f_{1}+f_{2}\right) & 1
\end{array}\right)
\]
\begin{eqnarray*}
 & = & \left(\begin{array}{ccc}
1 & 0 & 0\\
0 & 1 & 0\\
-h & -h\left(f_{1}+f_{2}\right) & 1
\end{array}\right)\left(\begin{array}{ccc}
1 & 0 & -\left(f_{1}+f_{2}\right)\\
0 & 1 & 1\\
0 & 0 & 1
\end{array}\right)\\
 &  & \cdot\left(\begin{array}{ccc}
1 & 0 & 0\\
0 & 1 & 0\\
h & h\left(f_{1}+f_{2}\right) & 1
\end{array}\right)\left(\begin{array}{ccc}
1 & 0 & \left(f_{1}+f_{2}\right)\\
0 & 1 & -1\\
0 & 0 & 1
\end{array}\right)
\end{eqnarray*}
\begin{eqnarray*}
 & = & \left(\begin{array}{ccc}
1 & 0 & 0\\
0 & 1 & 0\\
-h & -h\left(f_{1}+f_{2}\right) & 1
\end{array}\right)\\
 &  & \cdot\left[\left(\begin{array}{ccc}
1 & 0 & -(f_{1}+f_{2})\\
0 & 1 & 1\\
0 & 0 & 1
\end{array}\right)\left(\begin{array}{ccc}
1 & 0 & 0\\
0 & 1 & 0\\
0 & hf_{2} & 1
\end{array}\right)\left(\begin{array}{ccc}
1 & 0 & f_{1}+f_{2}\\
0 & 1 & -1\\
0 & 0 & 1
\end{array}\right)\right]\\
 &  & \cdot\left[\left(\begin{array}{ccc}
1 & 0 & -f_{2}\\
0 & 1 & 0\\
0 & 0 & 1
\end{array}\right)\left(\begin{array}{ccc}
1 & 0 & 0\\
0 & 1 & 0\\
h & hf_{1} & 1
\end{array}\right)\left(\begin{array}{ccc}
1 & 0 & f_{2}\\
0 & 1 & 0\\
0 & 0 & 1
\end{array}\right)\right]\\
 &  & \cdot\left[\left(\begin{array}{ccc}
1 & 0 & -f_{2}\\
0 & 1 & 0\\
0 & 0 & 1
\end{array}\right)\left(\begin{array}{ccc}
1 & 0 & 0\\
0 & 1 & 0\\
-h & -hf_{1} & 1
\end{array}\right)\left(\begin{array}{ccc}
1 & 0 & -f_{1}\\
0 & 1 & 1\\
0 & 0 & 1
\end{array}\right)\right.\\
 &  & \cdot\left.\left(\begin{array}{ccc}
1 & 0 & 0\\
0 & 1 & 0\\
h & hf_{1} & 1
\end{array}\right)\left(\begin{array}{ccc}
1 & 0 & f_{1}\\
0 & 1 & -1\\
0 & 0 & 1
\end{array}\right)\left(\begin{array}{ccc}
1 & 0 & f_{2}\\
0 & 1 & 0\\
0 & 0 & 1
\end{array}\right)\right]
\end{eqnarray*}
\begin{align*}
= & \left(\begin{array}{ccc}
1 & 0 & 0\\
0 & 1 & 0\\
-h & -h\left(f_{1}+f_{2}\right) & 1
\end{array}\right)\\
 & \cdot\left(\begin{array}{ccc}
1 & 0 & 0\\
0 & 1+hf_{2} & -hf_{2}\\
0 & hf_{2} & 1-hf_{2}
\end{array}\right)\left(\begin{array}{ccc}
1 & -(f_{1}+f_{2})hf_{2} & (f_{1}+f_{2})hf_{2}\\
0 & 1 & 0\\
0 & 0 & 1
\end{array}\right)\\
 & \cdot\left(\begin{array}{ccc}
1 & -hf_{1}f_{2} & 0\\
0 & 1 & 0\\
0 & hf_{1} & 1
\end{array}\right)\left(\begin{array}{ccc}
1-hf_{2} & 0 & -hf_{2}^{2}\\
0 & 1 & 0\\
h & 0 & 1+hf_{2}
\end{array}\right)\\
 & \cdot\left(\begin{array}{ccc}
1 & 0 & -hf_{1}f_{2}\\
0 & 1 & hf_{2}\\
0 & 0 & 1
\end{array}\right)\left(\begin{array}{ccc}
1-hf_{1} & -hf_{1}^{2} & 0\\
h & 1+hf_{1} & 0\\
0 & 0 & 1
\end{array}\right).
\end{align*}

Notice now that by assumption, and by the remark at the beginning
of the proof, mod $IA^{m}$, the latter expression is congruent to
\[
\equiv\left(\begin{array}{ccc}
1 & 0 & 0\\
0 & 1+hf_{2} & -hf_{2}\\
0 & hf_{2} & 1-hf_{2}
\end{array}\right).
\]
Consider now Equation (\ref{eq:cor-comp6}) in Corollary \ref{cor:computations},
and switch the roles of $f,g$ by $-h,f_{2}$ respectively. Using
this identity we deduce that, mod $IA^{m}$, the latter expression
is congruent to 
\[
\equiv\left(\begin{array}{ccc}
1-hf_{2} & -h & 0\\
hf_{2}^{2} & 1+hf_{2} & 0\\
0 & 0 & 1
\end{array}\right)\left(\begin{array}{ccc}
1+hf_{2} & 0 & h\\
0 & 1 & 0\\
-hf_{2}^{2} & 0 & 1-hf_{2}
\end{array}\right)
\]
that is $\equiv I_{n-1}$ by assumption. This finishes the proof of
the lemma.
\end{proof}
We pass to the next stage:
\begin{prop}
\label{prop:stage}The elements of the following form belong to $IA^{m}$:
\[
\left(I_{n-1}-fE_{j,i}\right)\left(I_{n-1}+hE_{i,j}\right)\left(I_{n-1}+fE_{j,i}\right)
\]
where $h\in\sigma_{n}\sigma_{r}^{2}U_{r,m},\sigma_{n}^{2}U_{n,m}$
for $1\leq r\leq n-1$, \uline{\mbox{$f\in\mathbb{Z}$}} and $i\neq j$.
\end{prop}

\begin{rem}
We note that some of the matrices that we use in the following computations
lie in $IGL'_{n-1,n}\hookrightarrow GL_{n-1}(R_{n})$ and not necessarily
in $IGL_{n-1,n}$ (see Definition \ref{def:tag} and Proposition \ref{prop:tag}).
\end{rem}

\begin{proof}
(of Proposition \ref{prop:stage}) According to Lemma \ref{lem:sum},
it is enough to prove the proposition for $f=\pm1$. Without loss
of generality, we will prove the proposition for $r=1$, i.e. $h\in\sigma_{n}\sigma_{1}^{2}U_{1,m}$,
and symmetrically, the same is valid for every $1\leq r\leq n-1$.
The case $h\in\sigma_{n}^{2}U_{n,m}$ will be considered separately.

So let $h\in\sigma_{n}\sigma_{1}^{2}U_{1,m}$ and write: $h=\sigma_{1}u$
for some $u\in\sigma_{n}\sigma_{1}U_{1,m}$. We will prove the proposition
for $i\neq j\in\left\{ 1,2,3\right\} $ - as one can see below, we
will do it simultaneously for all the options for $i\neq j\in\left\{ 1,2,3\right\} $.
The treatment in the other cases in which $i\neq j\in\left\{ 1,k,l\right\} $
such that $1<k\neq l\leq n-1$ is obtained symmetrically, so we get
that the proposition is valid for every $1\leq i\neq j\leq n-1$. 

As before, we denote a block matrix of the form 
\[
\left(\begin{array}{cc}
B & 0\\
0 & I_{n-4}
\end{array}\right)\in GL_{n-1}(R_{n})
\]
by $B\in GL_{3}(R_{n})$. In the following computations, the indices
of the matrices are intended to help the reader recognize the corresponding
matrix type in Corollary \ref{cor:computations}, as will be explained
below. We remind that the inverse of a matrix is denoted by the same
index, and one can observe that the inverse of each indexed matrix
is obtained by changing the sign of $u$. We also remind that $u\in\sigma_{n}\sigma_{1}U_{1,m}\subseteq\sigma_{n}R_{n}$.
Thus, by Proposition \ref{prop:type 1.1} we have
\begin{eqnarray*}
\left(\begin{array}{ccc}
1-\sigma_{1}u & -\sigma_{1}^{2}u & 0\\
u & 1+\sigma_{1}u & 0\\
0 & 0 & 1
\end{array}\right)_{12} & = & \left(\begin{array}{ccc}
x_{2} & -\sigma_{1} & 0\\
0 & 1 & 0\\
0 & 0 & 1
\end{array}\right)\left(\begin{array}{ccc}
1 & 0 & 0\\
ux_{2} & 1 & 0\\
0 & 0 & 1
\end{array}\right)\\
 &  & \cdot\left(\begin{array}{ccc}
x_{2}^{-1} & x_{2}^{-1}\sigma_{1} & 0\\
0 & 1 & 0\\
0 & 0 & 1
\end{array}\right)\in IA_{m}
\end{eqnarray*}
\begin{eqnarray*}
\left(\begin{array}{ccc}
1-\sigma_{1}u & 0 & -\sigma_{1}^{2}u\\
0 & 1 & 0\\
u & 0 & 1+\sigma_{1}u
\end{array}\right)_{7} & = & \left(\begin{array}{ccc}
x_{3} & 0 & -\sigma_{1}\\
0 & 1 & 0\\
0 & 0 & 1
\end{array}\right)\left(\begin{array}{ccc}
1 & 0 & 0\\
0 & 1 & 0\\
ux_{3} & 0 & 1
\end{array}\right)\\
 &  & \cdot\left(\begin{array}{ccc}
x_{3}^{-1} & 0 & x_{3}^{-1}\sigma_{1}\\
0 & 1 & 0\\
0 & 0 & 1
\end{array}\right)\in IA^{m}
\end{eqnarray*}
\begin{eqnarray*}
\left(\begin{array}{ccc}
1 & 0 & 0\\
0 & 1+\sigma_{1}u & u\\
0 & -\sigma_{1}^{2}u & 1-\sigma_{1}u
\end{array}\right)_{3} & = & \left(\begin{array}{ccc}
1 & 0 & 0\\
u\sigma_{2} & 1 & 0\\
-u\sigma_{1}\sigma_{2} & 0 & 1
\end{array}\right)\left(\begin{array}{ccc}
1 & 0 & 0\\
0 & 1 & 0\\
\sigma_{2} & -\sigma_{1} & 1
\end{array}\right)\\
 &  & \cdot\left(\begin{array}{ccc}
1 & 0 & 0\\
0 & 1 & u\\
0 & 0 & 1
\end{array}\right)\left(\begin{array}{ccc}
1 & 0 & 0\\
0 & 1 & 0\\
-\sigma_{2} & \sigma_{1} & 1
\end{array}\right)\in IA^{m}
\end{eqnarray*}
\begin{eqnarray*}
\left(\begin{array}{ccc}
1 & 0 & 0\\
0 & 1-\sigma_{1}u & -\sigma_{1}^{2}u\\
0 & u & 1+\sigma_{1}u
\end{array}\right)_{6} & = & \left(\begin{array}{ccc}
1 & 0 & 0\\
-u\sigma_{1}\sigma_{3} & 1 & 0\\
u\sigma_{3} & 0 & 1
\end{array}\right)\left(\begin{array}{ccc}
1 & 0 & 0\\
\sigma_{3} & 1 & -\sigma_{1}\\
0 & 0 & 1
\end{array}\right)\\
 &  & \cdot\left(\begin{array}{ccc}
1 & 0 & 0\\
0 & 1 & 0\\
0 & u & 1
\end{array}\right)\left(\begin{array}{ccc}
1 & 0 & 0\\
-\sigma_{3} & 1 & \sigma_{1}\\
0 & 0 & 1
\end{array}\right)\in IA^{m}.
\end{eqnarray*}
By switching the signs of $\sigma_{1},\sigma_{2}$ and $\sigma_{3}$
in the two latter computations we obtain also that
\[
\left(\begin{array}{ccc}
1 & 0 & 0\\
0 & 1-\sigma_{1}u & u\\
0 & -\sigma_{1}^{2}u & 1+\sigma_{1}u
\end{array}\right)_{1},\left(\begin{array}{ccc}
1 & 0 & 0\\
0 & 1+\sigma_{1}u & -\sigma_{1}^{2}u\\
0 & u & 1-\sigma_{1}u
\end{array}\right)_{8}\in IA^{m}.
\]
Consider now the identities which we got in Corollary \ref{cor:computations},
and switch the roles of $f,g$ in the corollary by $u,$$\sigma_{1}$,
respectively. Remember that $u\in\sigma_{n}\sigma_{1}U_{1,m}$. Hence,
as by the computations above matrices of Forms $7$ and $8$ belong
to $IA^{m}$, we obtain from the last part of Equation (\ref{eq:cor-comp1})
that also matrices of Form $13$ belong to $IA^{m}$. Thus, as we
showed that Forms $1,3,6$ also belong to $IA^{m}$, Equation (\ref{eq:cor-comp1})
shows that also Forms $2,4,5$ belong to $IA^{m}$. Similar arguments
show that Equations (\ref{eq:cor-comp1})-(\ref{eq:cor-comp6}) give
that all the $18$ forms belong to $IA^{m}$. In particular, the matrices
which correspond to Forms $13-18$ belong to $IA^{m}$, and these
matrices (and their inverses) are precisely the matrices of the form
(we remind that $h=\sigma_{1}u$)
\[
\left(I_{n-1}\pm E_{j,i}\right)\left(I_{n-1}+hE_{i,j}\right)\left(I_{n-1}\mp E_{j,i}\right),\,\,\,\,i\neq j\in\left\{ 1,2,3\right\} .
\]
Clearly, by similar arguments, the proposition holds for every $1\leq i\neq j\leq n-1$
and every $h\in\sigma_{n}\sigma_{r}^{2}U_{r,m}$ for $1\leq r\leq n-1$. 

The case $h\in\sigma_{n}^{2}U_{n,m}$ is a bit different, but easer.
In this case one can consider the same computations we built for $r=1$,
with the following fittings: Firstly, write $h\in\sigma_{n}^{2}U_{n,m}$
as $h=\sigma_{n}u$ for some $u\in\sigma_{n}U_{n,m}$. Secondly, change
$\sigma_{1}$ to $\sigma_{n}$, change $\sigma_{2},\sigma_{3}$ to
$0$ and change $x_{2},x_{3}$ to $1$ in the right side of the above
equations. It is easy to see that in this situation we obtain in the
left side of the equations the same matrices, just that instead of
$\sigma_{1}$ we will have $\sigma_{n}$. From here we continue exactly
the same.
\end{proof}
\begin{prop}
The elements of the following form belong to $IA^{m}$:
\[
\left(I_{n-1}-fE_{j,i}\right)\left(I_{n-1}+hE_{i,j}\right)\left(I_{n-1}+fE_{j,i}\right)
\]
where $h\in\sigma_{n}^{2}U_{n,m},\sigma_{n}\sigma_{r}^{2}U_{r,m}$
for $1\leq r\leq n-1$, \uline{\mbox{$f\in\sigma_{s}R_{n}$}} for
$1\leq s\leq n$ and $i\neq j$.
\end{prop}

\begin{proof}
We will prove it for $s=1$, $i\neq j\in\left\{ 1,2,3\right\} $,
and denote a block matrix of the form 
\[
\left(\begin{array}{cc}
B & 0\\
0 & I_{n-4}
\end{array}\right)\in GL_{n-1}(R_{n})
\]
 by $B\in GL_{3}(R_{n})$. We will use again the result of Corollary
\ref{cor:computations}, when we switch the roles of $f,g$ in the
corollary by $h,\sigma_{1}u$ respectively for some $u\in R_{n}$. 

As $h\in\sigma_{n}\sigma_{r}^{2}U_{r,m},\sigma_{n}^{2}U_{n,m}$, we
have also $\sigma_{1}uh\in\sigma_{n}\sigma_{r}^{2}U_{r,m},\sigma_{n}^{2}U_{n,m}$.
Hence, we obtain from the previous proposition, that the matrices
of Forms $13-18$ belong to $IA^{m}$. In addition
\begin{eqnarray*}
\left(\begin{array}{ccc}
1 & 0 & 0\\
0 & 1-u\sigma_{1}h & h\\
0 & -u^{2}\sigma_{1}^{2}h & 1+u\sigma_{1}h
\end{array}\right)_{1} & = & \left(\begin{array}{ccc}
1 & 0 & 0\\
-hu\sigma_{2} & 1 & 0\\
-hu^{2}\sigma_{1}\sigma_{2} & 0 & 1
\end{array}\right)\left(\begin{array}{ccc}
1 & 0 & 0\\
0 & 1 & 0\\
-u\sigma_{2} & u\sigma_{1} & 1
\end{array}\right)\\
 &  & \cdot\left(\begin{array}{ccc}
1 & 0 & 0\\
0 & 1 & h\\
0 & 0 & 1
\end{array}\right)\left(\begin{array}{ccc}
1 & 0 & 0\\
0 & 1 & 0\\
u\sigma_{2} & -u\sigma_{1} & 1
\end{array}\right)\in IA^{m}
\end{eqnarray*}
\begin{eqnarray*}
\left(\begin{array}{ccc}
1 & 0 & 0\\
0 & 1-u\sigma_{1}h & -u^{2}\sigma_{1}^{2}h\\
0 & h & 1+u\sigma_{1}h
\end{array}\right)_{6} & = & \left(\begin{array}{ccc}
1 & 0 & 0\\
-hu^{2}\sigma_{1}\sigma_{3} & 1 & 0\\
hu\sigma_{3} & 0 & 1
\end{array}\right)\left(\begin{array}{ccc}
1 & 0 & 0\\
u\sigma_{3} & 1 & -u\sigma_{1}\\
0 & 0 & 1
\end{array}\right)\\
 &  & \cdot\left(\begin{array}{ccc}
1 & 0 & 0\\
0 & 1 & 0\\
0 & h & 1
\end{array}\right)\left(\begin{array}{ccc}
1 & 0 & 0\\
-u\sigma_{3} & 1 & u\sigma_{1}\\
0 & 0 & 1
\end{array}\right)\in IA^{m}
\end{eqnarray*}
and by switching the signs of $u$ and $h$ simultaneously, we get
also Forms $3$ and $8$. So we easily conclude from Corollary \ref{cor:computations}
(Equations (\ref{eq:cor-comp1}) and (\ref{eq:cor-comp3})) that also
the matrices of the other eight forms are in $IA^{m}$. In particular,
the matrices of the form
\[
\left(I_{n-1}-\sigma_{1}uE_{j,i}\right)\left(I_{n-1}+hE_{i,j}\right)\left(I_{n-1}+\sigma_{1}uE_{j,i}\right),\,\,\,\,i\neq j\in\left\{ 1,2,3\right\} 
\]
belong to $IA^{m}$. The treatment for every $i\neq j$ and $1\leq s\leq n-1$
is similar, and the treatment in the case $s=n$ is obtained by replacing
$\sigma_{1}$ by $\sigma_{n}$ and $\sigma_{2},\sigma_{3}$ by $0$
in the above equations.
\end{proof}
\begin{cor}
As every $f\in R_{n}$ can be decomposed as $f=\sum_{s=1}^{n}\sigma_{s}f_{s}+f_{0}$
for some $f_{0}\in\mathbb{Z}$ and $f_{i}\in R_{n}$, we obtain from
Lemma \ref{lem:sum} and from the above two propositions that we actually
finised the proof of Proposition \ref{prop:form2}.
\end{cor}

\subsection{Elements of Form $3$}
\begin{prop}
\label{prop:form3}Recall $\bar{O}_{m}=mR_{n-1}$ where $R_{n-1}=\mathbb{Z}[x_{1}^{\pm1},\ldots,x_{n-1}^{\pm1}]\subseteq R_{n}$.
Then, The elements of the following form, belong to $IA^{m}$:
\[
A^{-1}\left[\left(I_{n-1}+hE_{i,j}\right),\left(I_{n-1}+fE_{j,i}\right)\right]A
\]
where $A\in GL_{n-1}(R_{n})$, $f\in\sigma_{n}R_{n}$, $h\in\bar{O}_{m}^{2}$
and $i\neq j$.
\end{prop}

We will prove the proposition in the case $i,j=2,1$, and the same
arguments are valid for arbitrary $i\neq j$. In this case one can
write: $h=m^{2}h'$ for some $h'\in R_{n-1}$, and thus, our element
is of the form
\[
A^{-1}\left(\begin{array}{ccc}
1-fm^{2}h' & f & 0\\
-f\left(m^{2}h'\right)^{2} & 1+fm^{2}h' & 0\\
0 & 0 & I_{n-3}
\end{array}\right)\left(\begin{array}{ccc}
1 & -f & 0\\
0 & 1 & 0\\
0 & 0 & I_{n-3}
\end{array}\right)A
\]
for some $A\in GL_{n-1}(R_{n})$, $f\in\sigma_{n}R_{n}$ and $h'\in R_{n-1}$.
The proposition will follow easily from the following lemma:
\begin{lem}
\label{lem:sum-1}Let $h_{1},h_{2}\in R_{n}$, $f\in\sigma_{n}R_{n}$
and denote a block matrix of the form $\left(\begin{array}{cc}
B & 0\\
0 & I_{n-4}
\end{array}\right)\in GL_{n-1}(R_{n})$ by $B\in GL_{3}(R_{n})$. Then
\[
A^{-1}\left(\begin{array}{ccc}
1-fm\left(h_{1}+h_{2}\right) & f & 0\\
-f\left(m\left(h_{1}+h_{2}\right)\right)^{2} & 1+fm\left(h_{1}+h_{2}\right) & 0\\
0 & 0 & 1
\end{array}\right)\left(\begin{array}{ccc}
1 & -f & 0\\
0 & 1 & 0\\
0 & 0 & 1
\end{array}\right)A
\]
\begin{eqnarray*}
 & \equiv & A^{-1}\left(\begin{array}{ccc}
1-fmh_{1} & f & 0\\
-f\left(mh_{1}\right)^{2} & 1+fmh_{1} & 0\\
0 & 0 & 1
\end{array}\right)\left(\begin{array}{ccc}
1 & -f & 0\\
0 & 1 & 0\\
0 & 0 & 1
\end{array}\right)\\
 &  & \cdot\left(\begin{array}{ccc}
1-fmh_{2} & f & 0\\
-f\left(mh_{2}\right)^{2} & 1+fmh_{2} & 0\\
0 & 0 & 1
\end{array}\right)\left(\begin{array}{ccc}
1 & -f & 0\\
0 & 1 & 0\\
0 & 0 & 1
\end{array}\right)A\,\,\,\,\textrm{mod}\,\,\,\,IA^{m}
\end{eqnarray*}
\end{lem}

Now, if Lemma \ref{lem:sum-1} is proved, one can deduce that for
$f\in\sigma_{n}R_{n}$ and $h=m^{2}h'$, $h'\in R_{n}$, we have
\[
A^{-1}\left(\begin{array}{ccc}
1-fm^{2}h' & f & 0\\
-f\left(m^{2}h'\right)^{2} & 1+fm^{2}h' & 0\\
0 & 0 & I_{n-3}
\end{array}\right)\left(\begin{array}{ccc}
1 & -f & 0\\
0 & 1 & 0\\
0 & 0 & I_{n-3}
\end{array}\right)A
\]
\[
\equiv\left[A^{-1}\left(\begin{array}{ccc}
1-fmh' & f & 0\\
-f\left(mh'\right)^{2} & 1+fmh' & 0\\
0 & 0 & I_{n-3}
\end{array}\right)\left(\begin{array}{ccc}
1 & -f & 0\\
0 & 1 & 0\\
0 & 0 & I_{n-3}
\end{array}\right)A\right]^{m}\,\,\,\,\textrm{mod}\,\,\,\,IA^{m}
\]
and as the latter element is obviously belong to $IA^{m}$, Proposition
\ref{prop:form3} follows. So it is enough to prove Lemma \ref{lem:sum-1}.
\begin{proof}
(of Lemma \ref{lem:sum-1}) Throughout the computation we will use
the observation that as $GL_{n-1}(R_{n},\sigma_{n}R_{n})$ is normal
in $GL_{n-1}(R_{n})$, every conjugate of an element of $GL_{n-1}\left(R_{n},\sigma_{n}R_{n}\right)\leq IA\left(\Phi\right)$
by an element of $GL_{n-1}(R_{n})$, belongs to $GL_{n-1}\left(R_{n},\sigma_{n}R_{n}\right)\leq IA\left(\Phi\right)$
(as was mentioned in Remark \ref{rem:Normal}) - even though $GL_{n-1}(R_{n})\nleq IA\left(\Phi\right)$.
Throughout the computation, we will use the notations which we used
in the proof of Lemma \ref{lem:sum}:
\begin{itemize}
\item A matrix $\left(\begin{array}{cc}
B & 0\\
0 & I_{n-4}
\end{array}\right)\in GL_{n-1}(R_{n})$ is denoted by $B\in GL_{3}(R_{n})$. 
\item ``$=$'' denotes an equality between matrices in $GL_{n-1}(R_{n})$.
\item ``$\equiv$'' denotes an equality in $IA\left(\Phi\right)/IA^{m}$
.
\item We use square brackets to help the reader follow the steps of the
computation. Whenever square brackets are used, it is recommended
to concentrate in the expression inside them separately in order to
follow the transition to the next step.
\end{itemize}
So let's compute:

\[
A^{-1}\left(\begin{array}{ccc}
1-fm\left(h_{1}+h_{2}\right) & f & 0\\
-f\left(m\left(h_{1}+h_{2}\right)\right)^{2} & 1+fm\left(h_{1}+h_{2}\right) & 0\\
0 & 0 & 1
\end{array}\right)\left(\begin{array}{ccc}
1 & -f & 0\\
0 & 1 & 0\\
0 & 0 & 1
\end{array}\right)A
\]
\begin{eqnarray*}
 & = & A^{-1}\left(\begin{array}{ccc}
1 & 0 & -f\\
0 & 1 & -fm\left(h_{1}+h_{2}\right)\\
-m\left(h_{1}+h_{2}\right) & 1 & 1
\end{array}\right)\\
 &  & \cdot\left(\begin{array}{ccc}
1 & 0 & f\\
0 & 1 & fm\left(h_{1}+h_{2}\right)\\
m\left(h_{1}+h_{2}\right) & -1 & 1
\end{array}\right)\left(\begin{array}{ccc}
1 & -f & 0\\
0 & 1 & 0\\
0 & 0 & 1
\end{array}\right)A
\end{eqnarray*}
\begin{eqnarray*}
 & = & A^{-1}\left(\begin{array}{ccc}
1 & 0 & 0\\
0 & 1 & 0\\
-m\left(h_{1}+h_{2}\right) & 1 & 1
\end{array}\right)\left(\begin{array}{ccc}
1 & 0 & -f\\
0 & 1 & -fm\left(h_{1}+h_{2}\right)\\
0 & 0 & 1
\end{array}\right)\\
 &  & \cdot\left(\begin{array}{ccc}
1 & 0 & 0\\
0 & 1 & 0\\
m\left(h_{1}+h_{2}\right) & -1 & 1
\end{array}\right)\left(\begin{array}{ccc}
1 & 0 & f\\
0 & 1 & fm\left(h_{1}+h_{2}\right)\\
0 & 0 & 1
\end{array}\right)\left(\begin{array}{ccc}
1 & -f & 0\\
0 & 1 & 0\\
0 & 0 & 1
\end{array}\right)A
\end{eqnarray*}
\begin{eqnarray*}
 & = & A^{-1}\left(\begin{array}{ccc}
1 & 0 & 0\\
0 & 1 & 0\\
-mh_{2} & 0 & 1
\end{array}\right)\left[\left(\begin{array}{ccc}
1 & 0 & 0\\
0 & 1 & 0\\
-mh_{1} & 1 & 1
\end{array}\right)\left(\begin{array}{ccc}
1 & 0 & -f\\
0 & 1 & -fmh_{1}\\
0 & 0 & 1
\end{array}\right)\right.\\
 &  & \cdot\left.\left(\begin{array}{ccc}
1 & 0 & 0\\
0 & 1 & 0\\
mh_{1} & -1 & 1
\end{array}\right)\left(\begin{array}{ccc}
1 & 0 & f\\
0 & 1 & fmh_{1}\\
0 & 0 & 1
\end{array}\right)\right]\left(\begin{array}{ccc}
1 & 0 & -f\\
0 & 1 & -fmh_{1}\\
0 & 0 & 1
\end{array}\right)\left(\begin{array}{ccc}
1 & 0 & 0\\
0 & 1 & 0\\
mh_{2} & 0 & 1
\end{array}\right)A\\
 &  & \cdot\left[A^{-1}\left(\begin{array}{ccc}
1 & 0 & 0\\
0 & 1 & 0\\
-m\left(h_{1}+h_{2}\right) & 1 & 1
\end{array}\right)\left(\begin{array}{ccc}
1 & 0 & 0\\
0 & 1 & -fmh_{2}\\
0 & 0 & 1
\end{array}\right)\left(\begin{array}{ccc}
1 & 0 & 0\\
0 & 1 & 0\\
m\left(h_{1}+h_{2}\right) & -1 & 1
\end{array}\right)A\right]\\
 &  & \cdot\left[A^{-1}\left(\begin{array}{ccc}
1 & 0 & 0\\
0 & 1 & f\left(h_{1}+h_{2}\right)\\
0 & 0 & 1
\end{array}\right)A\right]^{m}A^{-1}\left(\begin{array}{ccc}
1 & -f & f\\
0 & 1 & 0\\
0 & 0 & 1
\end{array}\right)A
\end{eqnarray*}
\begin{eqnarray*}
 & \equiv & A^{-1}\left(\begin{array}{ccc}
1 & 0 & 0\\
0 & 1 & 0\\
-mh_{2} & 0 & 1
\end{array}\right)\left(\begin{array}{ccc}
1-fmh_{1} & f & 0\\
-f\left(mh_{1}\right)^{2} & 1+fmh_{1} & 0\\
0 & 0 & 1
\end{array}\right)\\
 &  & \cdot\left(\begin{array}{ccc}
1 & 0 & -f\\
0 & 1 & -fmh_{1}\\
0 & 0 & 1
\end{array}\right)\left(\begin{array}{ccc}
1 & 0 & 0\\
0 & 1 & 0\\
mh_{2} & 0 & 1
\end{array}\right)A\left[A^{-1}\left(\begin{array}{ccc}
1 & 0 & 0\\
0 & 1 & 0\\
-m\left(h_{1}+h_{2}\right) & 1 & 1
\end{array}\right)\right.\\
 &  & \cdot\left.\left(\begin{array}{ccc}
1 & 0 & 0\\
0 & 1 & -fh_{2}\\
0 & 0 & 1
\end{array}\right)\left(\begin{array}{ccc}
1 & 0 & 0\\
0 & 1 & 0\\
m\left(h_{1}+h_{2}\right) & -1 & 1
\end{array}\right)A\right]^{m}A^{-1}\left(\begin{array}{ccc}
1 & -f & f\\
0 & 1 & 0\\
0 & 0 & 1
\end{array}\right)A
\end{eqnarray*}
\begin{eqnarray*}
 & \equiv & A^{-1}\left[\left(\begin{array}{ccc}
1 & 0 & 0\\
0 & 1 & 0\\
-mh_{2} & 0 & 1
\end{array}\right)\left(\begin{array}{ccc}
1-fmh_{1} & f & 0\\
-f\left(mh_{1}\right)^{2} & 1+fmh_{1} & 0\\
0 & 0 & 1
\end{array}\right)\right.\\
 &  & \cdot\left.\left(\begin{array}{ccc}
1 & 0 & 0\\
0 & 1 & 0\\
mh_{2} & 0 & 1
\end{array}\right)\right]\left[\left(\begin{array}{ccc}
1 & 0 & 0\\
0 & 1 & 0\\
-mh_{2} & 0 & 1
\end{array}\right)\left(\begin{array}{ccc}
1 & 0 & -f\\
0 & 1 & 0\\
0 & 0 & 1
\end{array}\right)\right.\\
 &  & \cdot\left.\left(\begin{array}{ccc}
1 & 0 & 0\\
0 & 1 & 0\\
mh_{2} & 0 & 1
\end{array}\right)\right]\left[\left(\begin{array}{ccc}
1 & 0 & 0\\
0 & 1 & 0\\
-mh_{2} & 0 & 1
\end{array}\right)\left(\begin{array}{ccc}
1 & 0 & 0\\
0 & 1 & -fmh_{1}\\
0 & 0 & 1
\end{array}\right)\right.\\
 &  & \cdot\left.\left(\begin{array}{ccc}
1 & 0 & 0\\
0 & 1 & 0\\
mh_{2} & 0 & 1
\end{array}\right)\right]\left(\begin{array}{ccc}
1 & -f & f\\
0 & 1 & 0\\
0 & 0 & 1
\end{array}\right)A
\end{eqnarray*}
\begin{eqnarray*}
 & = & A^{-1}\left(\begin{array}{ccc}
1-fmh_{1} & f & 0\\
-f\left(mh_{1}\right)^{2} & 1+fmh_{1} & 0\\
0 & 0 & 1
\end{array}\right)A\cdot\left[A^{-1}\left(\begin{array}{ccc}
1 & 0 & 0\\
0 & 1 & 0\\
fmh_{1}h_{2} & -fh_{2} & 1
\end{array}\right)A\right]^{m}\\
 &  & \cdot A^{-1}\left(\begin{array}{ccc}
1 & 0 & 0\\
0 & 1 & 0\\
-mh_{2} & 0 & 1
\end{array}\right)\left(\begin{array}{ccc}
1 & 0 & -f\\
0 & 1 & 0\\
0 & 0 & 1
\end{array}\right)\left(\begin{array}{ccc}
1 & 0 & 0\\
0 & 1 & 0\\
mh_{2} & 0 & 1
\end{array}\right)A\\
 &  & \cdot\left[A^{-1}\left(\begin{array}{ccc}
1 & 0 & 0\\
0 & 1 & 0\\
-mh_{2} & 0 & 1
\end{array}\right)\left(\begin{array}{ccc}
1 & 0 & 0\\
0 & 1 & -fh_{1}\\
0 & 0 & 1
\end{array}\right)\left(\begin{array}{ccc}
1 & 0 & 0\\
0 & 1 & 0\\
mh_{2} & 0 & 1
\end{array}\right)A\right]^{m}\\
 &  & \cdot A^{-1}\left(\begin{array}{ccc}
1 & -f & f\\
0 & 1 & 0\\
0 & 0 & 1
\end{array}\right)A
\end{eqnarray*}
\begin{eqnarray*}
 & \equiv & A^{-1}\left(\begin{array}{ccc}
1-fmh_{1} & f & 0\\
-f\left(mh_{1}\right)^{2} & 1+fmh_{1} & 0\\
0 & 0 & 1
\end{array}\right)\left[\left(\begin{array}{ccc}
1 & 0 & 0\\
0 & 1 & 0\\
-mh_{2} & 0 & 1
\end{array}\right)\right.\\
 &  & \cdot\left.\left(\begin{array}{ccc}
1 & 0 & -f\\
0 & 1 & 0\\
0 & 0 & 1
\end{array}\right)\left(\begin{array}{ccc}
1 & 0 & 0\\
0 & 1 & 0\\
mh_{2} & 0 & 1
\end{array}\right)\right]\left(\begin{array}{ccc}
1 & -f & f\\
0 & 1 & 0\\
0 & 0 & 1
\end{array}\right)A
\end{eqnarray*}
\begin{eqnarray*}
 & = & A^{-1}\left(\begin{array}{ccc}
1-fmh_{1} & f & 0\\
-f\left(mh_{1}\right)^{2} & 1+fmh_{1} & 0\\
0 & 0 & 1
\end{array}\right)\left(\begin{array}{ccc}
1 & -f & 0\\
0 & 1 & 0\\
0 & 0 & 1
\end{array}\right)\\
 &  & \cdot\left[\left(\begin{array}{ccc}
1 & f & 0\\
0 & 1 & 0\\
0 & 0 & 1
\end{array}\right)\left(\begin{array}{ccc}
1-mfh_{2} & 0 & -f\\
0 & 1 & 0\\
f(mh_{2})^{2} & 0 & 1+mfh_{2}
\end{array}\right)\left(\begin{array}{ccc}
1 & -f & 0\\
0 & 1 & 0\\
0 & 0 & 1
\end{array}\right)\right]\\
 &  & \cdot\left(\begin{array}{ccc}
1 & 0 & f\\
0 & 1 & 0\\
0 & 0 & 1
\end{array}\right)A
\end{eqnarray*}
\begin{eqnarray*}
 & = & A^{-1}\left(\begin{array}{ccc}
1-fmh_{1} & f & 0\\
-f\left(mh_{1}\right)^{2} & 1+fmh_{1} & 0\\
0 & 0 & 1
\end{array}\right)\left(\begin{array}{ccc}
1 & -f & 0\\
0 & 1 & 0\\
0 & 0 & 1
\end{array}\right)A\\
 &  & \cdot\left[A^{-1}\left(\begin{array}{ccc}
1 & f^{2}h_{2} & 0\\
0 & 1 & 0\\
0 & -m(fh_{2})^{2} & 1
\end{array}\right)A\right]^{m}\\
 &  & \cdot A^{-1}\left(\begin{array}{ccc}
1-mfh_{2} & 0 & -f\\
0 & 1 & 0\\
f(mh_{2})^{2} & 0 & 1+mfh_{2}
\end{array}\right)\left(\begin{array}{ccc}
1 & 0 & f\\
0 & 1 & 0\\
0 & 0 & 1
\end{array}\right)A\\
\end{eqnarray*}
\begin{eqnarray*}
 & \equiv & A^{-1}\left(\begin{array}{ccc}
1-fmh_{1} & f & 0\\
-f\left(mh_{1}\right)^{2} & 1+fmh_{1} & 0\\
0 & 0 & 1
\end{array}\right)\left(\begin{array}{ccc}
1 & -f & 0\\
0 & 1 & 0\\
0 & 0 & 1
\end{array}\right)\\
 &  & \cdot\left(\begin{array}{ccc}
1-fmh_{2} & 0 & -f\\
0 & 1 & 0\\
f\left(mh_{2}\right)^{2} & 0 & 1+fmh_{2}
\end{array}\right)\left(\begin{array}{ccc}
1 & 0 & f\\
0 & 1 & 0\\
0 & 0 & 1
\end{array}\right)A.
\end{eqnarray*}
So it remains to show that
\begin{equation}
A^{-1}\left(\begin{array}{ccc}
1-fmh_{2} & 0 & -f\\
0 & 1 & 0\\
f\left(mh_{2}\right)^{2} & 0 & 1+fmh_{2}
\end{array}\right)\left(\begin{array}{ccc}
1 & 0 & f\\
0 & 1 & 0\\
0 & 0 & 1
\end{array}\right)A\label{eq:cong1}
\end{equation}
\[
\equiv A^{-1}\left(\begin{array}{ccc}
1-fmh_{2} & f & 0\\
-f\left(mh_{2}\right)^{2} & 1+fmh_{2} & 0\\
0 & 0 & 1
\end{array}\right)\left(\begin{array}{ccc}
1 & -f & 0\\
0 & 1 & 0\\
0 & 0 & 1
\end{array}\right)A.
\]
By a similar computation as for Equation (\ref{eq:comp4}), we have
(switch the roles of $f,g$ in the equation by $f,mh_{2}$ respectively,
and then switch the roles of the first row and column with the third
row and column)
\[
\left(\begin{array}{ccc}
1 & 0 & 0\\
0 & 1+fmh_{2} & fmh_{2}\\
0 & -fmh_{2} & 1-fmh_{2}
\end{array}\right)=\left(\begin{array}{ccc}
1 & -f & -f\\
0 & 1 & 0\\
0 & 0 & 1
\end{array}\right)
\]
\begin{eqnarray*}
 &  & \cdot\left(\begin{array}{ccc}
1+fmh_{2} & 0 & f\\
0 & 1 & 0\\
-f\left(mh_{2}\right)^{2} & 0 & 1-fmh_{2}
\end{array}\right)\left(\begin{array}{ccc}
1 & 0 & 0\\
f\left(mh_{2}\right)^{2} & 1 & fmh_{2}\\
0 & 0 & 1
\end{array}\right)\\
 &  & \cdot\left(\begin{array}{ccc}
1-fmh_{2} & f & 0\\
-f\left(mh_{2}\right)^{2} & 1+fmh_{2} & 0\\
0 & 0 & 1
\end{array}\right)\left(\begin{array}{ccc}
1 & 0 & 0\\
0 & 1 & 0\\
f\left(mh_{2}\right)^{2} & -fmh_{2} & 1
\end{array}\right).
\end{eqnarray*}

Therefore, using Proposition \ref{prop:form1}, and the observation
\[
A^{-1}\left(\begin{array}{ccc}
1 & 0 & 0\\
0 & 1+fmh_{2} & fmh_{2}\\
0 & -fmh_{2} & 1-fmh_{2}
\end{array}\right)A=\left[A^{-1}\left(\begin{array}{ccc}
1 & 0 & 0\\
0 & 1+fh_{2} & fh_{2}\\
0 & -fh_{2} & 1-fh_{2}
\end{array}\right)A\right]^{m}\in IA^{m}.
\]
we obtain that mod $IA^{m}$ we have
\begin{eqnarray}
 &  & A^{-1}\left(\begin{array}{ccc}
1+fmh_{2} & 0 & f\\
0 & 1 & 0\\
-f\left(mh_{2}\right)^{2} & 0 & 1-fmh_{2}
\end{array}\right)\label{eq:cong2}\\
 &  & \cdot\left(\begin{array}{ccc}
1-fmh_{2} & f & 0\\
-f\left(mh_{2}\right)^{2} & 1+fmh_{2} & 0\\
0 & 0 & 1
\end{array}\right)A\equiv A^{-1}\left(\begin{array}{ccc}
1 & f & f\\
0 & 1 & 0\\
0 & 0 & 1
\end{array}\right)A.\nonumber 
\end{eqnarray}
From here, we easily get Equation (\ref{eq:cong1}) by noticing that
the inverse of every matrix in Equation (\ref{eq:cong2}) is obtained
by replacing $f$ by $-f$. This finishes the proof of the lemma,
and hence, also the proof of Proposition \ref{prop:form3}.
\end{proof}

\subsection{Elements of Form $4$}
\begin{prop}
\label{prop:form4}Recall $\bar{O}_{m}=mR_{n-1}$ and $\bar{U}_{r,m}=(x_{r}^{m}-1)R_{n-1}$
for $1\leq r\leq n-1$, where $R_{n-1}=\mathbb{Z}[x_{1}^{\pm1},\ldots,x_{n-1}^{\pm1}]\subseteq R_{n}$.
Then, the elements of the following form, belong to $IA^{m}$:
\[
A^{-1}\left[(I_{n-1}+hE_{i,j}),(I_{n-1}+fE_{j,i})\right]A
\]
where $A\in GL_{n-1}(R_{n})$, $f\in\sigma_{n}R_{n}$, $h\in\sigma_{r}^{2}\bar{U}_{r,m},\sigma_{r}\bar{O}_{m}$
for $1\leq r\leq n-1$ and $i\neq j$.
\end{prop}

As before, throughout the subsection we denote a block matrix of the
form 
\[
\left(\begin{array}{cc}
B & 0\\
0 & I_{n-4}
\end{array}\right)\in GL_{n-1}(R_{n})
\]
by $B\in GL_{3}(R_{n})$. We start the proof of this proposition with
the following lemma.
\begin{lem}
Let $f,h\in R_{n}$ and $A\in GL_{n-1}(R_{n})$. Then
\[
A^{-1}\left(\begin{array}{ccc}
1-fh & -fh & 0\\
fh & 1+fh & 0\\
0 & 0 & 1
\end{array}\right)A
\]
\begin{eqnarray*}
 & = & A^{-1}\left(\begin{array}{ccc}
1 & 0 & 0\\
fh & 1 & 0\\
fh^{2} & 0 & 1
\end{array}\right)AA^{-1}\left(\begin{array}{ccc}
1 & 0 & 0\\
0 & 1+fh & -f\\
0 & fh^{2} & 1-fh
\end{array}\right)\left(\begin{array}{ccc}
1 & 0 & 0\\
0 & 1 & f\\
0 & 0 & 1
\end{array}\right)A\\
 &  & \cdot A^{-1}\left(\begin{array}{ccc}
1 & 0 & 0\\
0 & 1 & -f\\
0 & 0 & 1
\end{array}\right)AA^{-1}\left(\begin{array}{ccc}
1 & -fh & 0\\
0 & 1 & 0\\
0 & -fh^{2} & 1
\end{array}\right)AA^{-1}\left(\begin{array}{ccc}
1 & 0 & 0\\
0 & 1 & f\\
0 & 0 & 1
\end{array}\right)A\\
 &  & \cdot A^{-1}\left(\begin{array}{ccc}
1-fh & 0 & f\\
0 & 1 & 0\\
-fh^{2} & 0 & 1+fh
\end{array}\right)AA^{-1}\left(\begin{array}{ccc}
1 & 0 & 0\\
f^{2}h^{2} & 1 & -f^{2}h\\
0 & 0 & 1
\end{array}\right)AA^{-1}\left(\begin{array}{ccc}
1 & 0 & -f\\
0 & 1 & 0\\
0 & 0 & 1
\end{array}\right)A.
\end{eqnarray*}
\end{lem}

\begin{proof}
The lemma follows from Proposition \ref{prop:computations}, Equation
(\ref{eq:comp1}), by substituting $g$ with $h$, combined with verifying
the identity 
\begin{align*}
 & \left(\begin{array}{ccc}
1 & 0 & 0\\
0 & 1 & -f\\
0 & 0 & 1
\end{array}\right)\left(\begin{array}{ccc}
1-fh & 0 & f\\
0 & 1 & 0\\
-fh^{2} & 0 & 1+fh
\end{array}\right)\left(\begin{array}{ccc}
1 & 0 & 0\\
0 & 1 & f\\
0 & 0 & 1
\end{array}\right)\\
 & =\left(\begin{array}{ccc}
1-fh & 0 & f\\
0 & 1 & 0\\
-fh^{2} & 0 & 1+fh
\end{array}\right)\left(\begin{array}{ccc}
1 & 0 & 0\\
f^{2}h^{2} & 1 & -f^{2}h\\
0 & 0 & 1
\end{array}\right).
\end{align*}
\end{proof}
Observe now that if we have $f\in\sigma_{n}R_{n}$ and $h\in\sigma_{r}^{2}\bar{U}_{r,m},\sigma_{r}\bar{O}_{m}$
for $1\leq r\leq n-1$, then by Propositions \ref{prop:form1} and
\ref{prop:form2}, we have
\[
A^{-1}\left(\begin{array}{ccc}
1-fh & -fh & 0\\
fh & 1+fh & 0\\
0 & 0 & 1
\end{array}\right)A
\]
\[
=A^{-1}\left(\begin{array}{ccc}
1 & 0 & 0\\
-1 & 1 & 0\\
0 & 0 & 1
\end{array}\right)\left(\begin{array}{ccc}
1 & -fh & 0\\
0 & 1 & 0\\
0 & 0 & 1
\end{array}\right)\left(\begin{array}{ccc}
1 & 0 & 0\\
1 & 1 & 0\\
0 & 0 & 1
\end{array}\right)A\in IA^{m}.
\]
Therfore, by Propositions \ref{prop:form1}, \ref{prop:form2} and
the previous lemma, for every $A\in GL_{n-1}(R_{n})$ we have the
following equality mod $IA^{m}$:
\[
A^{-1}\left(\begin{array}{ccc}
1 & 0 & 0\\
0 & 1+fh & -f\\
0 & fh^{2} & 1-fh
\end{array}\right)\left(\begin{array}{ccc}
1 & 0 & 0\\
0 & 1 & f\\
0 & 0 & 1
\end{array}\right)A
\]
\[
\equiv A^{-1}\left[\left(\begin{array}{ccc}
1-fh & 0 & f\\
0 & 1 & 0\\
-fh^{2} & 0 & 1+fh
\end{array}\right)\left(\begin{array}{ccc}
1 & 0 & -f\\
0 & 1 & 0\\
0 & 0 & 1
\end{array}\right)\right]^{-1}A.
\]
I.e. we have the following corollary (notice that we switched the
sign of $f$):
\begin{cor}
\label{cor:equivalence}For every $h\in\sigma_{r}^{2}\bar{U}_{r,m},\sigma_{r}\bar{O}_{m}$
for $1\leq r\leq n-1$, $f\in\sigma_{n}R_{n}$ and $A\in GL_{n-1}(R_{n}),$
the following elements are congruent mod $IA^{m}$
\[
A^{-1}\left[(I_{n-1}+hE_{3,2}),(I_{n-1}+fE_{2,3})\right]A\equiv A^{-1}\left[(I_{n-1}-fE_{1,3}),(I_{n-1}+hE_{3,1})\right]A.
\]
\end{cor}

We proceed with the following proposition:
\begin{prop}
\label{prop:induction- base}Let $h\in\sigma_{1}^{2}\bar{U}_{1,m},\sigma_{1}\bar{O}_{m}$
and $f\in\sigma_{n}R_{n}$. Then
\[
\left[(I_{n-1}+hE_{3,2}),(I_{n-1}+fE_{2,3})\right]\in IA^{m}.
\]
\end{prop}

\begin{proof}
Denote $h=\sigma_{1}u$ for some $u\in\sigma_{1}\bar{U}_{1,m},\bar{O}_{m}$.
By Proposition \ref{prop:type 1.1}, we have 
\[
\left(\begin{array}{ccc}
1 & 0 & 0\\
0 & 1 & 0\\
-\sigma_{2}u & \sigma_{1}u & 1
\end{array}\right)\in IA^{m}
\]
and hence
\[
IA^{m}\ni\left(\begin{array}{ccc}
1 & 0 & 0\\
0 & 1 & 0\\
-\sigma_{2}u & \sigma_{1}u & 1
\end{array}\right)\left(\begin{array}{ccc}
1 & 0 & 0\\
0 & 1 & f\\
0 & 0 & 1
\end{array}\right)\left(\begin{array}{ccc}
1 & 0 & 0\\
0 & 1 & 0\\
\sigma_{2}u & -\sigma_{1}u & 1
\end{array}\right)\left(\begin{array}{ccc}
1 & 0 & 0\\
0 & 1 & -f\\
0 & 0 & 1
\end{array}\right)
\]
\begin{eqnarray*}
 & = & \left(\begin{array}{ccc}
1 & 0 & 0\\
f\sigma_{2}u & 1 & 0\\
\sigma_{1}\sigma_{2}u^{2}f & 0 & 1
\end{array}\right)\\
 &  & \cdot\left(\begin{array}{ccc}
1 & 0 & 0\\
0 & 1 & 0\\
0 & \sigma_{1}u & 1
\end{array}\right)\left(\begin{array}{ccc}
1 & 0 & 0\\
0 & 1 & f\\
0 & 0 & 1
\end{array}\right)\left(\begin{array}{ccc}
1 & 0 & 0\\
0 & 1 & 0\\
0 & -\sigma_{1}u & 1
\end{array}\right)\left(\begin{array}{ccc}
1 & 0 & 0\\
0 & 1 & -f\\
0 & 0 & 1
\end{array}\right)
\end{eqnarray*}
As by Proposition \ref{prop:type 1.1} the first matrix in the right
hand side is also in $IA^{m}$, we obtain that
\[
\left[(I_{n-1}+hE_{3,2}),(I_{n-1}+fE_{2,3})\right]=\left[\left(\begin{array}{ccc}
1 & 0 & 0\\
0 & 1 & 0\\
0 & h & 1
\end{array}\right),\left(\begin{array}{ccc}
1 & 0 & 0\\
0 & 1 & f\\
0 & 0 & 1
\end{array}\right)\right]\in IA^{m}
\]
as required.
\end{proof}
We can now pass to the following proposition.
\begin{prop}
Let $h\in\sigma_{1}^{2}\bar{U}_{1,m},\sigma_{1}\bar{O}_{m}$, $f\in\sigma_{n}R_{n}$
and $A\in GL_{n-1}(R_{n})$. Then
\[
A^{-1}\left[(I_{n-1}+hE_{3,2}),(I_{n-1}+fE_{2,3})\right]A\in IA^{m}.
\]
\end{prop}

\begin{proof}
We will prove the proposition by induction. By a result of Suslin
\cite{key-33}, as $n-1\geq3$, the group $SL_{n-1}(R_{n})$ is generated
by the elementary matrices of the form
\[
I_{n-1}+rE_{l,k}\,\,\,\,\textrm{for}\,\,\,\,r\in R_{n},\,\,\,\,\textrm{and}\,\,\,\,1\leq l\neq k\leq n-1.
\]
So as the invertible elements of $R_{n}$ are the elements of the
form $\pm\prod_{i=1}^{n}x_{i}^{s_{i}}$ for $s_{i}\in\mathbb{Z}$
(see \cite{key-38}, chapter 8), $GL_{n-1}(R_{n})$ is generated by
the elementary matrices and the matrices of the form: $I_{n-1}+(\pm x_{i}-1)E_{1,1}$
for $1\leq i\leq n$. Therefore, by the previous proposition it is
enough to show that if
\[
A^{-1}\left[\left(I_{n-1}+hE_{3,2}\right),\left(I_{n-1}+fE_{2,3}\right)\right]A\in IA^{m}
\]
and $E$ is one of the above generators, then mod $IA^{m}$ we have
\begin{equation}
A^{-1}E^{-1}\left[\left(I_{n-1}+hE_{3,2}\right),\left(I_{n-1}+fE_{2,3}\right)\right]EA\label{eq:property}
\end{equation}
\[
\equiv A^{-1}\left[\left(I_{n-1}+hE_{3,2}\right),\left(I_{n-1}+fE_{2,3}\right)\right]A.
\]

So if $E$ is of the form $I_{n-1}+(\pm x_{i}-1)E_{1,1}$, we obviously
have Property (\ref{eq:property}). If $E$ is an elementary matrix
of the form $I_{n-1}+rE_{l,k}$ such that $l,k\notin\left\{ 2,3\right\} $
then we also have Property (\ref{eq:property}) in an obvious way.
Consider now the case $l,k=2,3$. In this case, by Corollary \ref{cor:equivalence}
we have the following mod $IA^{m}$
\[
A^{-1}E^{-1}\left[\left(\begin{array}{ccc}
1 & 0 & 0\\
0 & 1 & 0\\
0 & h & 1
\end{array}\right),\left(\begin{array}{ccc}
1 & 0 & 0\\
0 & 1 & f\\
0 & 0 & 1
\end{array}\right)\right]EA
\]
\[
\equiv A^{-1}\left(\begin{array}{ccc}
1 & 0 & 0\\
0 & 1 & -r\\
0 & 0 & 1
\end{array}\right)\left[\left(\begin{array}{ccc}
1 & 0 & -f\\
0 & 1 & 0\\
0 & 0 & 1
\end{array}\right),\left(\begin{array}{ccc}
1 & 0 & 0\\
0 & 1 & 0\\
h & 0 & 1
\end{array}\right)\right]\left(\begin{array}{ccc}
1 & 0 & 0\\
0 & 1 & r\\
0 & 0 & 1
\end{array}\right)A
\]
\[
=A^{-1}\left(\begin{array}{ccc}
1 & 0 & 0\\
0 & 1 & -r\\
0 & 0 & 1
\end{array}\right)\left(\begin{array}{ccc}
1-hf+h^{2}f^{2} & 0 & -hf^{2}\\
0 & 1 & 0\\
-h^{2}f & 0 & 1+hf
\end{array}\right)\left(\begin{array}{ccc}
1 & 0 & 0\\
0 & 1 & r\\
0 & 0 & 1
\end{array}\right)A
\]
\[
=A^{-1}\left(\begin{array}{ccc}
1-hf+h^{2}f^{2} & 0 & -hf^{2}\\
0 & 1 & 0\\
-h^{2}f & 0 & 1+hf
\end{array}\right)AA^{-1}\left(\begin{array}{ccc}
1 & 0 & 0\\
rh^{2}f & 1 & -rhf\\
0 & 0 & 1
\end{array}\right)A
\]
\[
=A^{-1}\left[\left(\begin{array}{ccc}
1 & 0 & -f\\
0 & 1 & 0\\
0 & 0 & 1
\end{array}\right),\left(\begin{array}{ccc}
1 & 0 & 0\\
0 & 1 & 0\\
h & 0 & 1
\end{array}\right)\right]AA^{-1}\left(\begin{array}{ccc}
1 & 0 & 0\\
rh^{2}f & 1 & -rhf\\
0 & 0 & 1
\end{array}\right)A.
\]
So by applying Propositions \ref{prop:form1}, \ref{prop:form2} and
Corollary \ref{cor:equivalence} once again on the opposite way, we
obtain Property (\ref{eq:property}). The other cases for $l,k$ are
treated by similar arguments: if $l,k=3,2$ we do exactly the same,
and if $l$ or $k$ are different from $2$ and $3$, then the situation
is easier - we use similar arguments, but without passing to $\left[(I_{n-1}-fE_{1,3}),(I_{n-1}+hE_{3,1})\right]$
through Corollary \ref{cor:equivalence}.
\end{proof}
\begin{cor}
Let $h\in\sigma_{1}^{2}\bar{U}_{1,m},\sigma_{1}\bar{O}_{m}$, $f\in\sigma_{n}R$
and $A\in GL_{n-1}\left(R_{n}\right)$. Then, for every $i\neq j$
we have
\[
A^{-1}\left[(I_{n-1}+hE_{i,j}),(I_{n-1}+fE_{j,i})\right]A\in IA^{m}.
\]
\end{cor}

\begin{proof}
Denote a permutation matrix, which its action on $GL_{n-1}\left(R_{n}\right)$
by conjugation, moves $2\mapsto j$ and $3\mapsto i$, by $P$. Then,
by the previous proposition, we have
\[
A^{-1}\left[(I_{n-1}+hE_{i,j}),(I_{n-1}+fE_{j,i})\right]A
\]
\[
=A^{-1}P^{-1}\left[(I_{n-1}+hE_{3,2}),(I_{n-1}+fE_{2,3})\right]PA\in IA^{m}.
\]
\end{proof}
Now, as one can see that symmetrically, the above corollary is valid
for every $h\in\sigma_{r}^{2}\bar{U}_{r,m},\sigma_{r}\bar{O}_{m}$
for $1\leq r\leq n-1$, we actually finished the proof of Proposition
\ref{prop:form4}.

\section{Index of notations}

For the convenient of the reader, we gathered here some notations
that play role along the paper, and mention the section in the body
of the paper when they appear for the first time. 
\begin{itemize}
\item $F_{n}$ = the free group on $n$ elements, Section \ref{sec:some-properties}.
\item $\Phi=\Phi_{n}=F_{n}/F''_{n}$= the free metabelian group on $n$
elements, Section \ref{sec:some-properties}.
\item $\Psi_{m}=\Phi/M_{m}$, where $M_{m}=\left(\Phi'\Phi^{m}\right)'\left(\Phi'\Phi^{m}\right)^{m}$,
Section \ref{sec:some-properties}.
\item $IA(\Phi)=\ker\left(Aut\left(\Phi\right)\to Aut\left(\Phi/\Phi'\right)\right)$,
Section \ref{sec:some-properties}.
\item $IG_{m}=G(M_{m})=\ker(IA\left(\Phi\right)\to Aut(\Psi_{m}))$, Section
\ref{sec:some-properties}.
\item $IA^{m}=\left\langle IA\left(\Phi\right)^{m}\right\rangle $, Section
\ref{sec:Second}.
\item $IA_{m}=\cap\left\{ N\vartriangleleft IA\left(\Phi\right)\,|\,[IA\left(\Phi\right):N]\,|\,m\right\} $,
Section \ref{sec:some-properties}.
\item $R_{n}=\mathbb{Z}[x_{1}^{\pm1},\ldots,x_{n}^{\pm1}]$ where $x_{1},\ldots,x_{n}$
are free commutative variables, Section \ref{sec:some-properties}.
\item $R_{n-1}=\mathbb{Z}[x_{1}^{\pm1},\ldots,x_{n-1}^{\pm1}]$, Section
\ref{sec:Second}.
\item $\mathbb{Z}_{m}=\mathbb{Z}/m\mathbb{Z}$, Section \ref{sec:some-properties}.
\item $\sigma_{i}=x_{i}-1$ for $1\leq i\leq n$, Section \ref{sec:some-properties}.
\item $\vec{\sigma}$ = the column vector which has $\sigma_{i}$ in its
$i$-th entry, Section \ref{sec:some-properties}.
\item $\mathfrak{A}=\sum_{i=1}^{n}\sigma_{i}R_{n}\vartriangleleft R_{n}$
= the augmentation ideal of $R_{n}$, Section \ref{sec:some-properties}.
\item $O_{m}=mR_{n}\vartriangleleft R_{n}$, Section \ref{sec:Second}.
\item $\bar{O}_{m}=mR_{n-1}\vartriangleleft R_{n-1}$, Section \ref{sec:Second}.
\item $U_{r,m}=(x_{r}^{m}-1)R_{n}\vartriangleleft R_{n}$, for $1\leq r\leq n$,Section
\ref{sec:Second}.
\item $\bar{U}_{r,m}=(x_{r}^{m}-1)R_{n-1}\vartriangleleft R_{n-1}$, for
$1\leq r\leq n$, Section \ref{sec:Second}.
\item $H_{m}=\sum_{i=1}^{n}(x_{i}^{m}-1)R_{n}+mR_{n}\vartriangleleft R_{n}$,
Section \ref{sec:some-properties}.
\item $S=\mathbb{Z}[x^{\pm1}]$, Section \ref{sec: not f.g.}.
\item $J_{m}=\left(x^{m}-1\right)S+mS\vartriangleleft S$, Section \ref{sec: not f.g.}.
\item $E_{d}\left(R\right)=\left\langle I_{d}+rE_{i,j}\,|\,r\in R,\,1\leq i\neq j\leq d\right\rangle \leq SL_{d}\left(R\right)$,
where $R$ is a ring and $E_{i,j}$ is the matrix that has $1$ in
its $\left(i,j\right)$-th entry and $0$ elsewhere, Section \ref{sec:K-theory}.
\item $SL_{d}\left(R,H\right)=\ker(SL_{d}\left(R\right)\to SL_{d}\left(R/H\right))$,
where $R$ is a ring and $H\vartriangleleft R$, Section \ref{sec:K-theory}.
\item $GL_{d}\left(R,H\right)=\ker(GL_{d}\left(R\right)\to GL_{d}\left(R/H\right))$,
where $R$ is a ring and $H\vartriangleleft R$, Section \ref{sec:K-theory}.
\item $E_{d}\left(R,H\right)$ = the normal subgroup of $E_{d}\left(R\right)$,
generated as a normal subgroup by the matrices of the form $I_{d}+hE_{i,j}$
for $h\in H$, Section \ref{sec:K-theory}.
\item $IGL_{n-1,i}=\left\{ I_{n}+A\in IA\left(\Phi\right)\,|\,\begin{array}{c}
\textrm{The\,\,}i\textrm{-th\,\, row\,\, of\,\,}A\textrm{\,\, is\,\,0,}\\
I_{n-1}+A_{i,i}\in GL_{n-1}(R_{n},\sigma_{i}R_{n})
\end{array}\right\} $, \\
\\
for $1\leq i\leq n$, Section \ref{sec:some-properties}.
\item $ISL_{n-1,i}\left(H\right)=IGL_{n-1,i}\cap SL_{n-1}(R_{n},H)$, under
the identification of $IGL_{n-1,i}$ with $GL_{n-1}(R_{n},\sigma_{i}R_{n})$,
Section \ref{sec:some-properties}.
\item $IE_{n-1,i}\left(H\right)=IGL_{n-1,i}\cap E{}_{n-1}(R_{n},H)$, under
the identification of the group $IGL_{n-1,i}$ with $GL_{n-1}(R_{n},\sigma_{i}R_{n})$,
Section \ref{sec:some-properties}.
\item $IGL'_{n-1,i}=\left\{ I_{n}+A\in IA\left(\Phi\right)\,|\,\textrm{The\,\,}i\textrm{-th\,\, row\,\, of\,\,}A\textrm{\,\, is\,\,0}\right\} $,
\\
\\
for $1\leq i\leq n$, Section \ref{sec:some-properties}.
\end{itemize}

Department of Mathematics\\
University of California in San-Diego\\
San-Diego, California, 92093\\
\\
davidel-chai.ben-ezra@mail.huji.ac.il\\

\end{document}